\documentclass[11pt]{article}

\usepackage{latexsym}
\usepackage{amsfonts}
\usepackage{amsmath}
\usepackage{amssymb}
\usepackage{amsthm}
\usepackage{enumerate}
\usepackage{graphics,epic}
\usepackage{graphicx}
\usepackage{pgf}
\usepackage{xcolor}
\usepackage{blkarray}
\usepackage{mathrsfs}
\usepackage{tikz}
\usepackage{float}
\usepackage{verbatim }
\usepackage{multirow}

\newcommand{\argmax}{\operatornamewithlimits{argmax}}
\newcommand{\argmin}{\operatornamewithlimits{argmin}}

\DeclareMathOperator{\cl}{cl}

\theoremstyle{plain}
\newtheorem{theorem}{Theorem}[section]
\newtheorem{corollary}[theorem]{Corollary}
\newtheorem{lemma}[theorem]{Lemma}
\newtheorem{proposition}[theorem]{Proposition}

\theoremstyle{definition}
\newtheorem{definition}[theorem]{Definition}

\numberwithin{equation}{section}

\newcommand\floor[1]{\lfloor#1\rfloor}

\newcommand{\om}{\omega}

\begin{document}

\title{Minimization of a Class of Rare Event Probabilities and Buffer Probabilities of Exceedance}

\author{Amarjit Budhiraja, Shu Lu, Yang Yu and Quoc Tran-Dinh}

\maketitle
\begin{abstract}
We consider the problem of choosing design parameters to minimize the probability of an undesired rare event that is described through the average
of $n$ iid random variables. Since the probability of interest for near optimal design parameters is very small, one needs to develop suitable accelerated Monte-Carlo methods for estimating the objective function of interest. One of the challenges in the study is that simulating from exponential twists of the laws of the summands may be computationally demanding since these transformed laws may be non-standard and intractable. We consider a setting where the summands are given as a nonlinear functional of random variables that are more tractable for importance sampling in that the exponential twists of their distributions take a 
simpler form (than that for the original summands). We use techniques from Dupuis and Wang (2004,2007) to identify the appropriate Issacs equations whose subsolutions are suitable for constructing tractable importance sampling schemes. We also study the closely related problem of estimating buffered probability of exceedance and provide the first rigorous results that relate the asymptotics of buffered probability and that of the ordinary probability under a large deviation scaling.
The analogous minimization problem for buffered probability, under conditions,  can be formulated as a convex optimization problem which makes it more tractable than the original optimization problem.  Once again importance sampling methods are needed in order to estimate the objective function since the events of interest have very small (buffered) probability. We show that, under conditions, changes of measures that are asymptotically efficient (under the large deviation scaling)
for estimating ordinary probability are also asymptotically efficient for estimating the buffered probability of exceedance. We embed the constructed importance sampling scheme in suitable gradient descent/ascent algorithms for solving the optimization problems of interest. Implementation of schemes for some examples
is illustrated through computational experiments.

\noindent



\end{abstract}



\section{Introduction}
\label{s:intro}


Optimization problems of the form $\min_{\theta \in \Theta} \mathbb{E}\left[ F(Y,\theta) \right]$
where $\Theta \subset \mathbb{R}^d$, $Y$ is a random vector in $\mathbb{R}^m$ with distribution $\mu$, and $F:\mathbb{R}^m \times \Theta \rightarrow \mathbb{R}$ is measurable have been studied extensively. In many applications the expectation cannot be computed explicitly, and it is common to estimate it by a sample average, such as $\frac{1}{N} \sum_{j=1}^{N} F(Y^j,\theta)$ where $Y^j: j=1, \cdots, N$ are i.i.d samples of $Y$. However, if the standard deviation of $F(Y,\theta)$ is large relative to its mean, then the sample size $N$ in the simulation needs to be very large for the sample average to reliably approximate the expected value. In such situations, it is desirable to use variance reduction techniques such as importance sampling to reduce the sample size needed, by 
replacing $\{F(Y^j,\theta)\}$ with 
samples of a random vector with the same mean but a smaller variance. The basic idea of importance sampling is to consider another probability measure $\nu$ such that $\mu$ is absolutely continuous with respect to $\nu$, with $\frac{d\mu}{d\nu}$ being the Radon-Nikodym derivative. Since $$\int F(y,\theta)\frac{d\mu}{d\nu}(y) \nu(dy) = \int F(y,\theta)\mu(dy) = \mathbb{E}\left[F(Y,\theta)\right],$$
just as $\frac{1}{N} \sum_{j=1}^{N} F(Y^j,\theta)$, $\frac{1}{N}\sum_{j=1}^{N}F(\bar{Y}^j,\theta)\frac{d\mu}{d\nu}(\bar{Y}^j)$ is
an unbiased estimator of $\mathbb{E}\left[F(Y,\theta)\right]$, where $\{ \bar{Y}^j \}_{j=1,\dots,N}$ are i.i.d samples from the distribution $\nu$. 
Extensive research has been conducted to identify an alternative measure $\nu$ from which one can simulate easily and is such that the  variance of $F(\bar{Y}^1,\theta)\frac{d\mu}{d\nu}(\bar{Y}^1)$ is lower than that of $F(Y,\theta)$, see \cite{ chen1993importance, dupuis2004importance,dupuis2007subsolutions, glasserman1997counterexamples, evans2000approximating, owen2000safe} and references therein.

In this paper we focus on a situation in which the random variable $Y$ has the form of the average of i.i.d random variables
$U_i, i=1,\cdots, n$ in $\mathbb{R}^m$. For each $i=1,\cdots, n$ the random variable $U_i$ is given as
$
U_i=G(X_i,\theta)
$
where $X_i, i=1,\cdots, n$ are i.i.d random variables in $\mathbb{R}^h$ and $G: \mathbb{R}^h \times \Theta \to \mathbb{R}^m$ is a continuous function.
Using $n$ as the subscript we write $Y_n= \frac{1}{n} \sum_{i=1}^{n}U_i$. We are interested in the following minimization problem
\begin{equation}\label{q:mainproblem}
\min_{\theta \in \Theta} \mathbb{E}\exp\{-nF(Y_n)\}= \min_{\theta \in \Theta}\mathbb{E}\exp\left\{-nF\bigg(\frac{1}{n} \sum_{i=1}^{n}G(X_i,\theta)\bigg)\right\}
\end{equation}
where $F:\mathbb{R}^m \rightarrow \mathbb{R}\cup \{\infty\}$ is a  measurable function. Here we use $Y_n$ as a shorthand for the complete notation $Y_n(\omega,\theta)$  for simplicity. Although not studied here, one can also consider in an analogous setting where $F$ is a function of $(y, \theta)$, namely a function on $\mathbb{R}^m \times \Theta$.


The formulation \eqref{q:mainproblem} includes a special case in which $F(y)= \infty 1_{A^c}(y)$, where $1_{A^c}$ is the indicator function of a measurable set $A^c\subset \mathbb{R}^m$, that takes the value of $\infty$ when $y\in A^c$ and $0$ otherwise (by convention $\infty \cdot 0 =0$). In this case \eqref{q:mainproblem} becomes
\begin{align}
\min_{\theta\in \Theta} \mathbb{P}\left(Y_n \in A\right) = \min_{\theta\in \Theta} \mathbb{P}\left(  \frac{1}{n}\sum_{i=1}^{n}G(X_i,\theta) \in A \right).
\label{q:min_probability}
\end{align}

In many applications in engineering, finance, and insurance, decisions need to be made to reduce the probability for an undesirable event (such as system breakdown) to occur. Such an event is often the result of the accumulative effects of a large number of individual events over a long period, which we model  as $\{Y_n\in A\}$, with $n$ being a fixed large number. Under conditions, for values of $\theta$ such that $E[U_1]\not\in \cl A$, $\mathbb{P}\left(Y_n \in A\right)$ converges to 0 exponentially fast as $n\to\infty$ by the theory of large deviations, so its value is extremely small for large $n$, making it very difficult to estimate using i.i.d samples of $Y_n$.

An effective way to estimate the probabilities of such rare events and expected values of more general risk sensitive functionals as on the right side of \eqref{q:mainproblem}
is using importance sampling techniques based on large deviations theory.
Large deviation based importance sampling techniques were introduced in Siegmund \cite{siegmund1976importance} in estimating the error probabilities of the sequential probability ratio test. Subsequent papers exhibited  the good performance of specific estimators developed using this technique, see \cite{bucklew1990large, collamore2002importance, sadowsky1991large}. However, such estimators can perform poorly as shown in Glasserman and Wang \cite{glasserman1997counterexamples}, if the necessary and sufficient conditions for effective
variance reduction in \cite{chen1993importance,sadowsky1990large,sadowsky1996monte} are violated. In order to address this, later papers introduced adaptive importance sampling schemes that are more generally applicable. Among these, the papers of Dupuis and Wang \cite{dupuis2004importance, dupuis2007subsolutions} are most related to our work. The paper \cite{dupuis2004importance} connects the problem of constructing asymptotically efficient adaptive (feedback)
importance sampling schemes with certain deterministic dynamic games. The second paper \cite{dupuis2007subsolutions} uses  subsolutions to the Isaacs equations associated with such games to construct  flexible and simple dynamic importance sampling schemes that achieve asymptotic efficiency. \par


For a direct application of the  importance sampling techniques from \cite{dupuis2004importance, dupuis2007subsolutions} to the situation here, one would need to use a parametric family of exponential changes of measure to generate the replacements for the $U_i$ given each fixed $\theta$. 
Such a scheme is easy to implement when the distribution of $U_i$ is of a simple form. For example if $U_i$ is a normal random variable then an exponential change of measure is also a normal distribution with a shifted mean. However, for more general distributions and when the dimension $m$ is large, sampling from the exponential tilt distribution can be computationally expensive (see discussion at the end of Section \ref{ss:change_on_U}). This problem gets much more severe in the optimization problem we study, in which estimates for the objective function need to be computed for many different values of $\theta$. By writing $U_i=G(X_i,\theta)$, we aim to capture the complexity of the distribution of $U_i$ through the function $G$ and leave the distribution of $X_i$ in a fixed simple form. In particular, we are interested in a setting where simulating from exponential tilts of distributions of $X_i$ is simpler than that from exponential tilts 
of $U_i$. In this work we develop an importance sampling technique based on a change of measure on the distribution of $X_i$, which 
is computationally much less demanding
compared to a scheme that uses a change of measure directly based on $U_i$. The scheme is inspired by \cite{dupuis2004importance, dupuis2007subsolutions} and,
as in these papers, is guided by the Issacs equation of a certain dynamic game. The Issacs equation is given in terms of a different Hamiltonian (see \eqref{eq:eq347}) than the one that arises in the formulation where the change of measure is done directly on the sequence $\{U_i\}$ (see \eqref{eq:eq828}). We show that generalized subsolutions of this Issacs equation can be used to construct importance sampling algorithms, with guaranteed lower bounds on asymptotic performance (as measured by the asymptotic exponential decay rate of the second moment), that are based on dynamic change of measure for the sequence $\{X_i\}$.
Similar to \cite{dupuis2004importance, dupuis2007subsolutions}, the decay rate is governed by the initial value of the subsolution (i.e. at $(t,x)=(0,0)$), with larger initial values implying a higher decay rate.

Next, we embed this importance sampling procedure in a gradient descent method to find the optimal $\theta$ for \eqref{q:mainproblem}.
Solution properties of \eqref{q:mainproblem} can be studied by investigating its limiting behavior as $n\to \infty$. It can be shown (see Theorem
\ref{t:limiting}) that under certain regularity conditions $$-\frac{1}{n} \log \exp\left\{-nF\bigg(\frac{1}{n} \sum_{i=1}^{n}G(X_i,\theta)\bigg)\right\}$$
 converges to a limiting function.
The optimal solution and optimal value of the limiting problem, when available, can be used as approximations of those of the original problem in which $n$ is a fixed large number.

Solving \eqref{q:min_probability} using a gradient descent method would require an estimate of the gradient of its objective function at each iteration. Although the objective function is continuously differentiable with respect to $\theta$ under some conditions, its gradient cannot be estimated
by the derivative of its sample average approximation function, because the sample average approximation is a piecewise constant function. Thus, instead of working with \eqref{q:min_probability} directly we will
use a surrogate reliability measure obtained by
 an approximation of  $\textbf{1}_A(\cdot)$ by a differentiable function, and apply the importance sampling methods to the expected values of the resulting risk sensitive functional.

The problem \eqref{q:min_probability} or its smooth approximation will not be convex in general, so the gradient descent algorithm will not distinguish local solutions from global solutions. For the case $m=1$ and $A=[c,\infty)$, there is an alternative reliability measure called the buffered failure probability \cite{rockafellar2010buffered} or the buffered probability of exceedance \cite{mafusalov2015estimation}.  Under mild conditions, minimization of the buffered probability over a class of probability distributions can be transformed into a convex optimization problem and is therefore more tractable.
The buffered probability is always greater than or equal to the corresponding probability, and the two values are often close to each other when the probability of the random variable of interest taking on large values is small (see e.g. \cite{rockafellar2010buffered} for a discussion of this point).
In this work we make the second statement precise in one particular setting. Specifically, we show that under conditions, probabilities of the form on the right side of \eqref{q:min_probability}
have the same exponential decay rate, as $n\to \infty$, as the corresponding buffered failure probabilities (see Theorem \ref{buff-ldp}). 
To the best of our knowledge this is the first rigorous result that relates the asymptotics of a buffered failure probability and ordinary probability under a large deviation scaling.
This result in particular suggests that the importance sampling change of measure that are appropriate for estimating the probability on the right side of \eqref{q:min_probability} should also be suitable for constructing estimators for the corresponding buffered failure probability.
Under appropriate conditions, this is indeed the case as is shown in Theorem \ref{prop:asympoptbp} and Theorem \ref{thm:main2}. One can view the buffered failure probability as a reliability measure that is of independent interest or, in view of its closeness to the ordinary exceedance probability, the solution to the buffered probability minimization problem can be used as an intermediate step
for selecting the initial point in the algorithm for the probability minimization problem.

 Comprehensive overviews on probability optimization and optimization under probabilistic (chance) constraints can be found in \cite{prekopa2013stochastic} and \cite[Chapter 4]{shapiro2009lectures}. In addition to  its direct practical applications, probability optimization is also commonly used to find initial feasible solutions for chance-constrained optimization \cite{prekopa2013stochastic}. Various methods for solving chance-constrained optimization have been proposed, including regularization methods based on approximations of level sets of the probability function \cite{dentcheva2013regularization}, the scenario approach replacing the chance constraints by finitely many sampling of the constraints \cite{calafiore2006scenario},  the sample average approximation (SAA) formulation by mixed integer programming \cite{pagnoncelli2009sample}, and convex analytical approximations of chance constraints \cite{nemirovski2006convex}. In the case of Gaussian or alternative distributions, one can also
 compute values and gradients of the probability function directly using methods such as spheric-radial decomposition \cite{van2014gradient,bremer2015probabilistic}. When the chance constraints involve the probability of a rare event, importance sampling techniques can be combined with the SAA approach to reduce the required sample size, by exploiting the structure of the problem under study to reduce the sample estimation variance uniformly with respect to the decision variables \cite{barrera2016chance}.

\par

The paper is organized as follows. Section 2 reviews importance sampling techniques that are based on  large deviation analyses and proposes a new importance sampling scheme that is based on changes of laws of the sequence $\{X_i\}$ rather than directly transforming the probability laws of the sequence $\{U_i\}$.
This section also provides an asymptotic bound on the second moment of the importance sampling estimator. Section 3 studies the limiting behavior of the problem \eqref{q:mainproblem} as $n\to \infty$, as well as convergence properties of the approximation problem for \eqref{q:min_probability} in which probabilities are replaced by expected values of certain risk sensitive functionals.
Section 4 studies the buffered probability in the present setting and its estimation using importance sampling methods. Section 5 presents the optimization algorithm and uses several numerical examples to illustrate the method. Throughout the paper, $\mathscr{P}(\mathbb{R}^h)$ denotes the space of all probability measures on $\mathbb{R}^h$.

\section{Importance sampling based on large deviations analysis}
\label{s:is}

In this section, we discuss how to estimate the objective value of \eqref{q:mainproblem} for a fixed value of $\theta$ by using importance sampling. Since $\theta$ is fixed, we suppress it in this section to reduce notational burden and consider the estimation of
\begin{equation}\label{q:mainproblem_fixedtheta}
\mathbb{E}\exp\{-nF(Y_n)\},
\end{equation}
where  $Y_n = \frac{1}{n}\sum_{i=1}^{n}U_i$ is the average of iid random variables $U_i= G(X_i)$ for $i=1,\cdots, n$. The function $G: \mathbb{R}^h  \to \mathbb{R}^m$ is continuous, and $F:\mathbb{R}^m \rightarrow \mathbb{R}\cup \{\infty\}$ is measurable. Let $\eta$ be the distribution of $X_1$ and $\xi$ be the distribution of $U_1$.

If the distribution of the random variable $Y_n$ takes a simple form, then one may  consider a change of measure with respect to the distribution of $Y_n$ directly. However, by its definition, the distribution of $Y_n$ is in general rather complicated and so one needs to construct the change of measure through the underlying distributions of $U_i$. 
Even in situations where the distribution of $Y_n$ is of simple form, e.g. Gaussian, it may be advantageous to construct a change of measure
that exploits the form of $Y_n$ and transforms the distributions of summands $U_i$ in a systematic manner.
Section 3.1 below reviews the estimation methods from \cite{dupuis2004importance,dupuis2007subsolutions} that construct a dynamic change of measure on the distributions of $\{U_i\}$ and provide results characterizing the asymptotic  performance of the resulting estimator.
One of the challenges in implementing these methods is that
 even if the distribution $\eta$ of $X_i$ were of a simple form, for a general $G$
 the distribution of $U_i$ may be rather complicated, so sampling from the exponential twists of the distribution of $U_i$ may become hard.
 In Section 3.2 we provide an alternative approach that constructs an estimator using a dynamic  change of measure with respect to the distributions of $X_i$, and establish an asymptotic bound on the second moment for the resulting importance sampling estimator.

In either approach, the replacement random variables will in general not be iid, and the conditional distribution of the $j$th random variable given the previous $j-1$ variables is related to the original distribution by an exponential tilt, i.e., the Radon-Nikodym derivative of the replacement measure with respect to the original measure is an exponential function with a linear exponent (see e.g. \eqref{eq:eq813}).
Parameters for these exponents are chosen based on solutions of certain partial differentiable equations. These equations arise when one considers the problem of minimizing the second moment as a certain stochastic control problem and studies the associated dynamic programming equations. The asymptotic performance of the resulting change of measure is established using methods from the theory of large deviations.

The starting point of the analysis are the logarithms of moment generating functions of the original random variables.  For $(a,\alpha)\in \mathbb{R}^{h+m}$, we define
\begin{equation}\label{q:def_H}
 H(a,\alpha) = \log\mathbb{E}\left[ e^{\langle a, X_1 \rangle + \langle \alpha, G(X_1) \rangle} \right].
\end{equation}
We also consider functions $H_1: \mathbb{R}^h \to \mathbb{R}$ and $H_2: \mathbb{R}^m \to \mathbb{R}$ as
\begin{equation}\label{q:def_H1}
H_1(a) = H(a,0), \ a \in \mathbb{R}^h
\end{equation}
and
\begin{equation}\label{q:def_H2}
H_2(\alpha) = H(0,\alpha), \ \alpha \in \mathbb{R}^m.
\end{equation}
Thus, $H_1$ is the  log-moment generating function of $X_1$ and $H_2$ is that of $U_1=G(X_1)$.

\subsection{The exponential change of measure on variables $U_i$}\label{ss:change_on_U}

In this subsection we review results from \cite{dupuis2004importance,dupuis2007subsolutions}.  Assume $H_2(\alpha) < \infty$ for all $\alpha \in \mathbb{R}^m$. We will replace the original random variables $U_1, \cdots, U_n$ by new random variables $\bar{U}^n_1, \cdots, \bar{U}^n_n$, that have (conditional) distributions of the form
\begin{equation}\label{eq:eq813}
e^{\langle \alpha, u \rangle - H_{2}(\alpha)}\xi(du)
\end{equation}
where $\alpha \in \mathbb{R}^m$ and $\xi$ is the distribution of $U_1$. In general, the parameter $\alpha$ that defines the sampling distribution does not need to be a constant, and can depend on values of summands that precede the current variable. Formally, suppose a function $\bar{\alpha}(x,t): \mathbb{R}^m \times [0,1] \to \mathbb{R}^m$ is given. To construct
a  dynamic  change of measure based on $\bar \alpha$ one proceeds as follows.  Suppose $\bar{U}^n_1,\cdots, \bar{U}^n_j$ have been simulated. Define
\begin{equation}\label{eq:eq817}
\bar{Y}^n_{j}= \frac{1}{n}\sum_{i=1}^j \bar{U}^n_i
\end{equation}
and simulate $\bar{U}^n_{j+1}$ from the distribution
\begin{equation}\label{q:tilted_U}
e^{\langle \bar{\alpha}(\bar{Y}^n_j,j/n), u \rangle - H_{2}(\bar{\alpha}(\bar{Y}^n_j,j/n))}\xi(du).
 \end{equation}
Thus the conditional distribution of $\bar{U}^n_{j+1}$ given $\{\bar{Y}^n_{i}, i=1, \ldots j\}$ is given by
\eqref{q:tilted_U}. Through this recursive procedure we obtain $\{\bar U^n_j\}_{1\le j \le n}$ and
$\{\bar Y^n_j\}_{1\le j \le n}$. It can be checked using a successive conditioning argument that
\begin{equation}\label{q:Zn}
Z^n= e^{-nF(\bar{Y}^n_n)}\prod_{j=0}^{n-1}e^{-\langle \bar{\alpha}(\bar{Y}^n_j,j/n),\bar{U}^n_{j+1} \rangle + H_2(\bar{\alpha}(\bar{Y}^n_j,j/n))}
\end{equation}
is an unbiased estimator for \eqref{q:mainproblem_fixedtheta}, and the above product of  exponentials is the Radon-Nikodym derivative of the distribution of $(U_1,\cdots, U_n)$ with respect to that of $(\bar{U}^n_1, \cdots, \bar{U}^n_n)$.

If the function $\bar{\alpha}$ is a constant, then the above scheme reduces to a static change of measure in which $(\bar{U}^n_1, \cdots, \bar{U}^n_n)$ are iid. Different choices of the function $\bar{\alpha}$ will produce different distributions for $Z^n$. In order to reduce the number of samples needed to the greatest extent, the idea is to choose $\bar{\alpha}$ in a way to minimize the variance (or equivalently the second moment ) of $Z^n$. It is hard to characterize the optimal choice of $\bar{\alpha}$ for a fixed value of $n$, as the distribution of $Y_n$ is rather complicated. However, as $n\to \infty$ the  (tails of the) distribution of $Y_n$ can be described using large deviations theory, which leads to a characterization of an asymptotically optimal choice of $\bar\alpha$ in terms of the solution of a partial differential equation known as the Isaacs equation\cite{dupuis2004importance}. We now introduce this equation. Let $L_2$ be the Legendre transform of $H_2$ defined as
\begin{equation}\label{q:def_L2}
L_2(\beta) = \sup_{\alpha \in \mathbb{R}^m} \left( \langle \alpha,\beta \rangle -H_2(\alpha) \right), \ \beta \in \mathbb{R}^m.
\end{equation}
It is possible that $L_2(\beta)=\infty$ for some $\beta$.
Define $\mathbb{H}_2: \mathbb{R}^{3m} \to \mathbb{R}\cup \{\infty\}$ as
\begin{equation}\label{eq:eq828}
	\mathbb{H}_2(s;\alpha,\beta) = \langle s,\beta \rangle + L_2(\beta) + \langle \alpha,\beta \rangle - H_2(\alpha).
\end{equation}
The Isaacs equation  is then given as
\begin{equation}
  W_t(y,t) + \sup_{\alpha\in\mathbb{R}^m}\inf_{\beta\in \mathbb{R}^m} \mathbb{H}_2(DW(y,t);\alpha,\beta)=0
  \label{q:isaccs}
\end{equation}
where $W: \mathbb{R}^m \times [0,1]\to \mathbb{R}$ is a continuously differentiable function, $W_t(y,t)$ is its derivative w.r.t. $t$, and $DW(y,t)$ is its derivative w.r.t. $y$. If $W$ satisfies
\begin{equation}
  W_t(y,t) + \sup_{\alpha\in\mathbb{R}^m}\inf_{\beta\in \mathbb{R}^m} \mathbb{H}_2(DW(y,t);\alpha,\beta) \ge 0
  \label{q:isaccs_sub}
\end{equation}
instead of \eqref{q:isaccs} then it is a (classical) subsolution to \eqref{q:isaccs}. If such a subsolution $W$ also satisfies the terminal condition $W(y,1)\le 2F(y)$ for all $y\in \mathbb{R}^m$, then, as is shown in \cite{dupuis2004importance,dupuis2007subsolutions}, the dynamic  change of measure as in \eqref{q:tilted_U}, constructed using the  supremizer $\alpha(y,t)$ for
the second term in \eqref{q:isaccs_sub},  produces an estimator $Z^n$ as in \eqref{q:Zn} (with $\bar \alpha$ replaced by $\alpha$) whose second moment decays exponentially at rate $W(0,0)$:
\begin{equation}\label{q:decayrate}
\liminf_{n\rightarrow\infty}-\frac{1}{n}\log\mathbb{E}[(Z^n)^2]\geq W(0,0).
\end{equation}
On the other hand, under standard conditions, the limit
\begin{equation}\label{q:gamma}
\gamma = \lim_{n\rightarrow\infty}-\frac{1}{n}\log\mathbb{E}\exp\{-nF(Y_n)\}
\end{equation}
exists \cite{dupuis2011weak}. By Jensen's inequality, if $\tilde Z^n$ is any unbiased estimator of \eqref{q:mainproblem_fixedtheta}
\[
\limsup_{n\rightarrow\infty}-\frac{1}{n}\log\mathbb{E}[(\tilde Z^n)^2] \le \limsup_{n\rightarrow\infty}-\frac{1}{n}\log (\mathbb{E}[\tilde Z^n])^2 = 2\gamma,
\]
so $2\gamma$ is the largest decay rate for the second moment among all unbiased estimators $\tilde Z^n$ of \eqref{q:mainproblem_fixedtheta}. In certain situations, one can find  a subsolution $W$ with $W(0,0) = 2\gamma$, in which case it follows from \eqref{q:decayrate} that
the importance sampling estimator  $Z^n$ in \eqref{q:Zn} constructed from the supermizer $\alpha$ in \eqref{q:isaccs} is {\em asymptotically efficient}.

In many examples one needs more than one subsolution in order to construct an importance sampling estimator that achieves asymptotic efficiency. This leads to the following notion of a generalized subsolution/control\cite{dupuis2007subsolutions}.
Let $K\ge 1$ and for $1\le k\le K$,  $\bar{\alpha}_k: \mathbb{R}^m \times [0,1] \to \mathbb{R}^m$. One of these $K$ functions is   randomly selected at each step to determine the change of measure for the summand at the given step
and the likelihood of a particular selection is determined by a probability vector valued function $\{\rho_k\}_{k=1}^K$, where $\rho_k: \mathbb{R}^m \times [0,1] \rightarrow [0,1]$. The collection $(\bar \alpha_k,
 \rho_k)$ is referred to as a generalized control pair. A precise definition is as follows.
\begin{definition}\label{def_gen_sub}
	Given $K \in \mathbb{N}$, consider functions $\bar{W}:\mathbb{R}^m\times [0,1] \rightarrow \mathbb{R}$, $\rho_k: \mathbb{R}^m \times [0,1] \rightarrow \mathbb{R}$, $\bar{\alpha}_k: \mathbb{R}^m \times [0,1] \rightarrow \mathbb{R}^m$, $1 \leq k \leq K$. The collection  $(\bar{W},\rho_k,\bar{\alpha}_k)$ is called a generalized subsolution/control to the Isaacs equation (\ref{q:isaccs}), and
	$(\bar \alpha_k,
	 \rho_k)$ the corresponding generalized control pair, if the following conditions hold: (i) For all $(y,t)$, $\{\rho_k(y,t)\}$ is a probability vector, i.e.,
	$$\rho_k(y,t) \ge 0, 1\le k \le K, \text{ and }
\sum_{k=1}^{K}\rho_k(y,t)=1 \text{ for all }(y,t)\in \mathbb{R}^m\times [0,1].$$
(ii) $\bar W$ is continuously differentiable and	$\bar{W}_t$ and $D\bar{W}$ have representations
	$$\bar{W}_t(y,t)=\sum_{k=1}^{K}\rho_k(y,t)r_k(y,t), \ D\bar{W}(y,t)=\sum_{k=1}^{K}\rho_k(y,t)s_k(y,t).$$
(iii) For each $k=1,\dots,K$,
	\begin{equation}\label{q:def_gen_sub}
	r_k(y,t) + \inf_{\beta\in \mathbb{R}^m} \mathbb{H}_2 (s_k(y,t);\bar{\alpha}_k(y,t),\beta) \geq 0.
	\end{equation}
(iv) The functions $(r_k,s_k,\rho_k,\bar{\alpha}_k)$ are uniformly bounded and continuous.
\end{definition}

With a generalized subsolution/control $(\bar{W},\rho_k,\bar{\alpha}_k)$ in hand, one can construct a dynamic change of measure as follows. Let $\bar{Y}^n_0=0$. For $j=0,\dots,n-1$, having constructed $\{\bar U^n_i\}_{1\le i\le j}$ and $\{\bar Y^n_i\}_{1\le i\le j}$,
we  generate a multinomial random variable $I$ such that $\mathbb{P}[I=k] = \rho_k(\bar{Y}^n_j,j/n)$ for $k\in\{1,2,\dots,K\}$.

Next, we simulate $\bar{U}^n_{j+1}$ from the distribution
\begin{equation}\label{q:exp_change}
e^{\langle \bar{\alpha}_I(\bar{Y}^n_j,j/n), u \rangle - H_{2}(\bar{\alpha}_I(\bar{Y}^n_j,j/n))}\xi(du),
\end{equation}
namely the conditional distribution of $\bar{U}^n_{j+1}$ given $\{\bar U^n_i\}_{i\le j}$ and $I$ is given by \eqref{q:exp_change}.
 Define
$\bar{Y}^n_{j+1}= \bar{Y}^n_j + \frac{1}{n}\bar{U}^n_{j+1}$.
It follows from a simple calculation (see \cite{dupuis2007subsolutions})
that
\begin{equation}\label{eq:eq525}
Z^n= e^{-nF(\bar{Y}^n_n)}\prod_{j=0}^{n-1}\left[\sum_{k=1}^{K}\rho_k(\bar{Y}^n_j,j/n)e^{\langle \bar{\alpha}_k(\bar{Y}^n_j,j/n),\bar{U}^n_{j+1} \rangle - H_2(\bar{\alpha}_k(\bar{Y}^n_j,j/n))}\right]^{-1}
\end{equation}
is an unbiased estimator for \eqref{q:mainproblem_fixedtheta} with the $n$-fold product above defining the Radon-Nikodym derivative of the distribution of $(U_1,\cdots, U_n)$ with respect to that of $(\bar{U}^n_1, \cdots, \bar{U}^n_n)$ (evaluated at $(\bar{U}^n_1, \cdots, \bar{U}^n_n)$).
Once again, when the terminal condition $\bar{W}(y,1)\leq2F(y)$ holds for all $y \in \mathbb{R}^m$, the second moment of $Z^n$ decays exponentially at a rate no slower than $\bar{W}(0,0)$, namely \eqref{q:decayrate} is satisfied with $W$ replaced by $\bar W$.
Thus if one can find a $\bar W$ as above with $\bar W(0,0) =2\gamma$, one has an asymptotically efficient importance sampling estimator.
In general one seeks a $\bar W$ which has the largest possible value at $(0,0)$.

\par

When $\xi$ is a simple form distribution (such as a Normal, Gamma, Poisson, exponential or a binomial), the tilted distribution \eqref{eq:eq813} typically belongs to the same distribution family with a different parameter. In such cases, samples  from \eqref{q:tilted_U} can be generated easily. However, in  general  the distribution of $U_i=G(X_i)$ may not take a simple form. To simulate from \eqref{q:tilted_U} in such a general situation, one needs to  invert the conditional cumulative  distributions and then evaluate them at uniform random variables. However, with a general nonlinear function $G$, the distribution $\xi$ is rarely available in a tractable form, making such a procedure difficult to start with. Even when $\xi$ is available in a closed form,  inverting the  conditional cumulative distributions  requires iteratively carrying out numerical integrations, which is highly computationally intensive. For these reasons, the practical utility of changing measures on $U_i$ is limited to situations in which $\xi$ takes a simple form.

\subsection{The exponential change of measure on variables $X_i$}\label{ss:change_on_X}

The computational issue of simulating from the tilted distribution \eqref{q:exp_change} is largely due to the complexity of $\xi$, the distribution of $U_i=G(X_i)$. This motivates us to consider the alternative approach of conducting the change of measure on variable $X_i$, whose distribution $\eta$ is assumed to be of a simpler form.

In this subsection, we assume that $H(a,\alpha)< \infty$ for all $(a,\alpha)\in \mathbb{R}^{h+m}$, and let $L$ be the Legendre transformation of $H$:
\begin{equation}\label{q:def_L}
L(b, \beta)=\sup_{(a,\alpha)\in \mathbb{R}^{h+m}}\big(\langle a, b \rangle + \langle \alpha, \beta \rangle-H(a,\alpha)\big), \ (b,\beta) \in \mathbb{R}^{h+m}.
\end{equation}
Then $L$ has the following representation  \cite[Lemma 6.2.3]{dupuis2011weak}:
\begin{align}
L(b, \beta) = \inf_{\mu \in \mathscr{P}(\mathbb{R}^h)}\left\{ R(\mu\|\eta): \int_{\mathbb{R}^h}x\mu(dx)=b, \int_{\mathbb{R}^h}G(x)\mu(dx)=\beta \right\},
\label{L representation}
\end{align}
where $R(\mu\|\eta)$ is the relative entropy of the probability measure $\mu$ with respect to $\eta$, defined as
\begin{equation}\label{q:def_relativeentropy}
R(\mu\|\eta) = \int_{\mathbb{R}^h} \log \frac{d\mu}{d\eta} d\mu
\end{equation}
when $\mu$ is absolutely continuous wrt $\eta$ and $\infty$ otherwise.

Recall that $H_1$ is the log-moment generating function of $X_1$. In the change of measure scheme, we will replace random variables $X_1,\cdots, X_n$ by new variables $\bar{X}^n_1,\cdots, \bar{X}^n_n$ that have (conditional) distributions $\eta_a$ of the form
\begin{equation}\label{q:eta_a}
\eta_a(dx) = e^{\langle a,x \rangle - H_1(a)} \eta(dx),
\end{equation}
where $a\in \mathbb{R}^h$ and $\eta$ as before is the distribution of $X_1$. The values of $a$ will be determined dynamically by a function $\bar{a}: \mathbb{R}^m \times [0,1] \to \mathbb{R}^h$ as follows. Let $\bar{Y}^n_0=0$. For $j=0,\cdots, n-1$,
having constructed $\{\bar X^n_i\}_{1\le i\le j}$, $\{\bar U^n_i = G(\bar X^n_i)\}_{1\le i\le j}$ and
$\{\bar Y^n_i\}_{1\le i\le j}$ via \eqref{eq:eq817},
let $\eta_{\bar{a}(\bar{Y}^n_j, j/n)}$ be the distribution of $\bar{X}^n_{j+1}$ conditioned on $\bar{X}^n_1,\dots,\bar{X}^n_j$ and draw a sample $\bar{X}^n_{j+1}$ from this conditional distribution.
Let
\[
\bar{Y}^n_{j+1}=\bar{Y}^n_j+G(\bar{X}^n_{j+1})/n.
\]
Thus recursively we obtain $\{\bar Y^n_i, \bar U^n_i, \bar X^n_i\}_{i=1}^n $.
Using these random variables we define the estimator
\begin{equation}\label{q:Zn2}
Z^n= e^{-nF(\bar{Y}^n_n)}\prod_{j=0}^{n-1}e^{H_1(\bar{a}(\bar{Y}^n_j,j/n))- \langle \bar{a}(\bar{Y}^n_j,j/n),\bar{X}^n_{j+1} \rangle}
\end{equation}
which as before is an unbiased estimator for \eqref{q:mainproblem_fixedtheta}. 

In comparison to schemes introduced in Section \ref{ss:change_on_U}, the main advantage of the  scheme proposed in the current section is the ease of implementation because, as discussed earlier, when $G$ takes a complex form, one can  simulate from $\eta_{\bar{a}(\bar{Y}^n_j, j/n)}$ more easily than from the distribution in \eqref{eq:eq813}. In order to motivate the choice of the function
$\bar a$ (or more generally a collection of functions $\{\bar a_k\}_{k=1}^K$) for constructing a ``good'' importance sampling estimator, we proceed as in \cite{dupuis2004importance} by identifying an Issacs equation associated with the control problem of minimizing the asymptotic second moment of $Z^n$. The discussion below leading to the partial differential equation in \eqref{q:isaccs_H} will be formal, however it will lead to an importance sampling estimator with rigorous asymptotic performance bounds, as is shown in Theorem \ref{t:main}.

For each $i=0,\dots,n-1$ and each $y \in \mathbb{R}^m$, we define a quantity $V^n(y,i)$ as
\begin{equation}\label{q:def_Vn}
V^n(y,i) = \inf_{\bar{a}} \mathbb{E}_y\left[ \left( e^{-nF(\bar{Y}^n_n)}\prod_{j=i}^{n-1}e^{H_1(\bar{a}(\bar{Y}^n_j,j/n))- \langle \bar{a}(\bar{Y}^n_j,j/n),\bar{X}^n_{j+1} \rangle} \right)^2 \right],
\end{equation}
where the minimum is taken among all possible choices of the function $\bar{a}$, the subscript $y$ in $\mathbb{E}_y$ refers to the fact that $\bar{Y}^n_i=y$, and the values of $\bar{X}^n_{i+1}, \cdots, \bar{X}^n_{n}, \bar{Y}^n_{i+1}, \cdots, \bar{Y}^n_{n}$ are generated using the conditional distributions
$\{\eta_{\bar a(\bar Y_j^n, j/n)}\}_{j=i}^{n-1}$ with $\bar{Y}^n_{j}= y + \sum_{l=i+1}^j G(\bar{X}^n_l)$. 
Let $V^n(y,n)=\exp\{ -2nF(y) \}$. Note that $V^n(0,0)$ is the minimum value of the second moment of $Z^n$ that can be achieved over all possible choices of functions $\bar a: \mathbb{R}^m \times [0,1]\to \mathbb{R}^h$.

Using the property of Radon-Nikodym derivatives, we can rewrite $V^n(y,i)$ in terms of the original random variables $X_{i+1},\cdots, X_n$ as
\[
V^n(y,i)  = \inf_{\bar{a}} \mathbb{E}_y\left[  e^{-2nF(Y_n)}\prod_{j=i}^{n-1}e^{H_1(\bar{a}(Y^n_j,j/n))- \langle \bar{a}(Y^n_j,j/n),X_{j+1} \rangle} \right]
\]
where $Y^n_i=y$ and $Y^n_{j+1}=Y^n_j+G(X_{j+1})/n$ for $j=i,\cdots,n-1$ and as before $\{X_i\}$ are iid with distribution $\eta$. 
By a  standard conditioning argument we get the following dynammic programming equation
\begin{align*}
V^n(y,i)
& = \inf _{a\in \mathbb{R}^h} \int_{\mathbb{R}^h} e^{H_1(a)-\langle a,x\rangle} V^n(y+G(x)/n,i+1)\eta(dx).
\end{align*}
Next, define $W^n(y,i) = -\frac{1}{n}\log V^n(y,i)$ for each $y\in \mathbb{R}^m$ and $i=0,\cdots,n$. For $i < n$ we can write $W^n(y,i)$ as
\begin{align*}
W^n(y,i) & = -\frac{1}{n}\log V^n(y,i)\\
& = -\frac{1}{n} \log\inf_{a\in \mathbb{R}^h} \int_{\mathbb{R}^h} e^{H_1(a)-\langle a,x\rangle +\log V^n(y+G(x)/n,i+1)}\eta(dx)\\
& =-\frac{1}{n} \inf_{a\in \mathbb{R}^h} \log \int_{\mathbb{R}^h} e^{H_1(a)-\langle a,x\rangle -n W^n(y+G(x)/n,i+1)}\eta(dx).
\end{align*}
From the Donsker-Varadhan relative entropy formula (see e.g. \cite[Proposition 1.4.2]{dupuis2011weak}) we have
\begin{align}
&W^n(y,i) \nonumber\\
&\quad = \sup_{a\in \mathbb{R}^h} \inf_{\mu \in \mathscr{P}(\mathbb{R}^h)}\left[\frac{1}{n}\bigg(R(\mu\|\eta) -H_1(a)+ \int_{\mathbb{R}^h} \langle a,x\rangle \mu(dx) \bigg) + \int_{\mathbb{R}^h} W^n(y+G(x)/n,i+1)\mu(dx) \right].
\label{q:W}
\end{align}\par

Continuing to proceed formally, suppose $W: \mathbb{R}^m\times[0,1] \rightarrow \mathbb{R}$ is a continuously differentiable function such that $W^n(y,i) = W(y, i/n)$. Applying the Taylor expansion on $W(y+G(x)/n,(i+1)/n)$, we have, neglecting higher order terms,
\begin{align*}
& \int_{\mathbb{R}^h} W\big(y+G(x)/n,(i+1)/n\big)\mu(dx) \\
\approx & W(y,i/n) + \frac{1}{n}W_t(y,i/n)+ \frac{1}{n}\int_{\mathbb{R}^h}  \langle DW(y,i/n), G(x) \rangle \mu(dx),
\end{align*}
where $W_t$ and $DW$ are the derivatives of $W$ w.r.t. $t$ and $y$ respectively. We can then rewrite \eqref{q:W} in terms of $W$ as:
\begin{align*}
0= \sup_{a\in \mathbb{R}^h} \inf_{\mu \in \mathscr{P}(\mathbb{R}^h)}\left[R(\mu\|\eta) -H_1(a)+ \int_{\mathbb{R}^h} \langle a,x\rangle \mu(dx) + W_t(y,t)+ \int_{\mathbb{R}^h}  \langle DW(y,t), G(x) \rangle \mu(dx) \right].
\end{align*}
Using the representation (\ref{L representation}) the above equation can be rewritten as
\begin{align*}
\sup_{a\in \mathbb{R}^h}\inf_{(b,\beta)\in \mathbb{R}^{h+m}}\left[L(b,\beta) -H_1(a)+ \langle a,b\rangle+ W_t(y,t)+ \langle DW(y,t), \beta \rangle \right]
= 0.
\end{align*}

Finally, we define $\mathbb{H}: \mathbb{R}^{2m+2h} \rightarrow \mathbb{R}\cup \{\infty\}$ as
\begin{equation}\label{eq:eq347}
	\mathbb{H}(s,a,b,\beta) = \langle a, b \rangle + \langle s, \beta \rangle +L(b,\beta) - H_1(a), \ s,\beta \in \mathbb{R}^m, \ a, b \in \mathbb{R}^h,\end{equation}
and obtain the following Issacs equation
\begin{equation}\label{q:isaccs_H}
W_t(y,t)+ \sup_{a\in \mathbb{R}^h}\inf_{(b,\beta)\in \mathbb{R}^{h+m}} \mathbb{H}(DW(y,t),a,b, \beta)=0,
\end{equation}
From the equality $V^n(y,n)=\exp\{ -2nF(y) \}$ we see that the above PDE is accompanied with the terminal condition $W(y,1)=2F(y)$.

With the above formal derivation as the basis, we now turn to rigorous results. As in Section \ref{ss:change_on_U}, we begin with some definitions.
A continuously differentiable function $\bar{W}: \mathbb{R}^m \times [0,1] \rightarrow \mathbb{R}$ is a classical subsolution to \eqref{q:isaccs_H} if it satisfies
\begin{align}\label{Isaacs subsolution}
\bar{W}_t(y,t) + \sup_{a\in \mathbb{R}^h}\inf_{(b,\beta)\in \mathbb{R}^{h+m}} \mathbb{H}(D\bar{W}(y,t),a,b, \beta) \geq 0
\end{align}
for each $(y,t)\in \mathbb{R}^{m}\times [0,1]$.
If functions $\bar{W}:\mathbb{R}^m\times [0,1] \rightarrow \mathbb{R}$, $\rho_k: \mathbb{R}^m \times [0,1] \rightarrow \mathbb{R}$, $\bar{a}_k: \mathbb{R}^m \times [0,1] \rightarrow \mathbb{R}^h$, $1 \leq k \leq K$ satisfy all conditions in
Definition \ref{def_gen_sub} (with $\bar \alpha_k$ replaced by $\bar a_k$) except that \eqref{q:def_gen_sub} is replaced by
\begin{equation}\label{q:def_gen_sub_H}
	r_k(y,t) + \inf_{(b,\beta)\in \mathbb{R}^{h+m}} \mathbb{H} (s_k(y,t);\bar{a}_k(y,t),b,\beta) \geq 0,
	\end{equation}
then $(\bar{W},\rho_k,\bar{a}_k)$ is said to be a generalized subsolution/control to \eqref{q:isaccs_H}. For the special case in which $K=1$ and $\rho_1=1$, we abbreviate the notation $(\bar{W}, \rho_k, \bar{a}_k)$ as $(\bar{W}, \bar{a})$ and call it a subsolution/control pair.

A dynamic change of measure, analogous to Section \ref{ss:change_on_U}, based on a generalized subsolution/control $(\bar{W},\rho_k,\bar{a}_k)$ is constructed as follows. Let $\bar{Y}^n_0=0$.
For $j=0,\dots,n-1$, having constructed $\{\bar X^n_i\}_{1\le i\le j}$ and $\{\bar Y^n_i\}_{1\le i\le j}$,
we generate a multinomial random variable $I$ with (conditional) probabilities $\mathbb{P}[I=k] = \rho_k(\bar{Y}^n_j,j/n)$ for $k\in\{1,2,\dots,K\}$. Next, we sample $\bar{X}^n_{j+1}$ from the distribution
\begin{equation}\label{q:exp_change_X}
e^{\langle \bar{a}_I(\bar{Y}^n_j,j/n), x \rangle - H_{1}(\bar{a}_I(\bar{Y}^n_j,j/n))}\eta(dx)
\end{equation}
 and define
$\bar{Y}^n_{j+1}= \bar{Y}^n_j + \frac{1}{n} G (\bar{X}^n_{j+1})$. Finally, we define
\begin{equation}\label{q:Z^n_H}
Z^n= e^{-nF(\bar{Y}^n_n)}\prod_{j=0}^{n-1}\left[\sum_{k=1}^{K}\rho_k(\bar{Y}^n_j,j/n)e^{\langle \bar{a}_k(\bar{Y}^n_j,j/n),\bar{X}^n_{j+1} \rangle - H_1(\bar{a}_k(\bar{Y}^n_j,j/n))}\right]^{-1}
\end{equation}
which as before is an unbiased estimator for \eqref{q:mainproblem_fixedtheta}.
To reiterate, the appeal of the estimator in \eqref{q:Z^n_H} over the estimator in \eqref{eq:eq525} is that, in many examples
it is much easier to simulate from \eqref{q:exp_change_X} than from \eqref{q:exp_change}.


Theorem \ref{t:main} below  which is an analogue of \cite[Theorem 8.1]{dupuis2007subsolutions} shows that the second moment of $Z^n$ decays exponentially at a rate no slower than $\bar{W}(0,0)$. The proof is given in the appendix.

\begin{theorem}\label{t:main}
	Assume that  $H(a,\alpha) < \infty$ for all $(a,\alpha)\in \mathbb{R}^{h+m}$, that $(\bar{W},\rho_k,\bar{a}_k)$ is a generalized subsolution/control to \eqref{q:isaccs_H} and satisfies the terminal condition $\bar{W}(y,1) \leq 2F(y)$ for all $y\in\mathbb{R}^m$, and that $Z^n$ is as defined in \eqref{q:Z^n_H}. Then
	$$\liminf_{n\rightarrow\infty}-\frac{1}{n}\log\mathbb{E}[(Z^n)^2]\geq \bar{W}(0,0).$$	
\end{theorem}

In practice, we will like to construct a generalized subsolution/control $(\bar{W},\rho_k,\bar{a}_k)$ that has a simple form and for which the value of $\bar W(0,0)$ is as large as possible. For this, we first consider subsolution/control pairs $(\bar W, \bar a)$, as defined below \eqref{q:def_gen_sub_H}, for which $\bar W$ is an affine function of $(y,t)$ and $\bar a$ is in fact a constant.

If we write $\bar{W}$ in the form
\begin{equation}\label{q:W_bar}
\bar{W}(y,t) = \bar{c} + \langle u,y \rangle - (1-t)v \text{ for some } \bar{c} \in \mathbb{R}, u \in \mathbb{R}^m, v\in \mathbb{R},
\end{equation}
then $(\bar{W},\bar{a})$ is a subsolution/control pair if the following inequality holds for all $(y,t)\in \mathbb{R}^{m+1}$:
\begin{equation}\label{q:a_bar}
\bar{W}_t(y,t) + \inf_{(b,\beta)\in \mathbb{R}^{h+m}} \mathbb{H}(D\bar{W}(y,t),\bar{a},b,\beta) \geq 0,
\end{equation}
namely
\begin{equation}\label{eq:eq932}
	v + \inf_{(b,\beta)\in \mathbb{R}^{h+m}} \mathbb{H}(u, \bar a, b, \beta) \ge 0.
\end{equation}
Next, we select a finite collection of pairs $\{(\bar{W}_k,\bar{a}_k), k=1,\dots,K \}$ from this family of subsolution/control pairs, such that the pointwise minimum $\bar W \doteq \wedge^K_{k=1}\bar{W}_k$ defined as $\bar W(y,t) = \wedge^K_{k=1}\bar{W}_k(y,t)=\min_{k=1,\cdots,K}\bar{W}_k(y,t)$ satisfies
\begin{equation}\label{q:wedge_W}
\wedge^K_{k=1}\bar{W}_k(y,1) \leq 2F(y) \text{ for all } y\in \mathbb{R}^m.
\end{equation}
In the process of choosing $\{(\bar{W}_k,\bar{a}_k), k=1,\dots,K \}$ we also maximize $\wedge^K_{k=1}\bar{W}_k(0,0)$ among all qualified choices.
Finally, we choose a small positive number $\delta$, and define
\begin{equation}\label{q:W_delta}
\bar{W}^{\delta}(y,t) \doteq -\delta\log \left( \sum_{k=1}^{K}e^{-(1/\delta)\bar{W}_k(y,t)} \right)
\end{equation}
	and
\begin{equation}\label{q:rho_delta}
\rho^{\delta}_k(y,t) \doteq \frac{e^{-(1/\delta)\bar{W}_k(y,t)} }{\sum_{i=1}^{K}e^{-(1/\delta)\bar{W}_i(y,t)} }
\text{ for } 1\leq k \leq K.
\end{equation}
Then, following  \cite{dupuis2007subsolutions}, we see that $(\bar{W}^{\delta}, \rho^{\delta}_k, \bar{a}_k)$ is a generalized subsolution/control with
$$\wedge^K_{k=1}\bar{W}_k(y,t) \geq \bar{W}^{\delta}(y,t) \geq \wedge^K_{k=1}\bar{W}_k(y,t)-\delta\log K \text{ for all } (y,t).$$
In particular, the difference between $ \bar{W}^{\delta}(0,0)$ and $\wedge^K_{k=1}\bar{W}_k(0,0)$ is no larger than $\delta\log K$.
Thus the estimator $Z^n$ based on this generalized subsolution/control satisfies
	\begin{equation}\label{eq:eq446}
		\liminf_{n\rightarrow\infty}-\frac{1}{n}\log\mathbb{E}[(Z^n)^2]\geq \bar{W}^{\delta}(0,0) \ge \bar{W}(0,0) - \delta\log K.
		\end{equation}	
In Section \ref{s:computation} we illustrate the implementation of such a construction for some  examples.
\section{Analysis of some approximate problems} \label{s:approximation}

In  this section, we consider the minimization problem introduced in \eqref{q:mainproblem} and study the relation between its optimal solution and  optimal solutions of certain associated approximating problems. This relationship provides a justification for using solutions of the approximating problems as estimates of the true solution of \eqref{q:mainproblem}, as we will do in numerical examples of Section \ref{s:computation}.

Denote the objective function of \eqref{q:mainproblem}  as
\begin{equation}\label{q:def_p}
p(\theta)=\mathbb{E}\exp\left\{-nF\bigg(\frac{1}{n} \sum_{i=1}^{n}G(X_i,\theta)\bigg)\right\}.
\end{equation}
It is possible for $p$ to be differentiable even if $F$ is not differentiable everywhere. However, the gradient of $p$ is not given by the expectation of the gradient of the function inside the expecation w.r.t. $\theta$, unless additional conditions hold (see, e.g., \cite[Theorem 7.49]{shapiro2009lectures}). Those conditions are not satisfied with $F(y)=\infty 1_{A^c}(y)$, which is one of the case we are  interested in. There are also formulas for gradients of certain types of probability functions, see, e.g., \cite{marti1996differentiation,uryasev1995derivatives,van2014gradient}, but it is not practical to apply those results to the problem here, because the large value of $n$ and the extremely small probability will require an extremely large sample size
in any such application. To utilize a gradient based optimization algorithm to solve \eqref{q:mainproblem}, we approximate $F$ by a continuous function $\varphi: \mathbb{R}^m \to \mathbb{R}$, and use a solution to the approximation problem
\begin{equation}\label{q:appro_phi}
\min_{\theta \in \Theta} \mathbb{E}\exp\left\{-n \varphi \bigg(\frac{1}{n} \sum_{i=1}^{n}G(X_i,\theta)\bigg)\right\}
\end{equation}
as an estimate for the solution of \eqref{q:mainproblem}. The function $\varphi$ will be chosen so that the gradient of the objective function of \eqref{q:appro_phi} is given by the expectation of the gradient of the function inside the expectation.

The following proposition shows that the function $\varphi$ can be chosen in an appropriate manner to guarantee the solution to \eqref{q:appro_phi} to be sufficiently close to the solution of \eqref{q:mainproblem}, as long as $\varphi$ is a sufficiently close approximation of $F$.

  \begin{proposition}\label{p:phi_approximate_prob}
	Suppose that $\Theta\subset \mathbb{R}^d$ is compact and $F:\mathbb{R}^m \to \mathbb{R}\cup\{\infty\}$ is upper semicontinuous. In addition, let $\{\varphi^k\}_{k\in \mathbb{N}}$ be a sequence of continuous functions from $\mathbb{R}^m$ to $\mathbb{R}$, such that $\varphi^1$ is bounded from below, $\varphi^k(y) \leq \varphi^{k+1}(y)$ for all $k\in \mathbb{N}$ and $y\in \mathbb{R}^m$, and $\lim_{k\to\infty }\varphi^k(y)= F(y)$ for all $y\in \mathbb{R}^m$. For each $k\in \mathbb{N}$, define
\begin{equation}\label{q:def_pk}
p^k(\theta)=\mathbb{E}\exp\left\{-n \varphi^k \bigg(\frac{1}{n} \sum_{i=1}^{n}G(X_i,\theta)\bigg)\right\}, \ \theta \in \Theta.
\end{equation}
		Then
	$$\lim_{k\to\infty} \min_{\theta\in \Theta} p^k(\theta) = \min_{\theta\in \Theta}  p(\theta),$$
	and for any choice of $\delta^k \downarrow 0$ and $\theta^k \in \delta^k-\argmin_{\theta\in \Theta} p^k$ (i.e.
	$p^k(\theta^k) \le \min_{\theta\in \Theta} p^k(\theta) + \delta^k$), all cluster points of the sequence $\{ \theta^k \}_{k \in \mathbb{N}}$ belong to $\argmin_{\theta\in \Theta} p$. If $\argmin_{\theta\in \Theta} p$ consists of a unique point $\theta^*$, one must actually have $\theta^k \rightarrow \theta^*$.
\end{proposition}
\begin{proof}
Because $\varphi^1$ is bounded from below and $\varphi^k(y)\uparrow F(y)$ for all $y$, we can apply the dominated convergence theorem to conclude that $p^k(\theta)\downarrow p(\theta)$ for all $\theta\in \Theta$. It follows from the continuity of $G$, the upper semicontinuity of $F$, the continuity of $\varphi^k$ and Fatou's lemma  that  $p$ is lower semicontinuous and each $p^k$ is continuous on $\Theta$.
By an application of \cite[Proposition 7.4(c)]{rockafellar2009variational}, $p^k$ epi-converges to $p$. With the compactness of $\Theta$, all conclusions of the present proposition follows from \cite[Theorem 7.33]{rockafellar2009variational}.
\end{proof}

When Proposition \ref{p:phi_approximate_prob} is applied to the case $F(y)=\infty1_{A^c}(y)$, the upper semicontinuity assumption of $F$ amounts to requiring $A$ to be an open set. If the distribution of $G(X_i,\theta)$ is absolutely continuous for each $\theta$ and the boundary of $A$ (denoted as $\text{bdry} A$) has Lebesgue measure 0, then $\mathbb{P}\left[ \frac{1}{n}\sum_{i=1}^{n}G(X_i,\theta) \in \text{bdry} A \right]=0$ and we can replace $A$ by the interior of $A$ without changing the value of $\mathbb{P}\left[ \frac{1}{n}\sum_{i=1}^{n}G(X_i,\theta) \in  A\right]$, which is in fact a continuous function of $\theta$  in this case.

Next, we consider the problem \eqref{q:appro_phi} with a fixed continuous function $\varphi$, and study its convergence as $n\to \infty$. Note that our main interest is in solving \eqref{q:mainproblem} or its approximation \eqref{q:appro_phi} for a fixed value of $n$. Nonetheless,  the convergence behavior of  \eqref{q:appro_phi} as $n\to \infty$ provides information about sensitivity of the solution of \eqref{q:appro_phi} with respect to $n$, and can also be used in computation to find an initial point in solving \eqref{q:appro_phi}. For this purpose, we define functions $g^n: \Theta\to \mathbb{R}$ and $g:\Theta \to \mathbb{R}$ as
\begin{equation}\label{q:def_gn}
g^n(\theta) = -\frac{1}{n} \log \mathbb{E}\left[  \exp \left\{ -n\varphi\left( \frac{1}{n}\sum_{i=1}^{n}G(X_i,\theta) \right) \right\} \right]
\end{equation}
and
\begin{equation}\label{q:def_g}
g(\theta) = \inf_{\nu \in \mathscr{P}(\mathbb{R}^h)} \left[  \varphi\left( \int_{\mathbb{R}^h} G(x,\theta)\nu(dx) \right) + R(\nu||\eta) \right].
\end{equation}
Note that $g^n(\theta)$ is simply $-\frac{1}{n}$ times the log of the objective function of \eqref{q:appro_phi}, so \eqref{q:appro_phi} is equivalent to
$\max_{\theta\in \Theta} g^n(\theta)$.

Theorem \ref{t:limiting} below shows that $g^n$ converges to $g$ uniformly under suitable conditions.
Let $H^{\theta}_2$ denote the log moment generating function of $G(X_1, \theta)$, namely,
\begin{equation}\label{q:def_H2_theta}
H^{\theta}_2(\alpha)=\log \mathbb{E} \ e^{\langle \alpha,G(X_1,\theta) \rangle }, \ \alpha \in \mathbb{R}^m.
\end{equation}
 Also, let  $L_2^{\theta}$ denote the Legendre transform of $H^{\theta}_2$, i.e.,
\begin{equation}\label{q:def_L2_theta}
L^\theta_2(\beta) = \sup_{\alpha \in \mathbb{R}^m} \left( \langle \alpha,\beta \rangle -H^\theta_2(\alpha) \right), \ \beta \in \mathbb{R}^m.
\end{equation}

\begin{theorem}
	Let $\Theta$ be a compact subset of $\mathbb{R}^d$.
	Assume that  $\sup_{\theta \in \Theta} H^{\theta}_2(\alpha)<\infty$ for all $\alpha\in \mathbb{R}^{m}$.
 If $\varphi$ is continuous and bounded,
	then $g^n\to g$ uniformly on $\Theta$.
	\label{t:limiting}
\end{theorem}
\begin{proof}
	 Let $\{X_i\}_{i \in \mathbb{N}}$ be iid random variables with distribution $\eta$, and let $\mathcal{L}^n$  be the empirical measure in $\mathbb{R}^h$ that puts mass $1/n$ at each of the first $n$ points $X_1,\cdots, X_n$, namely $\mathcal{L}^n(dx) = \frac{1}{n}\sum_{i=1}^n \delta_{X_i}(dx)$.
From the representation established in \cite[Section 2.3]{dupuis2011weak} we have, for $\theta \in \Theta$,
\begin{equation}\label{q:inf_bar_X}
\begin{split}
 g^n(\theta) = & -\frac{1}{n}\log \mathbb{E} \exp \left\{ -n \varphi\left( \int_{\mathbb{R}^h} G(x,\theta)\mathcal{L}^n(dx) \right) \right\}= \\
	&\inf_{\bar{\nu}^n} \mathbb{E} \left[ \varphi \left( \int_{\mathbb{R}^h} G(x,\theta)\bar{\mathcal{L}}^n(dx) \right) + \frac{1}{n} \sum_{i=1}^{n}R(\bar{\nu}^n_i \| \eta) \right]
\end{split}
\end{equation}
	where the infimum is over all probability distributions $\bar{\nu}^n \in \mathscr{P}(\mathbb{R}^{nh})$, with $(\bar{X}^n_1,\dots,\bar{X}^n_n)$ being a random variable with distribution $\bar{\nu}^n$,  $\bar{\mathcal{L}}^n$ being the empirical measure in $\mathbb{R}^h$ of the $n$ points $\bar{X}^n_1,\dots,\bar{X}^n_n$, and $\bar{\nu}^n_i$ being the
	conditional distribution of $\bar{X}^n_i$ given $\bar{X}^n_1,\cdots, \bar{X}^n_{i-1}$. Since $\varphi$ is bounded, the infimum in \eqref{q:inf_bar_X} is bounded above by $\| \varphi \|_{\infty} = \sup_{y \in \mathbb{R}^m} | \varphi(y)| < \infty$. It follows that for any fixed value of $n\in \mathbb{N}$, in taking the infimum in \eqref{q:inf_bar_X} we can restrict to distributions $\bar{\nu}^n$ for which
	\begin{equation}
     \mathbb{E} \left[ \frac{1}{n} \sum_{i=1}^{n} R(\bar{\nu}^n_i \| \eta) \right] \leq 2\|\varphi\|_{\infty} + 1.
     \label{q:entropy bound}
	\end{equation}
 Under our assumption $\sup_{\theta} H^\theta_2(\alpha)<\infty$, by a standard argument (see e.g.  the proof of Lemma \ref{lemma A.1}), for any sequence $\{\bar{\nu}^n\}_{n\in \mathbb{N}} $ that satisfies \eqref{q:entropy bound} for all $n$ we see that
	\begin{equation}
	\lim_{C\rightarrow \infty} \sup_{n\in\mathbb{N}} \sup_{\theta \in \Theta}\mathbb{E} \left[ \int_{\mathbb{R}^h} \|G(x,\theta)\| 1_{\{ \|G(x,\theta)\| \geq C \}} \bar{\mathcal{L}}^n(dx) \right] = 0.
	\label{q:u.i. of bar{L} 2}
	\end{equation}

Now let $\{\theta^n\}$ be a sequence in $\Theta$ such that $\theta^n\to \theta$ as $n\to \infty$. Fix $\varepsilon > 0$ and let the sequence $\{\bar{\nu}^n\}$ satisfy
	$$-\frac{1}{n} \log \mathbb{E} e^{-n \varphi\left( \int_{\mathbb{R}^h}G(x, \theta^n)\mathcal{L}^n(dx) \right)} + \varepsilon \geq \mathbb{E}\left[ \varphi\left( \int_{\mathbb{R}^h}G(x, \theta^n)\bar{\mathcal{L}}^n(dx) \right) + \frac{1}{n}\sum_{i=1}^{n}R(\bar{\nu}^n_i \| \eta) \right]$$
as well as \eqref{q:entropy bound} for each $n$, and  define $\hat{\nu}^n \doteq \frac{1}{n}\sum_{i=1}^{n}\bar{\nu}^n_i$. Using arguments similar to Proposition 8.2.5 and  Lemma 8.2.7 in \cite{dupuis2011weak}, $\{ (\bar{\mathcal{L}}^n,\hat{\nu}^n) \}_{n \in \mathbb{N}}$ is tight. Consider a subsequence along which $(\bar{\mathcal{L}}^n,\hat{\nu}^n)$ converges weakly to $(\bar{\mathcal{L}},\hat{\nu})$.
We now argue that (along the subsequence)
\begin{equation}\label{eq:eq1009}
	\lim_{n\to\infty}\mathbb{E}\left[ \varphi\left( \int_{\mathbb{R}^h}G(x, \theta^n)\bar{\mathcal{L}}^n(dx) \right) \right] =\mathbb{E}\left[ \varphi\left( \int_{\mathbb{R}^h}G(x, \theta)\bar{\mathcal{L}}(dx) \right) \right].
\end{equation}
By appealing to Skorohod representation theorem we can
assume that $\bar{\mathcal{L}}^n(\om) \to \bar{\mathcal{L}}(\om)$ for a.e. $\om$. We have
\begin{align*}
	\mathbb{E} \left| \int_{\mathbb{R}^h}G(x, \theta^n)\bar{\mathcal{L}}^n(dx) - \int_{\mathbb{R}^h}G(x, \theta)\bar{\mathcal{L}}(dx)\right| &\le
	\mathbb{E}  \int_{\mathbb{R}^h}\left|G(x, \theta^n)- G(x,\theta)\right| \bar{\mathcal{L}}^n(dx) \\
	& \quad + \mathbb{E}\left| \int_{\mathbb{R}^h}G(x, \theta)\bar{\mathcal{L}}^n(dx) - \int_{\mathbb{R}^h}G(x, \theta)\bar{\mathcal{L}}(dx)\right|.
\end{align*}
The second term on the right side in the above display converges to zero from the continuity of $G$, \eqref{q:u.i. of bar{L} 2} and the convergence of $\bar{\mathcal{L}}^n$ to $\bar{\mathcal{L}}$. The
first term also converges to zero as follows from \eqref{q:u.i. of bar{L} 2}, the fact that the sequence $\{\bar{\mathcal{L}}^n\}$ is tight,  and that for every compact subset $K$ of $\mathbb{R}^h$,
$\sup_{x\in K}|G(x, \theta^n)- G(x,\theta)|\to 0$ as $n\to \infty$.
From the boundedness and continuity of $\varphi$ and the dominated convergence theorem we now have the convergence in \eqref{eq:eq1009}.
Consequently, we have	
	\begin{align*}
	 &\liminf_{n \rightarrow \infty} g^n(\theta^n) + \varepsilon
	 =\liminf_{n \rightarrow \infty} - \frac{1}{n} \log \mathbb{E} \exp\left\{ -n \varphi\left( \int_{\mathbb{R}^h}G(x, \theta^n)\mathcal{L}^n(dx) \right) \right\} + \varepsilon\\
	 \geq & \liminf_{n \rightarrow \infty} \mathbb{E}\left[ \varphi\left( \int_{\mathbb{R}^h}G(x, \theta^n)\bar{\mathcal{L}}^n(dx) \right) + \frac{1}{n}\sum_{i=1}^{n}R(\bar{\nu}^n_i \| \eta) \right]\\
	 \geq & \liminf_{n \rightarrow \infty} \mathbb{E}\left[ \varphi\left( \int_{\mathbb{R}^h}G(x, \theta^n)\bar{\mathcal{L}}^n(dx) \right) + R(\hat{\nu}^n \| \eta) \right]\\
	 \geq & \mathbb{E}\left[ \varphi\left( \int_{\mathbb{R}^h}G(x, \theta)\bar{\mathcal{L}}(dx) \right) + R(\hat{\nu} \| \eta) \right] \\
	 \geq & \inf_{\nu \in \mathscr{P}(\mathbb{R}^h)} \left[ \varphi\left( \int_{\mathbb{R}^h}G(x, \theta)\nu(dx) \right) + R(\nu \| \eta) \right]
	 =  g(\theta),
	\end{align*}
where the second inequality holds by Jensen's inequality and convexity of relative entropy, the third inequality follows from the convergence in distribution, Fatou's Lemma and lower semicontinuity of relative entropy, and the fourth inequality follows from the fact that $\bar{\mathcal{L}} = \hat{\nu}$ a.s., see \cite[Theorem 8.2.8]{dupuis2011weak}.
Since $\varepsilon>0$ is arbitrary, we have $\liminf g^n(\theta^n) \ge g(\theta)$.

We now consider the reverse inequality. Once more, let $\theta^n\to \theta$. We first argue that $g(\theta^n) \to g(\theta)$. Note that,   for $\theta \in \Theta$
\begin{align*}
	g(\theta) &= \inf_{\nu \in \mathscr{P}(\mathbb{R}^h)} \left[ \varphi\left( \int_{\mathbb{R}^h}G(x, \theta)\nu(dx) \right) + R(\nu \| \eta) \right]\\
	&= \inf_{\nu \in \mathscr{P}(\mathbb{R}^h): R(\nu \| \eta)\le \|\varphi\|_{\infty}} \left[ \varphi\left( \int_{\mathbb{R}^h}G(x, \theta)\nu(dx) \right) + R(\nu \| \eta) \right].
\end{align*}
Fix $\varepsilon>0$ and let $\nu^n$, $\nu^0$ be $\varepsilon$-optimal for $g(\theta^n)$ and $g(\theta)$, respectively, and such that $R(\nu^n \| \eta) \le \|\varphi\|_{\infty}$,
$R(\nu^0 \| \eta) \le \|\varphi\|_{\infty}$.
Then the sequence $\{\nu^n\}$ is tight and in a similar manner as for the proof of \eqref{q:u.i. of bar{L} 2} we have
\begin{equation}
\lim_{C\rightarrow \infty} \sup_{n\ge 0} \sup_{\theta \in \Theta} \int_{\mathbb{R}^h} \|G(x,\theta)\| 1_{\{ \|G(x,\theta)\| \geq C \}} \nu^n(dx)  = 0.
\label{q:u.i. of bar{L} 2b}
\end{equation}
In particular, as $n\to \infty$
\begin{equation}\label{eq:eq1106}
\left|\varphi\left( \int_{\mathbb{R}^h}G(x, \theta^n)\nu^n(dx) \right) - \varphi\left( \int_{\mathbb{R}^h}G(x, \theta)\nu^n(dx) \right)\right| \to 0\end{equation}
and
\begin{equation}\label{eq:eq1106b}
\left|\varphi\left( \int_{\mathbb{R}^h}G(x, \theta^n)\nu^0(dx) \right) - \varphi\left( \int_{\mathbb{R}^h}G(x, \theta)\nu^0(dx) \right)\right| \to 0.
\end{equation}
From the $\varepsilon$-optimality of $\nu^n$, we have
\begin{align*}
	\limsup_{n\to \infty} (g(\theta)-g(\theta^n)) &\le 	\limsup_{n\to \infty}\left[\varphi\left( \int_{\mathbb{R}^h}G(x, \theta)\nu^n(dx) \right) - \varphi\left( \int_{\mathbb{R}^h}G(x, \theta^n)\nu^n(dx) \right)\right] +\varepsilon \\
	&\le \varepsilon,
\end{align*}
where the second inequality follows from \eqref{eq:eq1106}. Similarly, using \eqref{eq:eq1106b} we see that 	$\limsup_{n\to \infty} (g(\theta^n)-g(\theta)) \le \varepsilon$.
Since $\varepsilon>0$ is arbitrary, we have shown that
\begin{equation}
	g(\theta^n) \to g(\theta) \mbox{ as } n \to \infty.
\label{eq:eq1109}	
\end{equation}
Next with $\varepsilon, \nu^n$ as above, define $\bar{\mathcal{L}}^n$ as the empirical measure of $\{\bar X^n_i\}_{i=1}^n$ which are iid $\nu^n$.
Using \eqref{q:u.i. of bar{L} 2b}  it can be seen that the sequence $\{\bar{\mathcal{L}}^n\}$ satisfies \eqref{q:u.i. of bar{L} 2}. Also, for every bounded
$\tilde G: \Theta\times \mathbb{R}^h \to \mathbb{R}$, as $n\to \infty$,
$$\int_{\mathbb{R}^h}\tilde G(x, \theta^n)\bar{\mathcal{L}}^n(dx)  - \int_{\mathbb{R}^h}\tilde G(x, \theta^n)\nu^n(dx) \to 0, \mbox{ in probability.}$$
Combining these two observations with the fact that $\varphi$ is continuous and bounded, we have that, as $n \to \infty$,
\begin{equation}\label{eq:eq1216}
	\delta^n \doteq \left|\mathbb{E}\left[ \varphi\left( \int_{\mathbb{R}^h}G(x, \theta^n)\bar{\mathcal{L}}^n(dx) \right)\right] -
 \varphi\left( \int_{\mathbb{R}^h}G(x, \theta^n)\nu^n(dx) \right)\right| \to 0.\end{equation}

Finally, from the representation in \eqref{q:inf_bar_X},
	\begin{align*}
	  \limsup_{n \rightarrow \infty} g^n(\theta^n)
	 &=\limsup_{n \rightarrow \infty} - \frac{1}{n} \log \mathbb{E} \exp\left\{ -n \varphi\left( \int_{\mathbb{R}^h}G(x, \theta^n)\mathcal{L}^n(dx) \right) \right\} \\
	 &\leq  \limsup_{n \rightarrow \infty} \mathbb{E}\left[ \varphi\left( \int_{\mathbb{R}^h}G(x, \theta^n)\bar{\mathcal{L}}^n(dx) \right) + \frac{1}{n}\sum_{i=1}^{n}R(\bar{\nu}^n_i \| \eta) \right]\\
	&\le \limsup_{n \rightarrow \infty} \left(\mathbb{E}\left[ \varphi\left( \int_{\mathbb{R}^h}G(x, \theta^n)\nu^n(dx) \right) + R(\nu^n \| \eta) \right] + \delta^n\right)\\
	&\le \limsup_{n \rightarrow \infty} g(\theta^n) + \varepsilon  = g(\theta) +\varepsilon,
	\end{align*}
	where the second inequality uses the fact that $\bar{\nu}^n_i = \nu^n$ for each $i$, and the third inequality uses \eqref{eq:eq1216} and the $\varepsilon$-optimality of $\nu^n$.
Since $\varepsilon$ is arbitrary, we have proved $\limsup_{n\rightarrow \infty} g^n(\theta^n) \le g(\theta)$. This completes the proof.
\end{proof}

As an immediate consequence of the above theorem we have the following corollary.

\begin{corollary}
Suppose the assumptions in Theorem \ref{t:limiting} hold.
Then
	$$\max_{\theta\in \Theta} g^n(\theta) \rightarrow \max_{\theta\in \Theta}  g(\theta),$$
	and for any choice of $\delta^n \downarrow 0$ and $\theta^n \in \delta^n-\argmax_{\theta\in \Theta} g^n$, all cluster points of the sequence $\{ \theta^n \}_{n \in \mathbb{N}}$ belong to $\argmax_{\theta\in \Theta} g$. If $\argmax_{\theta\in \Theta} g$ consists of a unique point $\theta^*$, one must actually have $\theta^n \rightarrow \theta^*$.
	\label{cor_argmin convergence of g^n}
\end{corollary}

The function $g$ in \eqref{q:def_g} can be represented using $H^\theta_2$.
If $H^{\theta}_2(\alpha) < \infty$ for all $\theta \in \Theta$ and $\alpha \in \mathbb{R}^m$, then by Cram\'{e}r's Theorem we have
\begin{equation}\label{q:g_maxminrepresentation}
\begin{split}
g(\theta) &= \inf_{\beta \in \mathbb{R}^m} \left[ \varphi(\beta) + L^\theta_2(\beta)\right]\\
& = \inf_{\beta \in \mathbb{R}^m} \left[ \varphi(\beta) + \sup_{\alpha \in \mathbb{R}^m} \left[  \langle \alpha,\beta \rangle - H^\theta_2(\alpha)  \right]\right]\\
& = \inf_{\beta \in \mathbb{R}^m} \sup_{\alpha \in \mathbb{R}^m} \left[ \varphi(\beta) +  \langle \alpha,\beta \rangle - \log \mathbb{E}e^{\langle \alpha,G(X_1,\theta) \rangle } \right].
\end{split}
\end{equation}
With the above representation, the problem $\max_{\theta\in \Theta} g(\theta)$ can be solved as a constrained optimization problem by converting the inner max-min problem into optimality conditions. A useful feature of $g$ is that its evaluation does not involve a rare event probability and therefore does not require the use of importance sampling. In the numerical examples, we first choose a fixed function $\varphi$ to obtain the approximation problem \eqref{q:appro_phi}, then solve the limiting problem  $\max_{\theta\in \Theta} g(\theta)$ numerically.
The solution   of the latter problem is used as the starting point for solving  \eqref{q:appro_phi}.

\section{Minimization of the buffered probability}\label{s:buffered}

In this section, we consider the special case in which $F(y)=\delta_A(y)$,  $m=1$ and $A = [0,\infty)$. In such a setting, an alternative reliability measure known as the buffered failure probability or the buffered probability of exceedance (abbreviated as the {\em buffered probability} in rest of the paper) can be used in place of the standard probability. The buffered probability was introduced in \cite{rockafellar2010buffered}, which also showed how to convert optimization problems with buffered probability constraints into convex programs using a result in \cite{rockafellar2000optimization}. An extension and more properties of the buffered probability were provided in  \cite{mafusalov2014buffered}.

In general, for a continuous 1-dimensional random variable $X$, and a scalar $c \in (\mathbb{E}[X], \text{ess sup}X)$ ($\text{ess sup}X$ is the essential supremum of $X$), the buffered probability  is defined as
  $$\bar{p}_c (X)= \mathbb{P}[X>q]$$
  where $q$ is the unique solution to the equation $\mathbb{E}[X| X > q] = c$; in addition, we define $\bar{p}_c(X)=0$ for $c\ge \text{ess sup}X$ and  $\bar{p}_c(X)=1$ for $c\le \mathbb{E}[X]$. For a detailed discussion and the definition that applies to a general distribution, see \cite{mafusalov2014buffered}. A direct consequence of the above definition is that $q\le c$ and  $\mathbb{P}[X > c] \leq \bar{p}_c (X)$. It was shown in \cite{mafusalov2014buffered} that the buffered probability can be equivalently represented as
  	$$
  	\bar{p}_c(X) = \left\{
  	\begin{array}{l l}
  	0, & \text{if } c \ge \text{ess sup}X;\\
  	\min_{\lambda \geq 0}\mathbb{E}[ \lambda(X-c) + 1 ]^+, & \text{if }  c < \text{ess sup}X.
  	\end{array}
  	\right.
  	$$
  	\label{proposition: buffered}

 The following theorem gives an important connection between buffered probabilities and the large deviations rate function.
 Specifically, it shows that, under conditions, when $X$ is replaced by the the sample mean of iid random variables, the buffered probability and the corresponding ordinary probability have the same asymptotic decay rate. 
 \begin{theorem}
\label{buff-ldp}
 Let $U_i$, $i\ge 1$ be an iid sequence of $\mathbb{R}$ valued random variables, and suppose that $M(\lambda) \doteq \mathbb{E}(e^{\lambda U_1}) < \infty$ for every $\lambda \in \mathbb{R}$.
Let
 $H(\lambda) \doteq \log M(\lambda)$ for $\lambda \in \mathbb{R}$ and $L$ be the Legendre transform of $H$, and suppose that $L$ is finite on $(0,\infty)$.  Write $Y_n \doteq \frac{1}{n} \sum_{i=1}^n U_i$ for $n\ge 1$.   Then for every $c> E(U_1)$ and $\gamma \ge 0$
 $$\lim_{n\to \infty} \frac{1}{n}   \log\min_{\lambda \ge \gamma} \mathbb{E}[\lambda (Y_n-c) + 1]^+ = \lim_{n\to \infty} \frac{1}{n} \log \mathbb{P}[Y_n>c] = -L(c).$$
 \end{theorem}
\begin{proof}
Without loss of generality we assume that  $\mathbb{E}(U_1)=0$.  Fix $c>0$. Since for $\lambda=0$,
$\log \mathbb{E}[ \lambda(X-c) + 1 ]^+ = 0$ and $L(c)\ge 0$, it suffices to prove the result with the minimization over $\{\lambda: \lambda > \gamma\}$ for every $\gamma\ge 0$.
Note that under the assumptions of the theorem, for every $\kappa>0$
$$\liminf_{n\to \infty} \frac{1}{n} \log \mathbb{P}(Y_n > \kappa) = \limsup_{n\to \infty} \frac{1}{n} \log \mathbb{P}(Y_n \ge \kappa) = -L(\kappa).$$
For  $\lambda>0$
$$
\mathbb{E}[\lambda (Y_n-c) + 1]^+ \ge \mathbb{E}\left ( [\lambda (Y_n-c) + 1]1_{\{Y_n> c\}}\right) \ge  \mathbb{P} (Y_n > c).$$
Thus, for any $\gamma \ge 0$,
\begin{align*}
\frac{1}{n}\log \min_{\lambda >\gamma}\mathbb{E}[\lambda (Y_n-c) + 1]^+ &\ge   \frac{1}{n}\log  \mathbb{P} (Y_n > c).
\end{align*}
Taking limit as $n\to \infty$, we have
$$
\liminf_{n\to \infty} \frac{1}{n}\log \min_{\lambda >\gamma}\mathbb{E}[\lambda (Y_n-c) + 1]^+ \ge \liminf_{n\to \infty} \frac{1}{n}\log  \mathbb{P} (Y_n > c) = -L(c).$$

Now we prove the complementary inequality. Choose $m\ge 1$ such that $L(c+m) > L(c)+1$.
Note that for $\lambda>0$
\begin{align*}
\mathbb{E}[\lambda (Y_n-c) + 1]^+
&= \mathbb{E}\left([\lambda (Y_n-c) + 1] 1_{\{Y_n \ge c-1/\lambda\}}\right)\\
&= \left\{\mathbb{E}\left([\lambda (Y_n-c) + 1] 1_{\{ c-1/\lambda \le Y_n \le c+m\}}\right)
+ \mathbb{E}\left([\lambda (Y_n-c) + 1] 1_{\{ Y_n > c+m\}}\right)\right\}.
\end{align*}
Let $\alpha^*_0\in \mathbb{R}$ be the dual point to $(c+m)$, namely
\begin{equation}\label{eq:duallhcm}
	L(c+m) = \sup_{\alpha \in \mathbb{R}} [\alpha (c+m) - H(\alpha)] = \alpha^*_0 (c+m) - H(\alpha^*_0).
\end{equation}
Note that
$\alpha^*_0>0$,  since by Jensen's inequality $H(\alpha^*_0)\ge \log (e^{\alpha^*_0  \mathbb{E}(U_1)}) =0$.
Given $\lambda>0$, choose $n(\lambda) \in \mathbb{N}$ such that for all $n \ge n(\lambda)$,
$\gamma_n \doteq \alpha^*_0 - \lambda/n>0$. Then for all such $n$
\begin{align}
	\mathbb{E}\left([\lambda (Y_n-c) + 1] 1_{\{ Y_n > c+m\}}\right)
	&\le \mathbb{E}\left(e^{\lambda (Y_n-c)}e^{n\gamma_n(Y_n-c-m)}\right)\nonumber\\
	&=  e^{-\lambda c -n\gamma_n(c+m)}\mathbb{E}\left(e^{(\lambda+n\gamma_n)Y_n}\right)\nonumber\\
	&= e^{n H(\gamma_n + \lambda/n)}e^{-\lambda c -n\gamma_n(c+m)}\nonumber\\
	& = e^{n H(\alpha^*_0)}e^{-n\alpha^*_0(c+m)}   e^{\lambda m} = e^{-nL(c+m) + \lambda m}.
	\label{eq:eq417}
\end{align}
Thus
$$\frac{1}{n}\log\mathbb{E}\left([\lambda (Y_n-c) + 1] 1_{\{ Y_n > c+m\}}\right) \le -L(c+m) + \frac{\lambda m}{n}
\le -L(c) -1 + \frac{\lambda m}{n}.$$
Also, for $\lambda>0$
$$\frac{1}{n}\log \mathbb{E}\left([\lambda (Y_n-c) + 1] 1_{\{ c-1/\lambda \le Y_n \le c+m\}}\right)
\le \frac{1}{n}\log \mathbb{P}\left(Y_n\ge c - 1/\lambda\right) + \frac{\log (m\lambda+1)}{n}.$$
We now have that for all $n \ge n(\lambda)$.
$$
\frac{1}{n}\log \mathbb{E}[\lambda (Y_n-c) + 1]^+ \le \frac{\log 2}{n} + \max\left\{-L(c) -1 + \frac{\lambda m}{n}
,  \frac{\log (m\lambda+1)}{n}+ \frac{1}{n}\log \mathbb{P}\left(Y_n \ge c-\frac{1}{\lambda}\right)\right\}.$$
Let $\varepsilon>0$ be arbitrary and let  $0< \delta_0 < \min\{\frac{1}{\gamma},c\}$, $n_0 \in \mathbb{N}$, be such that for all  $n\ge n_0$
$$
 \frac{1}{n}\log \mathbb{P}\left(Y_n \ge c-\delta_0\right) \le -L(c) + \varepsilon.$$
 Let $\lambda_0 = 1/\delta_0$ and $n_1 = \max\{n_0, n(\lambda_0)\}$. Then, for $n\ge n_1$
 \begin{align*}
 \min_{\lambda >\gamma}\frac{1}{n}\log \mathbb{E}[\lambda (Y_n-c) + 1]^+ &\le \frac{1}{n}\log \mathbb{E}[\lambda_0 (Y_n-c) + 1]^+\\
&\le \frac{\log 2}{n} + \max\left\{-L(c) -1 + \frac{\lambda_0 m}{n}
,  \frac{\log(m\lambda_0+1)}{n}+ \varepsilon -L(c) \right\}.
\end{align*}
Now choose $n_2\ge n_1$ such that for all $n\ge n_2$, $\lambda_0 m/n <1$.
Then for all $n\ge n_2$
$$
\max\left\{-L(c) -1 + \frac{\lambda_0 m}{n}
,  \frac{\log(m\lambda_0+1)}{n}+ \varepsilon -L(c) \right\} = \frac{\log(m\lambda_0+1)}{n}+ \varepsilon -L(c).$$
Thus for all $n\ge n_2$
$$\min_{\lambda >\gamma}\frac{1}{n}\log \mathbb{E}[\lambda (Y_n-c) + 1]^+ \le \frac{\log(m\lambda_0+1)}{n} +\varepsilon-L(c)+\frac{\log 2}{n} .$$
Since $\varepsilon>0$ is arbitrary, we have the desired complementary inequality on first sending $n\to \infty$ and then $\varepsilon\to 0$.
\end{proof}

The above theorem suggests that the change of measure that is asymptotically optimal for importance sampling Monte-Carlo for estimating $\mathbb{P}(Y_n>c)$ may be useful for Monte-Carlo estimation
of $\min_{\lambda >\alpha}\frac{1}{n}\log \mathbb{E}[\lambda (Y_n-c) + 1]^+$ as well.
Recall that the asymptotically optimal probability measure for importance sampling  for estimating $\mathbb{P}(Y_n>c)$ with $\{Y_n\}$ as in Theorem \ref{buff-ldp}, is given as
$$\nu_{\alpha^*}(dz) \doteq e^{\alpha^* z - H(\alpha^*)} \xi(dz),$$
where $\xi$ is the probability distribution of $U_1$ and $\alpha^*$ is the conjugate dual  of $c$, namely
\begin{equation}\label{eq:duallh}
	L(c) = \sup_{\alpha \in \mathbb{R}} [\alpha c - H(\alpha)] = \alpha^* c - H(\alpha^*).
\end{equation}
We will now show that this change of measure is nearly asymptotically optimal for importance sampling estimation
of $\frac{1}{n}\log \mathbb{E}[\lambda (Y_n-c) + 1]^+$ for large values of $\lambda$.
Note that by an elementary application of Jensen's inequality, if $T_n(\lambda)$ is any unbiased estimate of $\mathbb{E}[\lambda (Y_n-c) + 1]^+$, then for any $\lambda>0$,
\begin{align*}
	\liminf_{n\to \infty} \frac{1}{n} \log \mathbb{E}(T_n^2(\lambda)) &\ge 2 \liminf_{n\to \infty} \frac{1}{n} \log \mathbb{E}[\lambda (Y_n-c) + 1]^+\\
&\ge 2 \liminf_{n\to \infty} \min_{\lambda' > 0}\frac{1}{n} \log \mathbb{E}[\lambda' (Y_n-c) + 1]^+ = -2L(c).
\end{align*}
The following result shows that this lower asymptotic bound is nearly achieved when the estimator $T_n^2(\lambda)$ is constructed using the change of measure $\nu_{\alpha^*}$ and $\lambda$ is large.
The second moment of this estimator is given as
$$R_n(\lambda) \doteq  \mathbb{E}(T_n^2(\lambda)) = \mathbb{E}\left[  \left([\lambda (Y_n-c) + 1]^+\right)^2 e^{-n\alpha^* Y_n + n H(\alpha^*)}  \right].$$
\begin{theorem}
	\label{prop:asympoptbp}
	Suppose that the conditions of Theorem \ref{buff-ldp} are satisfied. Then for every $\varepsilon >0$, there exists a $\gamma>0$ such that
	$$\sup_{\lambda \ge \gamma} \limsup_{n\to \infty} \frac{1}{n} \log R_n(\lambda) \le - 2L(c) + \varepsilon .$$
\end{theorem}
\begin{proof}
	As before we assume without loss of generality that $E(U_1)=0$ and fix $c>0$.
 For any $\lambda>0$
	\begin{align}
	\frac{1}{n} \log R_n(\lambda) &= 	\frac{1}{n} \log \mathbb{E}\left[  \left([\lambda (Y_n-c) + 1]^+\right)^2 e^{-n\alpha^* Y_n + n H(\alpha^*)}  \right]\nonumber\\
	&= H(\alpha^*) + \frac{1}{n} \log \mathbb{E}\left[  \left([\lambda (Y_n-c) + 1]^+\right)^2 e^{-n\alpha^* Y_n }  \right]\nonumber\\
	&= -L(c) + \alpha^*c + \frac{1}{n} \log \mathbb{E}\left[  \left([\lambda (Y_n-c) + 1]^+\right)^2 e^{-n\alpha^* Y_n }  \right].\label{eq:1234}
	\end{align}
	Choose $m\ge 1$ such that $L(c+m) \ge L(c)+ \alpha^*c + 1$.
	Then, for $\lambda>0$,
	\begin{align*}
	\mathbb{E}\left[  \left([\lambda (Y_n-c) + 1]^+\right)^2 e^{-n\alpha^* Y_n }  \right] &= \mathbb{E}\left[  \left([\lambda (Y_n-c) + 1]^+\right)^2 e^{-n\alpha^* Y_n }1_{\{Y_n \ge c-1/\lambda\}}  \right]\\
	&= \mathbb{E}\left[  \left([\lambda (Y_n-c) + 1]^+\right)^2 e^{-n\alpha^* Y_n }1_{\{ c-1/\lambda \le Y_n \le c+m\}}  \right]\\
	&\quad + \mathbb{E}\left[  \left([\lambda (Y_n-c) + 1]^+\right)^2 e^{-n\alpha^* Y_n }1_{\{ Y_n > c+m\}}  \right].
	\end{align*}
	For the second term on the right side we have with $\gamma_n$ as in Theorem \ref{buff-ldp},
	\begin{align*}
		\mathbb{E}\left[  \left([\lambda (Y_n-c) + 1]^+\right)^2 e^{-n\alpha^* Y_n }1_{\{ Y_n > c+m\}}  \right] &\le
		4 \mathbb{E}\left[\left(1 + \frac{(\lambda (Y_n-c))^2}{2}\right)1_{\{ Y_n > c+m\}}\right]\\
		&\le 4 \mathbb{E}\left(e^{\lambda (Y_n-c)}e^{n\gamma_n(Y_n-c-m)}\right)
		\end{align*}
		where the first inequality is a consequence of the inequality $(1+x)^2 \le 4(1+ \frac{x^2}{2})$ and the observation that
		$\alpha^*\ge 0$.	
	
		
		Therefore from \eqref{eq:eq417}, for all $n \ge n(\lambda)$, where $n(\lambda)$ is as in Theorem \ref{buff-ldp},
		\begin{align*}
			 \frac{1}{n} \log \mathbb{E}\left[  \left([\lambda (Y_n-c) + 1]^+\right)^2 e^{-n\alpha^* Y_n }1_{\{ Y_n > c\}}  \right]
		&\le -L(c+m) + \frac{\lambda m}{n} + \frac{\log 4}{n}\\
		&\le -L(c) - \alpha^*c -1 + \frac{\lambda m}{n} + \frac{\log 4}{n}.
	\end{align*}
		Next,
		\begin{align*}
		&\frac{1}{n} \log\mathbb{E}\left[  \left([\lambda (Y_n-c) + 1]^+\right)^2 e^{-n\alpha^* Y_n }1_{\{ c-1/\lambda \le Y_n \le c+m\}}  \right]\\
		&\le -\alpha^*(c - \frac{1}{\lambda}) + \frac{1}{n} \log \mathbb{P} (Y_n > c-1/\lambda) + \frac{2\log(1+m\lambda)}{n}.
	\end{align*}
		Therefore, for all $n \ge n(\lambda)$
		\begin{align*}
			&\frac{1}{n} \log \mathbb{E}\left[  \left([\lambda (Y_n-c) + 1]^+\right)^2 e^{-n\alpha^* Y_n }  \right]\\
		&\quad \le \frac{\log 2}{n} + \max \Big\{
	-L(c) - \alpha^*c -1 + \frac{\lambda m+\log 4}{n},\\
	&\quad\quad\quad -\alpha^*(c - \frac{1}{\lambda}) + \frac{1}{n} \log \mathbb{P} (Y_n > c-1/\lambda)+\frac{2\log(1+m\lambda)}{n}\Big\}.\end{align*}
		
			Fix $\varepsilon >0$ and let $0< \delta_0 \le c$ and $n_0 \in \mathbb{N}$ be such that for all $n\ge n_0$
			$$\frac{1}{n}\log \mathbb{P}\left(Y_n \ge c-\delta_0\right) \le -L(c) + \frac{\varepsilon}{2}.$$
			Then for all $n\ge n_0$ and $\delta<\delta_0$
			$$\frac{1}{n}\log \mathbb{P}\left(Y_n \ge c-\delta\right) \le \frac{1}{n}\log \mathbb{P}\left(Y_n \ge c-\delta_0\right) \le -L(c) + \frac{\varepsilon}{2}.$$
			Let $\gamma \doteq \max\{ \frac{1}{\delta_0}, \frac{2\alpha^*}{\varepsilon}\}$. Then for every $\lambda \ge \gamma$ and $n \ge \max\{n_0, n(\lambda)\}$
			\begin{align*}
				&\frac{1}{n} \log \mathbb{E}\left[  \left([\lambda (Y_n-c) + 1]^+\right)^2 e^{-n\alpha^* Y_n }  \right] \\
				&\le 	\frac{\log 2}{n} + \max \left\{-L(c) - \alpha^*c -1 + \frac{\lambda m+\log 4}{n}, -L(c)-\alpha^*c + \varepsilon+\frac{2\log(1+m\lambda)}{n}\right\}.
			\end{align*}
			Choose $n_1\ge n_0$ such that $\frac{\lambda m+\log 4}{n_1}<1$. Then for $n\ge \max\{n_1, n(\lambda)\}$ the maximum on the right side equals
			$$ -L(c)-\alpha^*c + \varepsilon+\frac{2\log(1+m\lambda)}{n}.$$
			Combining the above with \eqref{eq:1234}, for every $\lambda\ge \gamma$
		$$ 	\limsup_{n\to\infty}\frac{1}{n} \log R_n(\lambda) \le -L(c)+ \alpha^*c -L(c) -\alpha^*c + \varepsilon = -2L(c)+\varepsilon.$$
		The result follows.
\end{proof}

We now return to our main optimization problem.
Replacing the probability in \eqref{q:min_probability} with the corresponding buffered probability for the random variable $Y_n=\frac{1}{n}\sum_{i=1}^{n} G(X_i,\theta)$, and assuming $c=0 < \text{ess sup} Y_n$, we obtain the following problem:
 \begin{equation}
 \inf_{\lambda\geq 0,\theta \in \Theta} \mathbb{E}\left[ \lambda\left(\frac{1}{n}\sum_{i=1}^{n} G(X_i,\theta)-c\right)+1 \right]^+.
 \label{problem: buffered with lambda}
 \end{equation}
As discussed below Theorem \ref{thm:main2}, the above optimization problem has some appealing features.
We now present a result that makes connections between a change of measure used for solving the minimization problem   in \eqref{q:min_probability} and the minimization problem for the corresponding buffered probability, namely the problem in \eqref{problem: buffered with lambda}. For this result we recall
the definition of a subsolution of \eqref{q:isaccs_H} and the associated generalized subsolution/control, given in Section \ref{ss:change_on_X}. We will use the notation and setting of Section \ref{ss:change_on_X}
but here $m=1$ and $F(y)= \infty 1_{(-\infty, c]}(y)$. The following is the main theorem which gives the same lower bound on the exponential decay rate of the second moment of the estimator
for $\mathbb{E}[\lambda (Y_n-c) + 1]^+$ as was obtained in Theorem \ref{t:main}. Proof is given in the appendix.
\begin{theorem}\label{thm:main2}
	Let $c>0$. Assume that $H(a,\alpha)<\infty$ for all $(a,\alpha)\in \mathbb{R}^{n+1}$, and that $(\bar W, \{\rho_k, \bar a_k\}_{k=1}^K)$ is a generalized subsolution/control to \eqref{q:isaccs_H} with
	$\bar W(y,1) <0$ for all $y\ge c$. Let $\{\bar{X}^n_j\}_{1\le j\le n}$ and $\{\bar{Y}^n_j\}_{0\le j\le n}$ be as defined above Theorem \ref{t:main}.
For $\lambda>0$, define
	$Z^n(\lambda) \doteq [\lambda (\bar{Y}^n_n-c) + 1]^+ \bar\Upsilon^n$, where
	$$\bar\Upsilon^n\doteq \prod_{j=0}^{n-1}\left[\sum_{k=1}^{K}\rho_k(\bar{Y}^n_j,j/n)e^{\langle \bar{a}_k(\bar{Y}^n_j,j/n),\bar{X}^n_{j+1} \rangle - H_1(\bar{a}_k(\bar{Y}^n_j,j/n))}\right]^{-1}.$$
	Then $Z^n(\lambda)$ is an unbiased estimator of $\mathbb{E}[\lambda (Y_n-c) + 1]^+$ and there exists a $\gamma>0$ such that
	$$\sup_{\lambda \ge \gamma} \limsup_{n\to \infty} \frac{1}{n} \log \mathbb{E}[Z^n(\lambda)]^2 \le - \bar W(0,0).$$
\end{theorem}

Suppose $c=0 < \text{ess sup} Y_n$ and suppose further that $G(x,\theta)$ can be decomposed as
 \[G(x,\theta) = G_1(x,\theta) + G_2(x),
  \]
  where $G_1$ is positively homogeneous, i.e., $G_1(\lambda x, \lambda\theta) = \lambda G_1(x,\theta)$ for $\lambda \geq 0$. Then (\ref{problem: buffered with lambda}) can be rewritten as
  \begin{align}
 &\inf_{\lambda\geq 0,\theta \in \Theta} \mathbb{E}\left[ \frac{\lambda}{n}\sum_{i=1}^{n} G_1(X_i,\theta)+ \frac{\lambda}{n}\sum_{i=1}^{n} G_2(X_i)+1 \right]^+ \nonumber \\
 &\quad= \inf_{\lambda\geq 0,\bar{\theta}\in \lambda\Theta}  \mathbb{E}\left[ \frac{1}{n}\sum_{i=1}^{n} G_1(\lambda X_i,\bar{\theta})+ \frac{\lambda}{n}\sum_{i=1}^{n} G_2(X_i)+1 \right]^+.
 \label{problem: convex buffered}
 \end{align}
If $\Theta$ is a convex set and $G_1$ is convex in $(x,\theta)$, the above minimization is a convex problem  with variables $\lambda$ and $\bar{\theta}$. The above problem is convex and can be solved with well studied methods such as the gradient descent method. We will exploit this convexity property in Section \ref{s:computation} where we study some numerical examples.
%
%

\section{Computational experiments}\label{s:computation}

 In the numerical experiments we consider, the problems of interest are of the form \eqref{q:min_probability} with $A=\mathbb{R}^m_+$ and $\Theta$  a compact, convex set. We approximate the problem by \eqref{q:appro_phi}, in which $\varphi: \mathbb{R}^m \rightarrow \mathbb{R}$ is defined as
 \begin{equation}\label{q:def_phi}
 \varphi(y) = \Lambda \min(\| \min(y,0) \|^2_2, \varepsilon^2), \ y\in \mathbb{R}^m
 \end{equation}
 with $\varepsilon>0$ and $\Lambda>0$ being fixed parameters. Here $\min(y,0)$ stands for the $m$ dimensional vector whose $i$th component equals $\min(y_i,0)$. The function $\varphi$  is a bounded, Lipschitz continuous (and hence a.e. differentiable) function. It can be written as the pointwise minimum $\varphi_1\wedge\varphi_2$ of the constant function $\varphi_1(y) \equiv \Lambda\varepsilon^2$ and the continuously differentiable convex function
 $  \varphi_2(y) = \Lambda\| \min(y,0) \|^2_2$.

As noted below \eqref{q:def_g}, the problem \eqref{q:appro_phi} is equivalent to
\begin{equation}\label{q:max_gn}
\max_{\theta\in \Theta} g^n(\theta),
\end{equation}
 where $g^n$ is defined in \eqref{q:def_gn}. The latter  converges to
\begin{equation}\label{q:max_g}
 \max_{\theta\in \Theta} g(\theta)
\end{equation}
as $n\to \infty$, as shown in Corollary \ref{cor_argmin convergence of g^n}. In view of this convergence, before solving \eqref{q:max_gn} we solve the limiting problem \eqref{q:max_g} in order to find an initial point for solving \eqref{q:max_gn}. This limiting problem is discussed in Section \ref{ss:limiting}. We then apply a gradient ascent method to \eqref{q:max_gn}, in which we make use of the importance sampling techniques from Section \ref{ss:change_on_X} to estimate the objective function and its gradient. Section \ref{ss:importance sampling} provides details on implementing importance sampling techniques in the algorithm. In Section \ref{ss:comp_buffered}, the function $G$ is from $\mathbb{R}^h\times\Theta$ to $\mathbb{R}$ (i.e., $m=1$) and has a special form such that the minimization of the corresponding buffered probability can be written in the form of \eqref{problem: convex buffered}. For this specific function $G$, we solve both the buffered probability problem \eqref{problem: convex buffered} and the optimization problem \eqref{q:max_gn} (equivalently \eqref{q:appro_phi}). Section \ref{ss:numerical} summarizes the results of our numerical study.

 \subsection{Reformulation and solution of the limiting problem}\label{ss:limiting}

To solve \eqref{q:max_g}, we reformulate it as a constrained optimization problem. As before we assume that $H_2^{\theta}(\alpha)<\infty$ for all $\theta \in \Theta$ and all $\alpha \in \mathbb{R}^m$.
Recall the representation of $g$ in \eqref{q:g_maxminrepresentation}
and note that $L^\theta_2(\beta)\ge0$ for all $\beta \in \mathbb{R}^m$ and $\theta \in \Theta$. Suppose
\[
\sup_{\theta\in \Theta} \inf_{\beta \ge 0} L^\theta_2 (\beta) < \infty.
\]
Then, by choosing the parameters $\Lambda$ and $\epsilon$ in the definition of $\varphi$ in \eqref{q:def_phi} to satisfy $\Lambda\epsilon^2 \ge \sup_{\theta\in \Theta} \inf_{\beta \ge 0} L^\theta_2 (\beta) $,  for each $\theta\in \Theta$ and $\beta\in \mathbb{R}^m$
we have
\begin{equation}\label{q:phi1_phi2}
\varphi_1(\beta) + L^\theta_2(\beta) \ge \Lambda \epsilon^2  \ge \inf_{\beta' \ge 0} L^\theta_2(\beta') = \inf_{\beta' \ge 0} (L^\theta_2(\beta')+\varphi_2(\beta')),
\end{equation}
where the first inequality holds because $\varphi_1\equiv \Lambda\epsilon^2$ and $L^\theta_2(\beta) \ge 0$, and the last equality holds because $\varphi_2(\beta)=0$ for $\beta\ge 0$. Consequently, for any $\theta \in \Theta$ we have  from \eqref{q:g_maxminrepresentation} and \eqref{q:phi1_phi2}
\begin{equation}\label{q:g_theta}
\begin{split}
g(\theta)=
\inf_{\beta \in \mathbb{R}^m} (\varphi(\beta)+ L^\theta_2(\beta)) = \inf_{\beta \in \mathbb{R}^m} (\varphi_2(\beta)+ L^\theta_2(\beta))\\
=  \inf_{\beta \in \mathbb{R}^m} \sup_{\alpha \in \mathbb{R}^m} \left[ \varphi_2(\beta) +  \langle \alpha,\beta \rangle - \log \mathbb{E}e^{\langle \alpha,G(X_1,\theta) \rangle } \right].
\end{split}
\end{equation}

For each $\theta \in \Theta$ define a function $\Phi^\theta:\mathbb{R}^m\times \mathbb{R}^m \to \mathbb{R}$ as
\begin{equation}\label{q:def_Phi}
\Phi^\theta(\alpha,\beta) = \varphi_2(\beta) +  \langle \alpha,\beta \rangle -\log \mathbb{E}e^{\langle \alpha,G(X_1,\theta) \rangle }.
\end{equation}
It is clear that $\Phi^\theta$ is a continuous function and is convex with respect to $\beta$ and concave with respect to $\alpha$. The following proposition gives the existence of saddlepoints of $\Phi^\theta$. We use $S_\theta\subset \mathbb{R}^m$ to denote the support of the random variable $G(X_1,\theta)$, i.e., the smallest closed set in $\mathbb{R}^m$ such that $\mathbb{P}(G(X_1,\theta)\in S_\theta) = 1$. We then use $\text{cc} S_\theta$ to denote the closed convex hull of $S_\theta$.
Recall that we assume that,  for each $\theta\in \Theta$,  $H^\theta_2(\alpha)<\infty$  for all $\alpha\in \mathbb{R}^m$.

\begin{proposition}\label{p:saddlepoint}
 Suppose that $\text{cc} S_\theta$ has a nonempty interior. Then for each $\theta\in \Theta$ the set of saddle points of $\Phi^\theta$ is nonempty and compact.
\end{proposition}

\begin{proof}
Fix $\theta \in \Theta$. By \cite[Proposition 5.5.7]{bertsekas2009convex}, it suffices to show that for some $\bar{\alpha}\in \mathbb{R}^m$, $\bar{\beta}\in \mathbb{R}^m$ and $\bar{\gamma}\in \mathbb{R}$, the sets
\begin{equation}\label{q:level_sets}
\{\alpha \in \mathbb{R}^m \mid \Phi^\theta(\alpha, \bar{\beta}) \ge \bar{\gamma}\} \text{ and } \{\beta \in \mathbb{R}^m \mid \Phi^\theta(\bar{\alpha}, \beta) \le \bar{\gamma}\}
\end{equation}
are nonempty and compact.

First, choose $\bar{\alpha} > 0$, and we show that the level sets of $\Phi^\theta(\bar{\alpha}, \cdot)$ (namely sets of the form  $ \{\beta \in \mathbb{R}^m \mid \Phi^\theta(\bar{\alpha}, \beta) \le \bar{\gamma}\}$
for $\bar \gamma \in \mathbb{R}$) are compact. It is not hard to check that the recession function of $\Phi^\theta(\bar{\alpha}, \cdot)$ evaluated at a direction $d \in \mathbb{R}^m$ takes the value of $\langle \bar{\alpha}, d \rangle$ for $d \ge 0$ and $\infty$ for all other $d$. The recession function is nonpositive only at $d=0$. By \cite[Propositions 1.4.5-1.4.6]{bertsekas2009convex}, all level sets of $\Phi^\theta(\bar{\alpha}, \cdot)$ are compact.

Second, choose $\bar{\beta}$ from the interior of $\text{cc} S_\theta$; then 0 belongs to the interior of $\text{cc} (S_\theta - \bar{\beta})$, where $S_\theta - \bar{\beta}$ is the support of the random variable $G(X_1, \theta) - \bar{\beta}$. As shown in Step 3 of the proof of \cite[Theorem VIII.4.3]{ellis2007entropy}, the level sets of the log-moment generating function of $G(X_1, \theta) - \bar{\beta}$ are all compact, which are exactly sets of the form $\{\alpha \in \mathbb{R}^m \mid \Phi^\theta(\alpha, \bar{\beta}) \ge \bar{\gamma}\}$.

We have so far shown that the sets \eqref{q:level_sets} are compact for all $\bar{\gamma}\in \mathbb{R}$. By choosing $\bar{\gamma}$ to be sufficiently large, these sets are also nonempty.
\end{proof}

When saddle points of $\Phi^\theta$ exist, they provide solutions to the outer minimization and  inner maximization problems of $\inf_\beta \sup_\alpha \Phi^\theta(\alpha,\beta)$.  When $\Phi^\theta$ is differentiable, saddle points of $\Phi^\theta$ can be further characterized by points where the partial derivatives vanish, which means for each fixed $\theta$ the solution to $\inf_{\beta\in \mathbb{R}^m} \sup_{\alpha \in \mathbb{R}^m} \Phi^{\theta}(\alpha,\beta)$ is the solution to the following equations
$$\triangledown \varphi_2(\beta) + \alpha = 0$$
$$	\beta - \triangledown_{\alpha} \log \mathbb{E}e^{\langle \alpha,G(X_1,\theta)\rangle} = 0.$$
So \eqref{q:max_g} can be written as
\begin{equation}
\begin{split}
\max_{\theta\in\Theta,\alpha\in\mathbb{R}^m,\beta\in \mathbb{R}^m}\quad & \Phi^\theta(\alpha,\beta)= \varphi_2(\beta) +  \langle \alpha,\beta \rangle -\log \mathbb{E}e^{\langle \alpha,G(X_1,\theta) \rangle }\\
\textup{s.t.}\quad & \mathrm{E}[e^{\langle\alpha,G(X_1,\theta)\rangle}]\beta = \mathrm{E}[G(X_1,\theta)e^{\langle\alpha,G(X_1,\theta)\rangle}],\\
& 2\Lambda\min(\beta,0)+\alpha=0.
\end{split}
\label{problem: limiting problem with partial derivative}
\end{equation}

With the equality constraints the above problem is nonconvex, but it has a favorable feature that evaluating the expected values in the objective function and the constraints does not necessitate the use of importance sampling. In our numerical examples, we replace the expected values by a numerical quadrature or a sample average approximation when the numerical quadrature is not available, and solve the problem with the interior point method to find a local minimum, see \cite{byrd2000trust,byrd1999interior,waltz2006interior}.

\subsection{Implementing importance sampling in the gradient method}\label{ss:importance sampling}

In the numerical examples, $X_i$ is a normal random variable and the function $G(x,\theta)$ is piecewise linear in $(x,\theta)$. Then $\exp \left\{ -n\varphi\left( \frac{1}{n}\sum_{i=1}^{n}G(X_i,\theta) \right) \right\}$ is Lipschitz continuous in $\theta$ (with a Lipschitz constant that is uniform over values of $X_i$'s), and is thus almost everywhere differentiable with respect to $\theta$ for any fixed $X_i$'s. By an application of \cite[Theorem 7.49]{shapiro2009lectures}, the gradient of $g^n$ is given as
\begin{equation}
\triangledown g^n(\theta) = \frac{\mathbb{E}\left[  \exp \left\{ -n\varphi\left( \frac{1}{n}\sum_{i=1}^{n}G(X_i,\theta) \right) \right\} \triangledown_{\theta}\left[\varphi(\frac{1}{n}\sum_{i=1}^{n}G(X_i,\theta))\right] \right]}{ \mathbb{E}\left[  \exp \left\{ -n\varphi\left( \frac{1}{n}\sum_{i=1}^{n}G(X_i,\theta) \right) \right\} \right]}.
\label{q:gradient}
\end{equation}

For a given $\theta \in \Theta$, let $\hat{\triangledown} g^n(\theta)$ be an SAA estimator for $\triangledown g^n(\theta)$. The gradient ascent update at the $l$th iteration is then given as
$$\theta^{l+1} = \Pi_{\Theta} (\theta^l + o_l \hat{\triangledown} g^n(\theta^l))$$
where $o_l$ is the step size and $\Pi_{\Theta}$ is the projection operator from $\mathbb{R}^d$ onto the set $\Theta$. The algorithm stops when the distance from $-\hat{\triangledown} g^n(\theta^l)$ to $N_{\Theta}(\theta^l)$, the normal cone to $\theta$ at $\theta^l$, is no more than a pre-specified threshold $\Delta$.

Because the denominator of \eqref{q:gradient} is in the form of \eqref{q:mainproblem_fixedtheta},  with $\varphi$ and $G(\cdot, \theta)$ playing roles of $F$ and $G(\cdot)$ respectively, we can follow the procedures in Section \ref{s:is} to estimate it using importance sampling. Although the importance sampling methods give guaranteed asymptotic performance bounds only for estimators of the denominator
in \eqref{q:gradient}, for our numerical studies we use the same change of measure to estimate the numerator as well. As discussed in Section \ref{s:is}, there are two approaches depending on whether $X_i$ or $U_i=G(X_i,\theta)$ is used for the change of measure. Below we outline the implementation for both approaches.\\

{\bf Change of measure on $X_i$.} To implement the importance sampling scheme based on a change of measure on $X_i$, we  follow the procedure outlined below Theorem \ref{t:main} to construct a generalized subsolution/control. We select $\{(\bar{W}_k,\bar{a}_k) \}_{k=1,2}$ from the family of affine subsolution/control pairs $(\bar{W}, \bar{a})$, where $\bar{W}$ is of the form \eqref{q:W_bar} and $\bar{a}$ satisfies \eqref{eq:eq932}. We  impose the requirements $\bar{W}_1(y,1)\le 2 \phi_1 (y)$ and $\bar{W}_2(y,1)\le 2 \phi_2 (y)$ for all $y\in \mathbb{R}^m$, to guarantee \eqref{q:wedge_W} holds with $\varphi$ in place of $F$. Since $\varphi_1(y)\equiv \Lambda\epsilon^2$, we simply let $\bar{W}_1(y,t) \equiv 2\Lambda\epsilon^2$; it can be verified that $\bar{a}_1=0$ satisfies \eqref{eq:eq932}. The coefficients for $\bar{W}_2$ and the corresponding $\bar{a}_2$ are determined by the following optimization problem:
\begin{align*}
\max_{\bar{a}_2,\bar{c},u} \quad &  \bar{c} -H(-\bar{a}_2,-u)  - H_1(\bar{a}_2)\\
\text{s.t.  } & u \leq 0, \qquad \bar{c} \leq 0,   \qquad \bar{c}+ \frac{u^T u}{8\Lambda} \leq 0.
\end{align*}
The constraints arise from the requirement $\bar{W}_2(y,1)\le 2 \phi_2 (y)$ for all $y$ and the objective function reflects the fact that we aim to maximize $\bar{W}_2(0,0)$ and that \eqref{q:a_bar} should be satisfied with $\bar W$ replaced with $\bar W_2$.
After finding the optimal solution of the above optimization problem, we define $\bar{W}_2$ as
\[
\bar{W}_2(y,t) = \bar{c} + \langle u, y \rangle - (1-t)\big(H(-\bar{a}_2,-u)  + H_1(\bar{a}_2)\big).
\]
It is easily checked that $(\bar W_2, \bar a_2)$ is a subsolution/control pair.
With $\{(\bar{W}_k,\bar{a}_k) \}_{k=1,2}$ obtained, we next construct a generalized subsolution/control by defining $\bar{W}^\delta$ and $\rho^\delta_k$ as in \eqref{q:W_delta} and \eqref{q:rho_delta}, and then follow the procedure given below \eqref{q:def_gen_sub_H} to obtain an unbiased sample average estimator for the denominator of \eqref{q:gradient} of the form in \eqref{q:Z^n_H}  (with $F$ replaced by $\varphi$ and
$(\bar W, \rho_k)$ by $(\bar W^{\delta}, \rho_k^{\delta})$).
For the numerator of \eqref{q:gradient}, we use the same generalized subsolution/control as above to construct the change of measure on $X_i$, so the unbiased estimator for the numerator is similar to \eqref{q:Z^n_H} except that $e^{-nF(\bar{Y}^n_n)}$ is replaced by $e^{-n\varphi(\bar{Y}^n_n)}\nabla_\theta \varphi(\bar{Y}^n_n)$.\\

{\bf Change of measure on $U_i$.} To conduct importance sampling scheme based on a change of measure on $U_i$ we follow \cite{dupuis2007subsolutions}. For $k=1,2$ we let
\[
\beta_k \in \text{argmin}_{\beta\in \mathbb{R}^m} \left[L^\theta_2(\beta)+\varphi_k(\beta)\right] \text{ and }
\alpha_k \in \text{argmax}_{\alpha\in \mathbb{R}^m} \left[\langle \alpha, \beta_k \rangle - H^\theta_2(\alpha)\right]
\]
where the functions $H^\theta_2$ and $L^\theta_2$ are defined in \eqref{q:def_H2_theta} and \eqref{q:def_L2_theta},
and then define for $k=1,2$, functions $\bar{W}_k: \mathbb{R}^m \times [0,1]\to \mathbb{R}$ as
 $$\bar{W}_k(y,t) = -2\langle \alpha_k, y \rangle + 2[\varphi_k(\beta_k) + \langle \alpha_k, \beta_k \rangle ] - 2(1-t)H^\theta_2(\alpha_k).$$
Note that since $\varphi_1$ is a constant function, $\alpha_1=0$.
We then define $\bar{W}^\delta$ and compute $\rho^\delta_k$ similarly as in the first approach, to obtain a generalized subsolution/control
 as in Definition \ref{def_gen_sub} (see \cite{dupuis2007subsolutions}). Using this generalized subsolution/control, we follow the procedure given below \eqref{q:def_gen_sub} to obtain an unbiased estimator for the denominator of \eqref{q:gradient}. Again the numerator of \eqref{q:gradient} is estimated using the same change of measure. Note that the above definitions of $\bar{W}_k$, $\alpha_k$ and $\beta_k$ imply that
 \[
 \bar{W}_1(0,0) \wedge  \bar{W}_2(0,0) = \min_{k=1,2} \left(2(\varphi_k(\beta_k)+ L^\theta_2(\beta_k)\right) = 2 \inf_{\beta \in \mathbb{R}^m}[\varphi(\beta)+L^\theta_2(\beta)] = 2 \gamma,
 \]
where $\gamma$ is as defined in \eqref{q:gamma} with $\varphi$ in place of $F$. By a similar argument as  below \eqref{q:gamma}, it follows that the estimator for the denominator constructed using $(\bar{W}^\delta, \rho^\delta_k)$
is $\delta \log 2$ - asymptotically optimal (see \eqref{eq:eq446}).

Although the importance sampling estimator (for the denominator) constructed using the change of measure on $U_i$ is nearly asymptotically optimal in theory, it is hard
to implement in practice due to the difficulty of simulating from the distribution \eqref{q:exp_change} in typical situations. In contrast, the change of measure on $X_i$ is much easier to implement. This difference will be discussed further when we present our numerical results in Section \ref{ss:numerical} below.

\subsection{Minimization of the buffered probability}\label{ss:comp_buffered}

For problems in which $m=1$, we can use the buffered probability as an alternative measure of reliability, as discussed in Section \ref{s:buffered}.
In our numerical examples with $m=1$, we consider a function $G: \mathbb{R}^d\times \Theta \to \mathbb{R}$ of the form
$$G(x,\theta)= f^T(x-\theta)^+-b^T(c-\theta),\ \theta\in\Theta \subset \mathbb{R}^d,\ x\in\mathbb{R}^d$$
where $b,c,f \in \mathbb{R}^d$ are fixed parameters and for $x \in \mathbb{R}^d$, $x^+ = (\max(x_i, 0))_{i=1}^d$. When the set $\Theta$ is convex, the minimization of the buffered probability can be formulated as the following convex optimization problem as observed in \eqref{problem: convex buffered}:
$$\min_{\lambda\geq 0,\bar{\theta}\in \lambda\Theta} \mathbb{E}\left[ \frac{1}{n}\sum_{i=1}^{n} f^T(\lambda X_i- \bar{\theta})^+-  b^T(\lambda c-\bar{\theta}) +1 \right]^+.$$
In the numerical examples, $X_i$ follows normal distribution, so at each fixed $\bar{\theta}$ the SAA approximation of the above expectation is smooth with probability one and the gradient descent method can be applied. When estimating the objective value and the gradient, we apply the importance sampling scheme discussed in Section \ref{ss:change_on_X} with $F=\infty 1_{A^c}$.

\subsection{Numerical results}\label{ss:numerical}
\subsubsection{Example 1}
\label{exa:exa1}

We use this simple example, in which $h = m = d= 1$, to compare the two importance sampling schemes discussed in Sections \ref{ss:change_on_U} and \ref{ss:change_on_X} with the ordinary Monte-Carlo simulation. We also illustrate how the solution to the limiting problem (\ref{q:max_g}) is used as an initial point for the problem (\ref{q:max_gn}).\par

The parameters of the function $\varphi$ are $\Lambda=10^5$ and $\varepsilon = 0.01$. The function $G$ is defined as
$$G(x,\theta) = (x-\theta)^+ -0.4(1.5 - \theta).$$
We let $\Theta=[0,1.5]$, $n = 100$ and $\eta$ be the standard normal distribution.
Without using any variance reduction method, a sample average approximation of $g^n(\theta)$ based on a sample of size $N$ is
\begin{equation}
-\frac{1}{n}\log\left\{ \frac{1}{N}\sum_{j=1}^{N}\exp \left\{ -n\varphi\left( \frac{1}{n}\sum_{i=1}^{n}G(x^j_i,\theta) \right) \right\}\right\},
\label{eq: numerical 1}
\end{equation}
where $\{x^j_i\}$ are independent realizations from the distribution $\eta$. \par

$N$ independent realizations of  $e^{ -n\varphi\left( \frac{1}{n}\sum_{i=1}^{n}G(X_i,\theta) \right)}$ are simulated to compute (\ref{eq: numerical 1}). To compare the performance with the two importance sampling schemes, we calculate the sample average and the sample standard deviation of these $N$ realizations. Since these values are very close to zero, we compute the natural logarithm and denote them as ``log sample mean" and ``log sample std" in Table \ref{table: no IS}. For notation simplicity, the expectation in \eqref{q:appro_phi} is denoted as $p(\theta)$.
\begin{table}[H]
	\centering
	\tiny{
		\begin{tabular}{ c| c| c c c c c c c c }
			\hline
			\multicolumn{2}{c|}{$\theta$} & 0 & 0.2 & 0.4 & 0.6 & 0.8 & 1.0 & 1.2 & 1.4 \\
			\hline
			\multirow{2}{*}{$N=5\times10^3$} & log sample mean &  -7.1308 & -35.4495 & $-\infty$ &$-\infty$ & $-\infty$ &  $-\infty$ &  -5.2430 & -0.8957\\
			& log sample std& -7.8243 & -35.4495 & $-\infty$ & $-\infty$ & $-\infty$ & $-\infty$ & -6.8907 & -4.9723\\
			&CPU time (sec)& 0.0500 & 0.0200  & 0.0200 & 0.0600 & 0.0600 & 0.0400 & 0.0500 & 0.0299\\
			\hline
			\multirow{2}{*}{$N=5\times10^5$} &  log sample mean & -7.2532& -8.9291& -10.8197& -11.3306 & -10.9251 & -8.9533 & -5.3179 & -0.9424\\
			&  log sample std &  -10.1896 & -11.0293 & -11.9710 & -12.2264 & -12.0237 & -11.0423 & -9.2264 & -7.2831\\
			& CPU time (sec)& 2.3699 & 2.8100 & 2.6100 & 2.4000 &  2.5100 & 2.3999 & 2.3899 &  2.2400\\
			\hline
		\end{tabular}
	}
	\caption{Estimation of $p(\theta)$
	using ordinary Monte-Carlo simulation in Example 1}
	\label{table: no IS}
\end{table}

Table \ref{table: no IS} summarizes the performance of the ordinary Monte-Carlo simulation for different sample sizes $N$ and different values of $\theta$. The CPU time in Table \ref{table: no IS} includes the time for sampling,  and calculating the ``log sample mean" and the ``log sample std". When $N=5\times 10^3$, some of the ``log sample mean" and the ``log sample std" are $-\infty$. This is because none of the $5\times 10^3$ realizations correspond to the occurrence of the  rare event $\frac{1}{n}\sum_{i=1}^{n}G(X_i,\theta)>0$. When the sample size is increased to $N=5\times 10^5$, we get better estimates for $p(\theta)$.\par

Next, we implement the importance sampling scheme discussed in Section \ref{ss:change_on_U}, namely the change of measure on $U_i = G(X_i,\theta)$. Let $\bar{U}_i$ denote the random vatiable corresponding to $U_i$ under the replacement measure.
\begin{table}[H]
	\centering
	\tiny{
		\begin{tabular}{ c| c| c c c c c c c c }
			\hline
			\multicolumn{2}{c|}{$\theta$} & 0 & 0.2 & 0.4 & 0.6 & 0.8 & 1.0 & 1.2 & 1.4 \\
			\hline
			\multirow{2}{*}{$N=5\times10^3$} & log sample mean & -7.2594 & -9.2654 & -10.7920 & -11.5318 & -11.0085 &  -8.9169&  -5.3440 &  -0.9630\\
			& log sample std & -10.8499 & -12.7699 & -14.2367 & -14.9364 & -14.3812 & -12.3912 & -8.9455 & -5.1265\\
			& CPU time (sec)& 24.7700 & 25.9899& 25.0000 & 26.7000 & 23.1599 & 19.5699 &  17.7700 & 15.8299\\
			\hline
		\end{tabular}
	}
	\caption{Estimation of $p(\theta)$
	with change of measure on $U_i$ in Example 1}
	\label{table: on G}
\end{table}
Under the replacement measure, about 50\% of the realizations of $\frac{1}{n}\sum_{i=1}^{n}\bar U_i$ are positive (i.e. rare events) for each fixed $\theta$. This is in contrast to the ordinary Monte-Carlo simulation where only few or none of the realizations correspond to the occurrence of the corresponding rare event as shown in Table \ref{table: no IS}. The ``log sample std" in Table \ref{table: on G} are relatively smaller compared to the ``log sample mean", which indicates that the estimates using this  scheme are more accurate. However, the computation required for this scheme is significantly more as indicated by the high CPU time. This is because the construction of a single realization of $\bar{U}_i$ under the replacement measure is computationally intensive. Since the distribution of $\bar U_i$ does not have a tractable closed form, we need to numerically solve an equation that inverts the cumulative distribution function of $\bar U_i$ at each step in order to draw a sample from its distribution. For each fixed $\theta$, the calculation of Table \ref{table: on G} involves solving $Nn$ such equations, which takes up most of the CPU time.\par

The importance sampling scheme of Section \ref{ss:change_on_X} where one applies an exponential change of measure on $X_i$ can significantly reduce the computational burden. From Table \ref{table: on X}, we see that this scheme performs significantly better than the ordinary Monte-Carlo simulation. As expected, the sample standard deviation decreases when the sample size $N$ increases to $5\times10^5$, in which case it is approximately similar to that in Table \ref{table: on G} (where $N=5\times10^3$). An indicator that this scheme is not as efficient as the one in Section \ref{ss:change_on_U} is that the proportion of rare events is significantly smaller. The proportions of rare events for each fixed $\theta$ are recorded in the last row ``prop" in Table \ref{table: on X}. These results suggest that the change of measure in Section \ref{ss:change_on_X} may not be asymptotically efficient. Nevertheless, for the values of $\theta$ between $0$ and $0.8$, the scheme in Section \ref{ss:change_on_X} improves the proportion of rare events by a few hundred times in comparison  to the ordinary Monte-Carlo simulation. Moreover, a key advantage of this scheme over that in Section \ref{ss:change_on_U} is that drawing $\bar X_i$ is much simpler than drawing $\bar U_i$. We do not need to numerically solve the equations or calculate the inverse cumulative function to get a realization of $G(\bar{X}_i,\theta)$ under the replacement measure. We only need to draw samples from the standard normal distribution and then suitably translate  and scale these values. Hence, even with a larger sample ($N= 5\times10^5$), the CPU time required by this scheme is still significantly lesser (by a factor of $10$) than that required by the importance sampling scheme with $N= 5\times10^3$ based on an exponential change of measure on $U_i$.
\begin{table}[H]
	\centering
	\tiny{
		\begin{tabular}{ c| c| c c c c c c c c }
			\hline
			\multicolumn{2}{c|}{$\theta$} & 0 & 0.2 & 0.4 & 0.6 & 0.8 & 1.0 & 1.2 & 1.4 \\
			\hline
			\multirow{2}{*}{$N=5\times10^3$} & log sample mean &  -7.2653 & -9.1440 & -10.7649 & -11.5153 & -10.9374 &  -8.3427 &  -5.1589 & -0.8984\\
			& log sample std & -9.9762 & -11.2694 & -12.0563 & -12.7407 & -11.8451 & -9.5270 & -7.3090 & -4.9920\\
			& CPU time & 0.2700 & 0.0900 & 0.1299 & 0.1100 & 0.1499 & 0.1000 & 0.1199 & 0.0999\\
			\hline
			\multirow{2}{*}{$N=5\times10^5$} & log sample mean & -7.2719& -9.2986& -10.8466& -11.6375 & -11.0927 & -9.0575 & -5.2923 & -0.9423\\
			& log sample std & -12.0156 & -13.4715 & -14.4857 & -14.7782 & -14.0001 & -12.2827 & -9.6416 & -7.3022\\
			& CPU time & 3.7100 & 5.0000 & 4.9199 & 4.3499 & 3.8299 & 3.5699 & 3.3500 & 3.3699\\
			& prop & 0.1216 & 0.0530 & 0.0195 & 0.0068 & 0.0034 & 0.0040 & 0.0198 & 0.3937\\
			\hline
		\end{tabular}
	}
	\caption{Estimation of $p(\theta)$
	with change measure on $X$ in Example 1}
	\label{table: on X}
\end{table}

Recall from Theorem \ref{t:main} that the decay rate of the scheme in Section \ref{ss:change_on_X} is between $\bar{W}(0,0)$ and $2\gamma$ where $\gamma$ is as in \eqref{q:gamma}. As shown by Table \ref{table: performance on X}, this scheme does not
achieve the upper bound $2\gamma$. As a consequence, a larger sample is needed to match the performance  in Table \ref{table: on G}.
\begin{table}[H]
	\centering
	\small{
		\begin{tabular}{ c| c c c c c c c c }
			\hline
			$\theta$ & 0 & 0.2 & 0.4 & 0.6 & 0.8 & 1.0 & 1.2 & 1.4 \\
			\hline
			$\bar{W}(0,0)$ & 0.0829 & 0.1082 & 0.1246 & 0.1278 & 0.1144 & 0.0834 & 0.0382 & 0.0002\\
			$2\gamma$ & 0.1012 & 0.1378 & 0.1664 & 0.1794 & 0.1694 & 0.1304 & 0.0633 & 0.0004\\
			\hline
		\end{tabular}
	}
	\caption{Lower and upper bounds of the decay rate for change measure on $X_i$ in Example 1}
	\label{table: performance on X}
\end{table}

From Table \ref{table: on G} and Table \ref{table: on X}, we find that the optimal value of $\theta$ for the objective function in  (\ref{q:max_gn}) is close to 0.6. With 0.6 as the initial point, an SAA solution to the limiting problem (\ref{q:max_g}) is $\theta^*=0.6229$ with an optimal value 0.0898. This solution is obtained by directly using the Matlab nonlinear programming solver $fmincon$. We then implement the gradient ascent method to \eqref{q:max_gn} with an initial point $\theta^0 = \theta^*$ and a diminishing step size $o_l=\frac{0.1}{\sqrt{l+1}}$ for fifty iterations. Figure \ref{figure: d = 1, ordinary MC, trajectory} shows nine trajectories of objective values for the the gradient ascent method where the objective values of \eqref{q:max_gn} and the corresponding gradients are estimated by the ordinary Monte-Carlo simulation (sample size $N = 2.5\times 10^6$). Figure \ref{figure: d = 1, IS on X, trajectory} shows nine trajectories where the the objective values and the gradients are estimated by the importance sampling scheme from Section \ref{ss:change_on_X}. As a result of the variance reduction, the trajectories in Figure \ref{figure: d = 1, IS on X, trajectory} are more concentrated.

	\begin{figure}[H]
	\centering
	\begin{minipage}[t]{.45\linewidth}
		\centering
		\includegraphics[width=\linewidth]{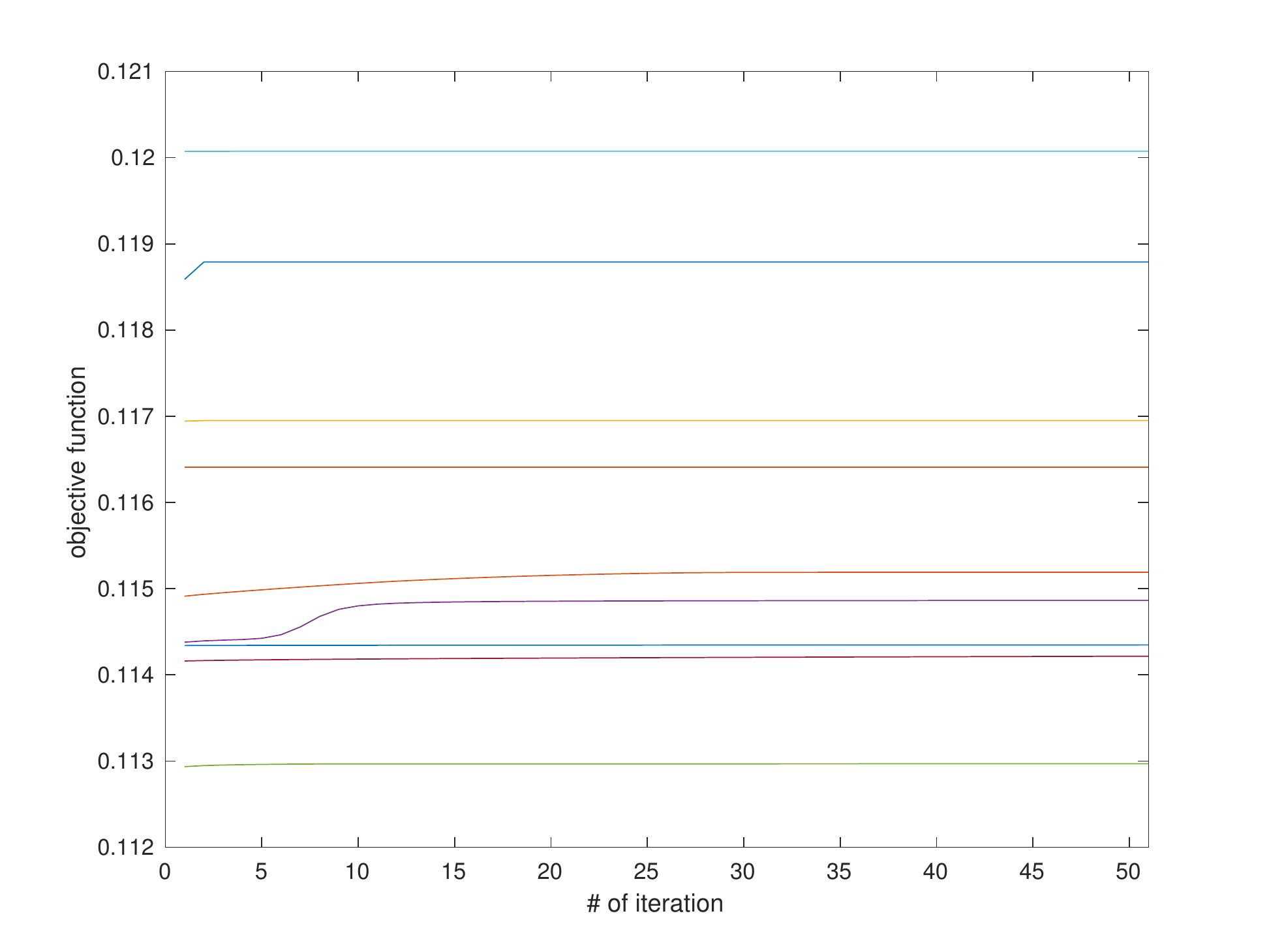}
		\caption{Trajectories of objective values of \eqref{q:max_gn} for the gradient method in Example 1 with the ordinary Monte-Carlo simulation}	
		\label{figure: d = 1, ordinary MC, trajectory}
	\end{minipage}
	\hfill
	\begin{minipage}[t]{.45\linewidth}
		\centering
		\includegraphics[width=\linewidth]{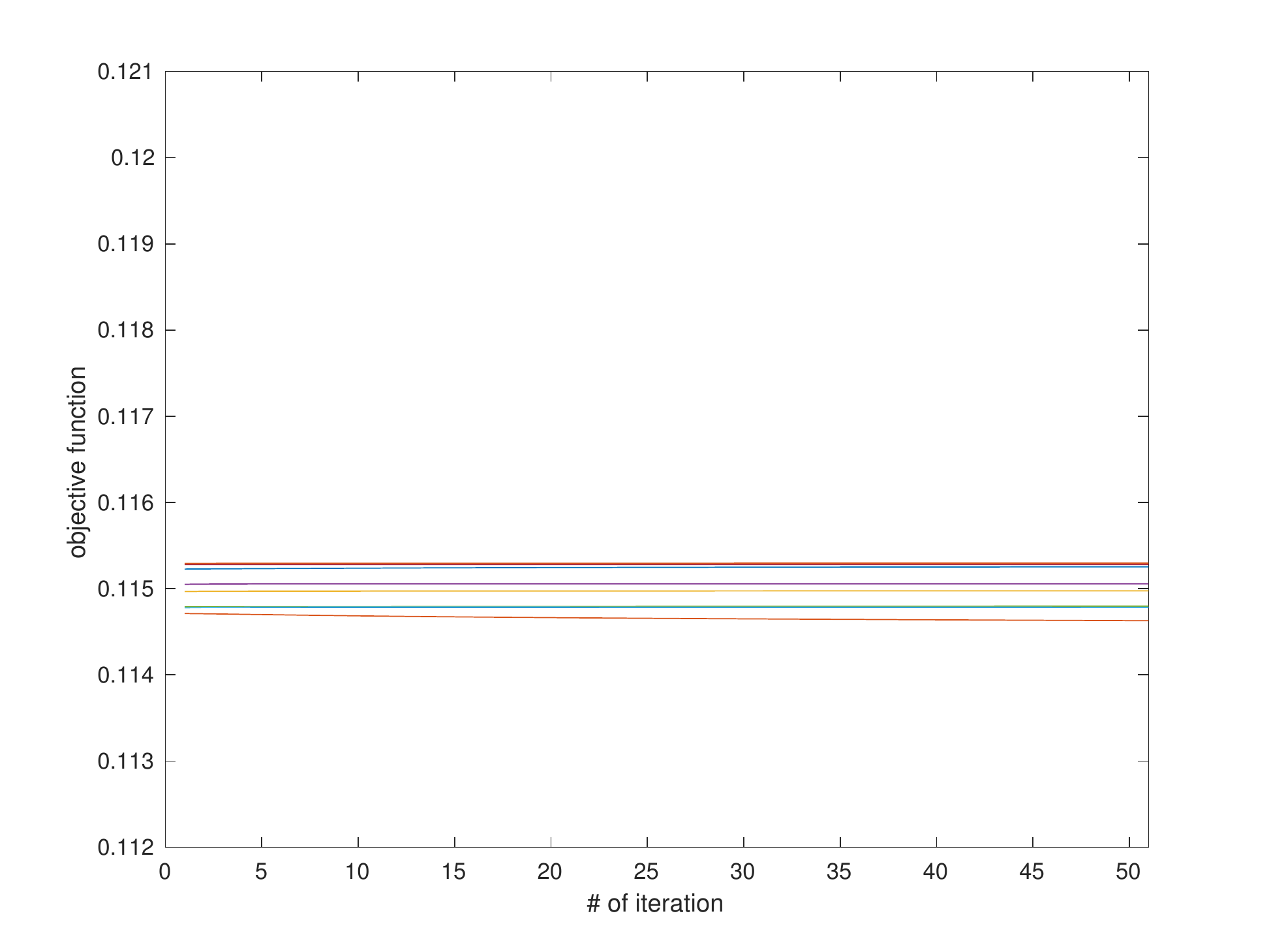}
		\caption{Trajectories of objective values of \eqref{q:max_gn} for the gradient method in Example 1 with the importance sampling scheme from Section \ref{ss:change_on_X}}
		\label{figure: d = 1, IS on X, trajectory}
	\end{minipage}
\end{figure}

Next we consider examples where implementing the importance sampling scheme from Section \ref{ss:change_on_U} will be prohibitively complicated
and therefore the scheme introduced in Section \ref{ss:change_on_X}, although suboptimal, provides a computationally feasible approach.

\subsubsection{Example 2}
\label{exa:exa2}

In this section, a 2-dim example ($h=m=d=2$) and a 5-dim example ($h=m=d=5$) will be illustrated. For these two examples, the $i$th component of the function $G$ is defined as
$$G_i(x_i,\theta_i) \doteq (x_i-\theta_i)^+ - b_i (c_i-\theta_i).$$
The other parameters are summarized in Table \ref{table: ex 2 parameters}.
\begin{table}[H]
	\centering
	\small{
		\begin{tabular}{ c c c c c c}
			\hline
			$\Lambda$ & $\varepsilon$ & $n$ & $N$ & $o_l$ & $\Delta$ \\
			\hline
			$10^5$ & 0.01& 50 & $2.5 \times 10^6$& $o_l=\frac{0.5}{\sqrt{l+1}}$& $10^{-4}$\\
			\hline
		\end{tabular}
	}
	\caption{Parameters in Example 2}
	\label{table: ex 2 parameters}
\end{table}

In the 2-dim example, the measure $\eta$, namely the distribution of $X_i$, is bivariate normal with mean 0, standard deviation 1 and covariance 0.6. The feasible set $\Theta$ is $[0,1.5]\times[0,2]$, and the parameters for the function $G$ are $b=[0.4, 0,3]^T$ and $c = [1.5,2]^T$. An SAA solution to the limiting problem is $\theta^*=[0.6415,1.1595]^T$ and the corresponding optimal value is 0.2065. As in Example 1, the limiting problem is solved by the Matlab function $fmincon$ and different initial points are considered. Starting from $\theta^*$, the gradient ascent algorithm for problem (\ref{q:max_gn}) stops after 29 iterations with $\theta^{29}=[0.6284,1.1301]^T$ and an optimal value 0.2714 (with the corresponding unnormalized value $e^{-n g^n(\theta^{29})} = 1.28\times 10^{-6}$).
Among the $2.5\times 10^6$ realizations, 0.05\% of them correspond to the occurrence of the event of interest, while the probability is of order $10^{-6}$. Figure \ref{figure: d=2 opt result} shows the objective values for each iteration and Figure \ref{figure: d=2 contour} is the contour map of the objective function $g^{50}(\theta)$ which shows that the obtained $\theta^{29}$ is close to a local minimum.
\begin{figure}[H]
	\centering
	\begin{minipage}[t]{.45\linewidth}
		\centering
		\includegraphics[width=\linewidth]{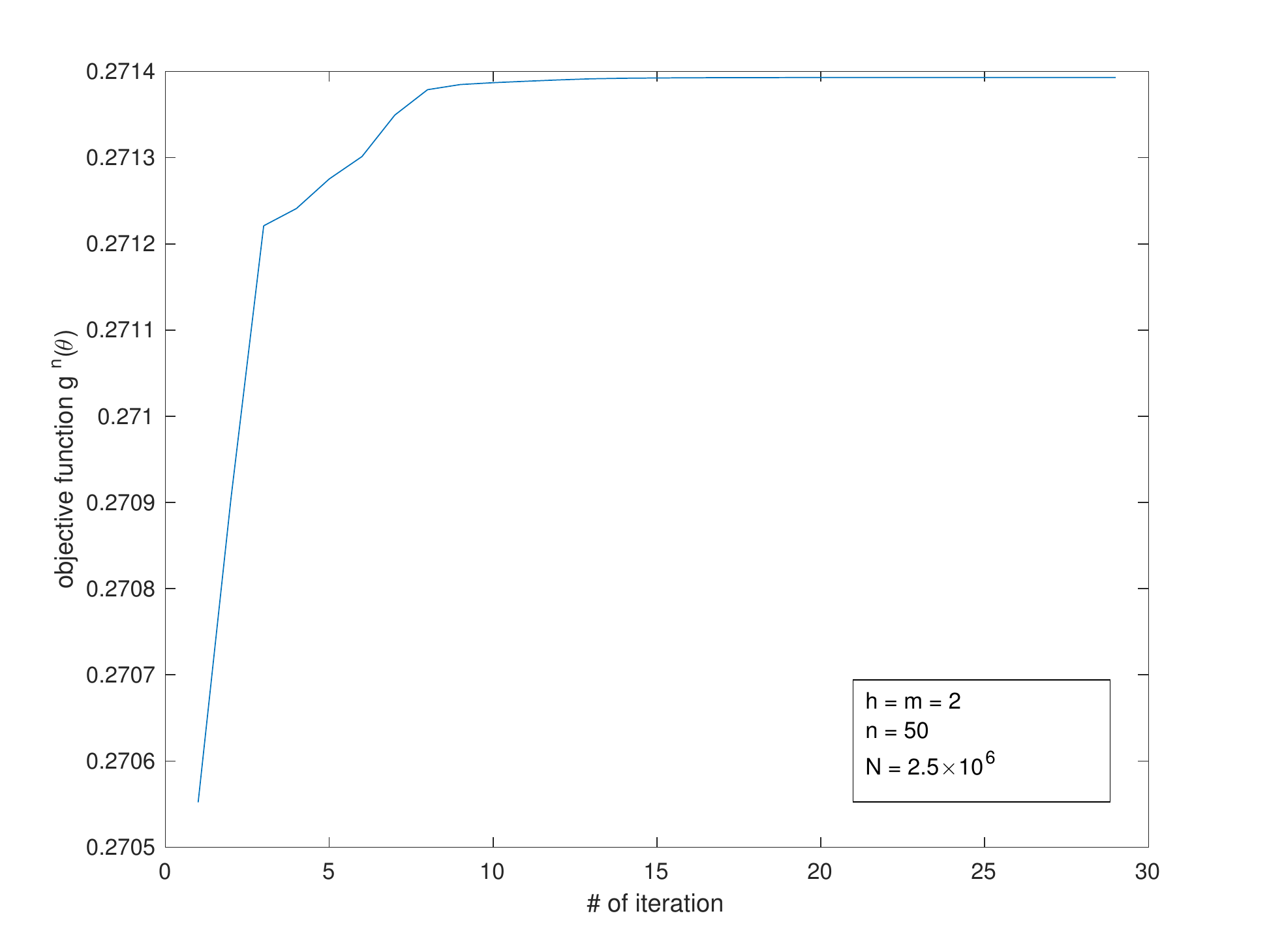}
		\caption{Objective values of \eqref{q:max_gn} for the gradient method in Example 2 (the 2-dim example)}
		\label{figure: d=2 opt result}
	\end{minipage}
	\hfill
	\begin{minipage}[t]{.45\linewidth}
		\centering
		\includegraphics[width=\linewidth]{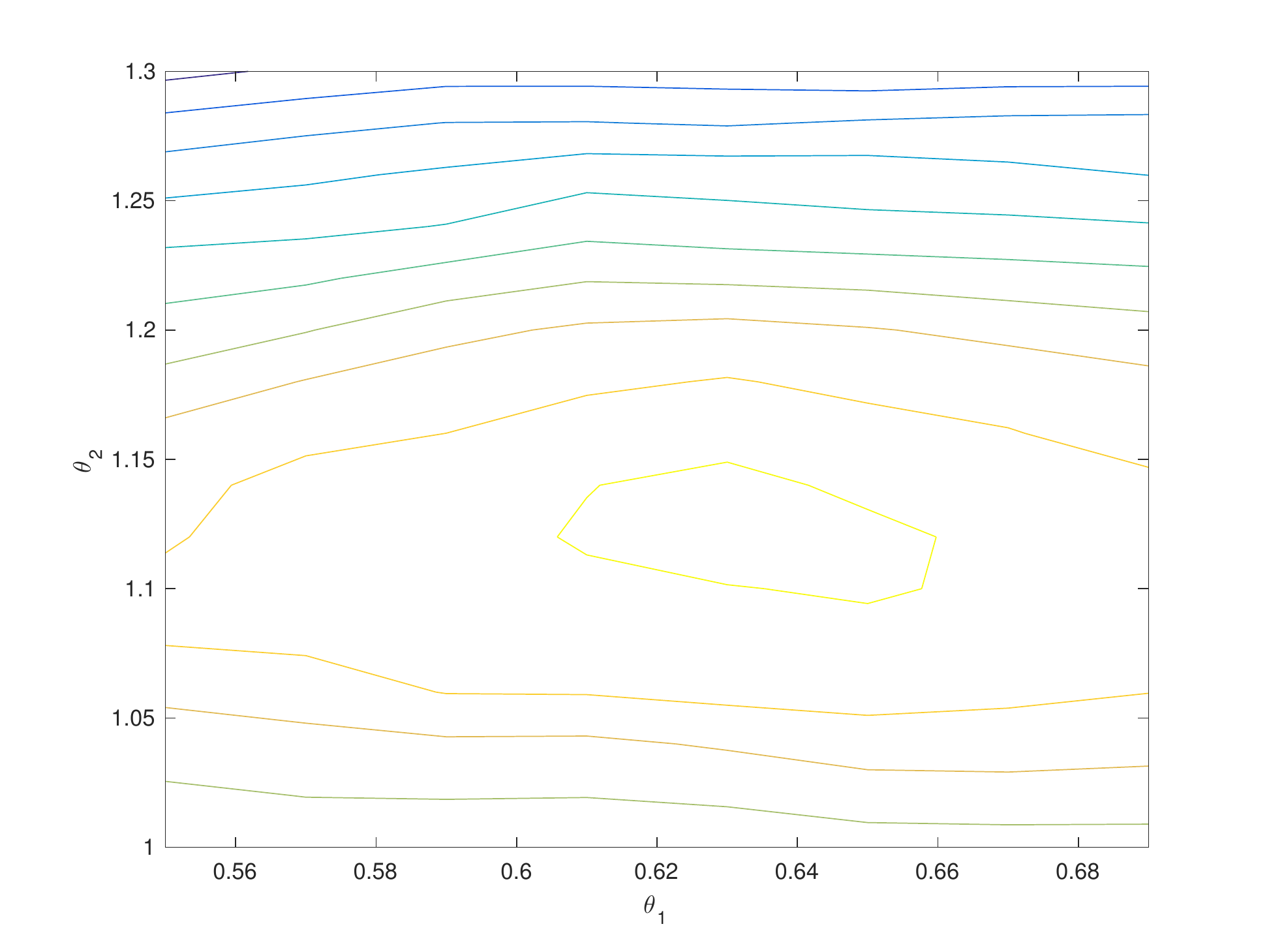}
		\caption{The contour map of $g^n(\theta)$ near $\theta^{29}$ in Example 2 (the 2-dim example)}
		\label{figure: d=2 contour}
	\end{minipage}
\end{figure}

We repeat the above procedure for the 5-dim example. The parameters of the function $G$ are $b=[0.3,0.2,0.3,0.3,0.2]^T$ and $c=[1,2,2,1,2]^T$, and the feasible set $\Theta$ is $[0,c]$. The random variables $\{X_i\}^n_{i=1}$ are i.i.d. multivariate normal with mean 0 and a randomly generated covariance matrix
$$\left[
\begin{array}{ccccc}
1 &0.3750&	0.1066&	0.7878  & -0.9006\\
0.3750 & 1 & 0.9390& 0.5709&	-0.4219\\
0.1066&	0.9390&	1&	0.2726&	-0.0910\\
0.7878&	0.5709&	0.2726&	1&	-0.9228\\
-0.9006&	-0.4219&	-0.0910&	-0.9228&	1
\end{array}
\right].$$

The SAA solution to the limiting problem is $\theta^*=[0.6270,1.6872,0.0000,0.4105,1.2983]^T$ with an optimal value 0.1143. For the problem (\ref{q:max_gn}), the stopping criterion is satisfied after 292 iterations. The optimal solution is $\theta^{292}= [0.6256,1.5272,0.5443,0.42321.2149]^T$ and the optimal value is 0.3423 (with the corresponding unnormalized value $e^{-n g^n(\theta^{292})} = 3.7\times 10^{-8}$ ). About 0.02\% of the realizations correspond to the occurrence of the event of interest while the probability close to the optimal solution is of order $10^{-8}$. Figure \ref{figure: d=5 opt result} records the objective values for each iteration.

\begin{figure}[H]
	\centering
	\includegraphics[width=0.6\textwidth]{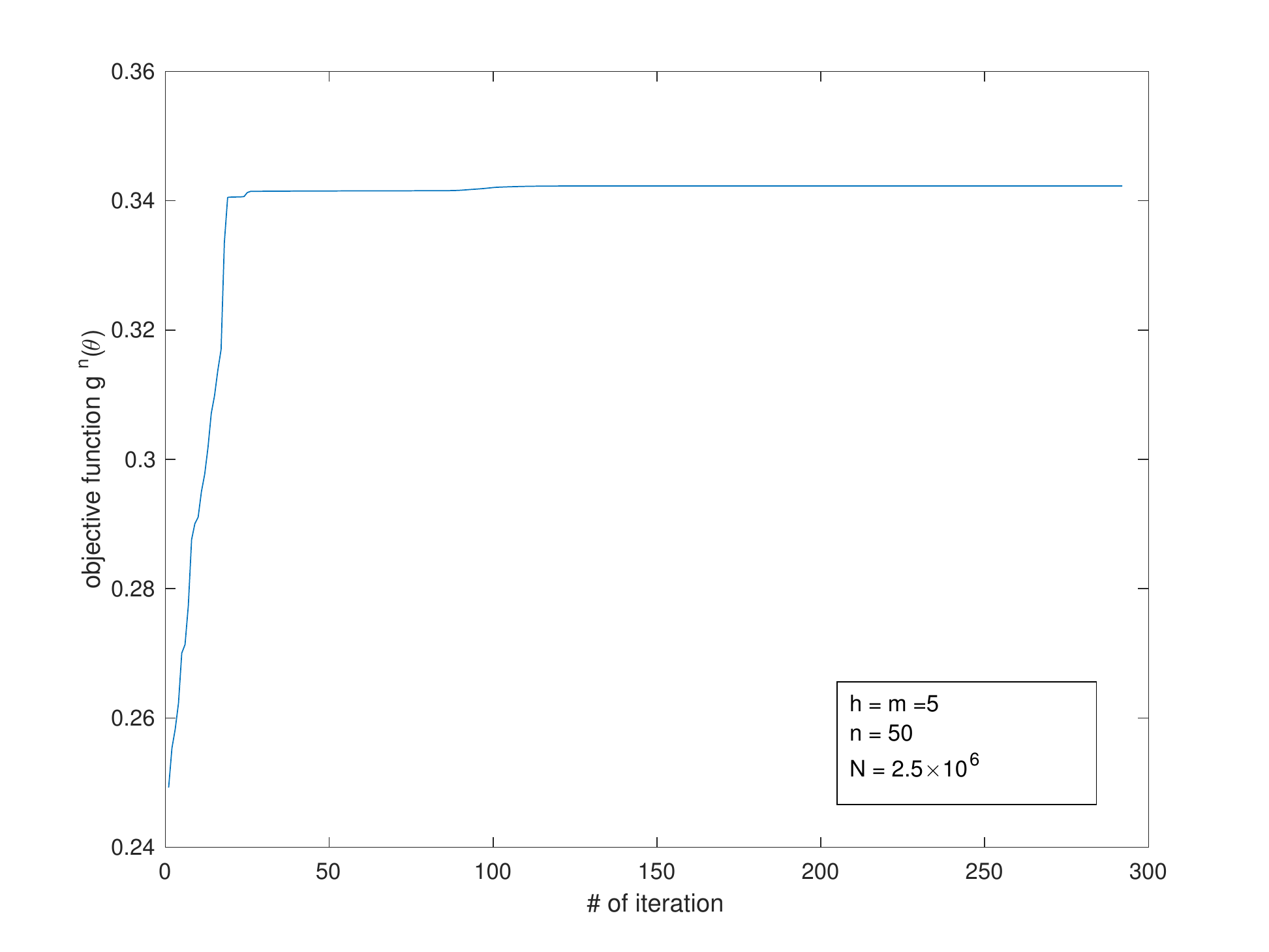}
	\caption{Objective values of \eqref{q:max_gn} for the gradient method in Example 2 (the 5-dim example)}
	\label{figure: d=5 opt result}
\end{figure}

\subsubsection{Example 3}
\label{exa:exa3}
In this example, we let $h=d=5$ and $m = 1$. The function $G$ is defined in Section \ref{ss:comp_buffered} with $b=[0.3,0.2,0.3,0.3,0.2]^T$, $c=[1,2,2,1,2]^T$ and $f=[1,1,1,1,1]^T$, which is from $\mathbb{R}^5$ to $\mathbb{R}$. For this function $G$, the buffered probability is well defined, so we numerically solve the problem (\ref{q:max_g}) and then use the importance sampling scheme from Section \ref{ss:change_on_X} to solve the optimization problems (\ref{q:max_gn}) and (\ref{problem: convex buffered}).

The distribution of $X_i$ is the same as the 5-dim example of Example 2. The SAA solution to the limiting problem is $\theta^*=[0.7863, 1.2361, 0.7860, 0.7647, 0.8842]^T$ with the optimal value 0.0894. We solve the problem (\ref{q:max_gn}) and the problem (\ref{problem: convex buffered}) at $n=50$ and 100 for each case.

For the problem (\ref{q:max_gn}) with $n=50$, we let $N=5 \times 10^5$, $o_l=\frac{0.5}{\sqrt{l+1}}$ and $\Delta= 10^{-4}$. After 1886 iterations, the stopping criterion is satisfied. The optimal solution is $\theta^{1886}= [0.7359, 1.1708, 0.7526, 0.7656, 0.8524]^T$ with the optimal value 0.1020 (with the corresponding unnormalized value $e^{-n g^n(\theta^{1886})} = 6.1\times 10^{-3}$). Approximately 2.78\% of the realizations are rare events while the probability close to the optimal solution is about 0.0061. Figure \ref{figure: d=5 Gdim1 opt result} records the objective values for each iteration.
\begin{figure}[H]
	\centering
	\includegraphics[width=0.6\textwidth]{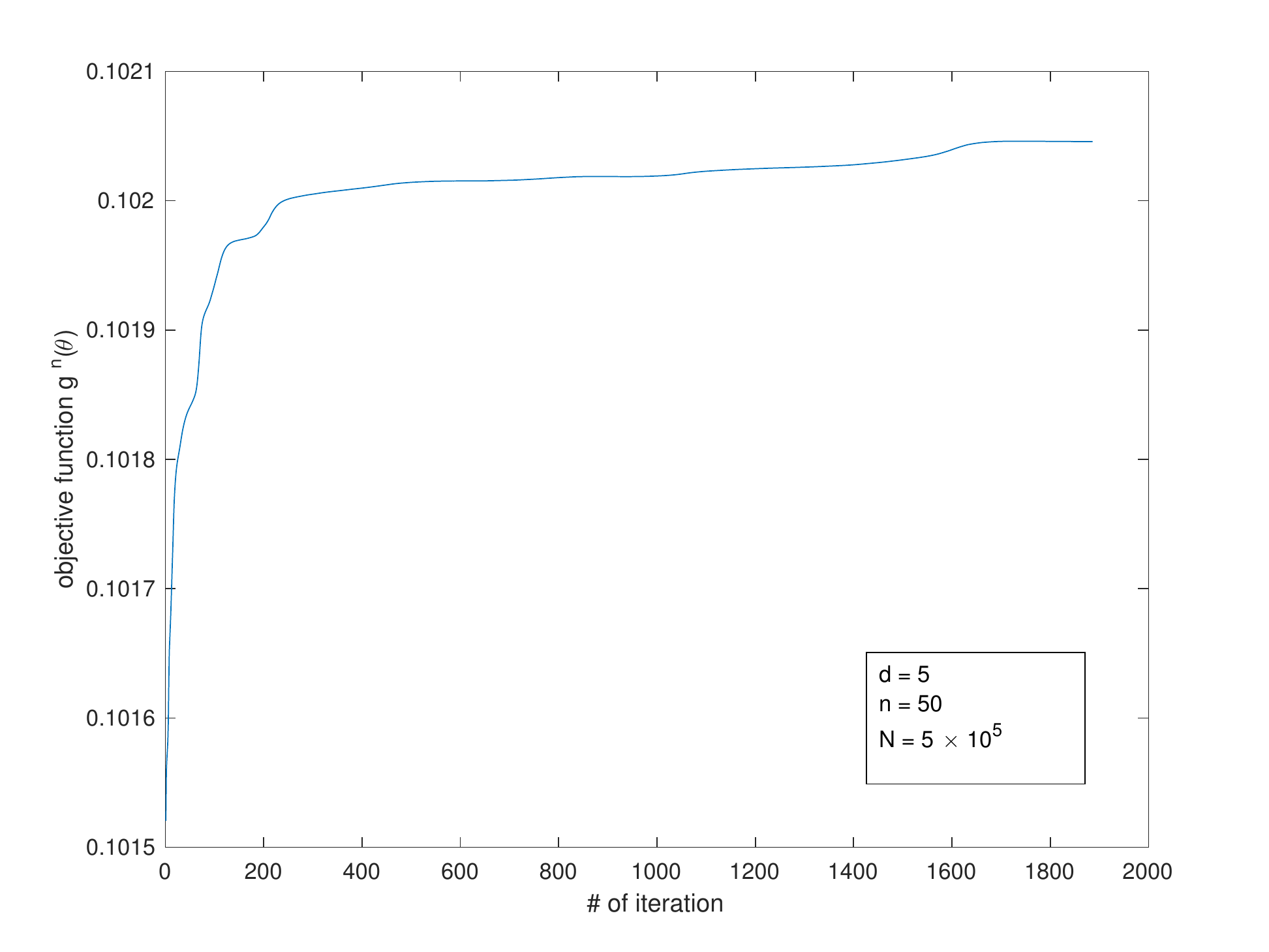}
	\caption{Objective values of \eqref{q:max_gn} for the gradient method in Example 3 ($n = 50$)}
	\label{figure: d=5 Gdim1 opt result}
\end{figure}

For the problem (\ref{problem: convex buffered}) with $n=50$, we arbitrarily select the initial point $(\theta^0,\lambda^0)=(f/2,1)$ which is the same as $(\bar{\theta}^0,\lambda^0)= (\lambda^0\theta^0,\lambda^0) = (f/2,1)$. We use a fixed length stepsize $o_l=0.1/\|\hat{\triangledown}h(\bar{\theta}^l,\lambda^l)\|_2$ (i.e., $\| \bar{\theta}^{l+1} - \bar{\theta}^l \|_2 = 0.1 $ for all $l$) to achieve a relatively large progress at each step. We also track $\theta^l$ at each step by calculating $\theta^l= \bar{\theta}^l/ \lambda^l$. The numerical solution is $\theta^{292}=[0.7314, 1.1534, 0.7312, 0.7631, 0.8369]^T$ with the optimal value $0.0159$. The estimated probability at $\theta^{292}$ is $0.0059$. The objective values at each iteration are showed in Figure \ref{figure: d=5 n=50 buff}.  Note that the objective value at iteration $l$ is not guaranteed to be the buffered probability corresponding to $\theta^l$. This is because $\lambda^l$ is not necessarily close to the optimal $\lambda$ for the minimization problem defining the buffered probability associated with $\theta^{l}$ before the algorithm terminates. The corresponding probability  and buffered probability at each iteration are calculated and shown in Figure \ref{figure: n=50 prob_vs_buffer}. The solid line shows the estimated buffered probability and the dashed line shows the estimated probability.\par

\begin{figure}[H]
	\centering
	\begin{minipage}[t]{.45\linewidth}
		\centering
		\includegraphics[width=\linewidth]{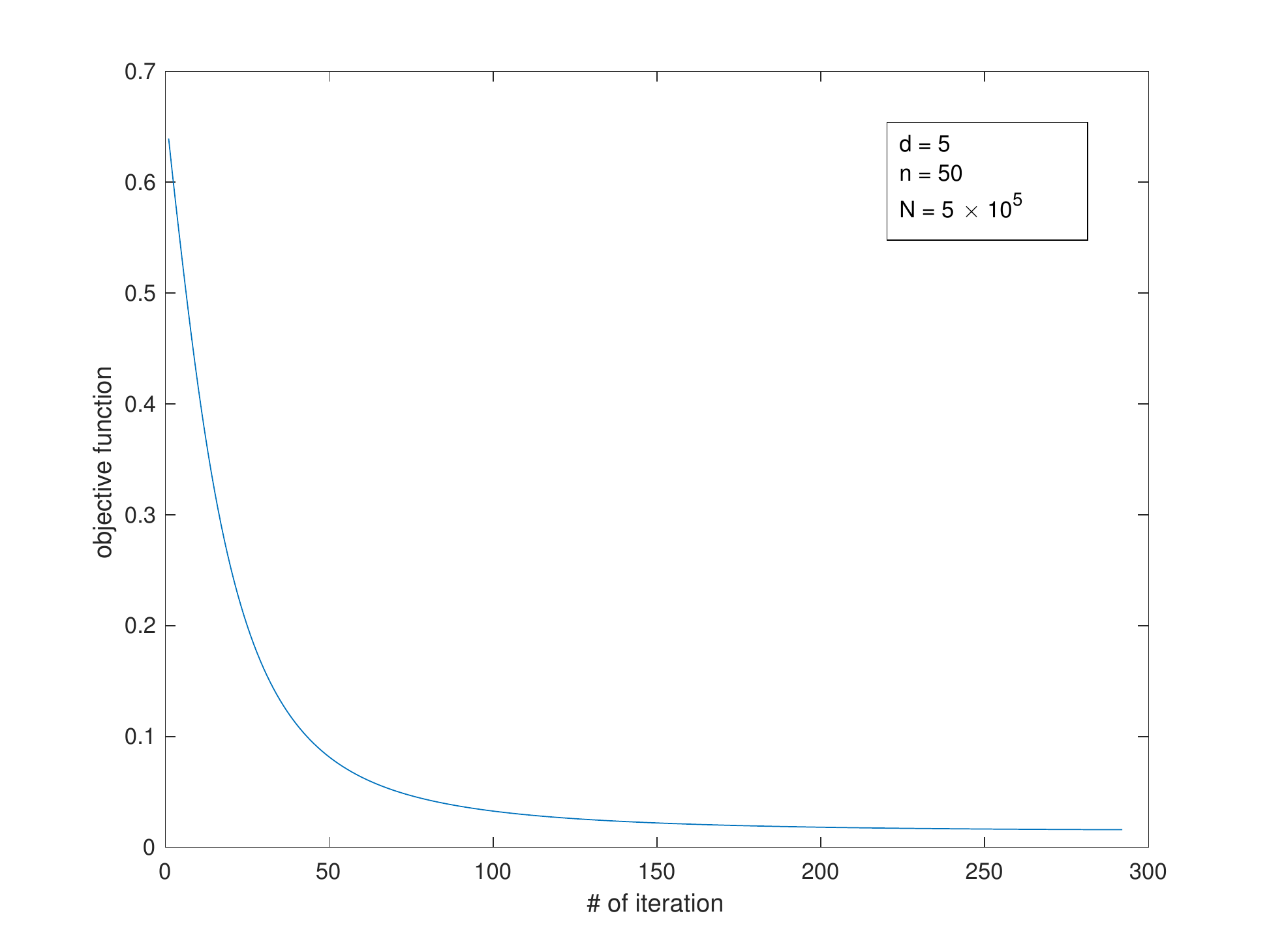}
		\caption{Objective values of \eqref{problem: convex buffered} for the gradient method in Example 3 ($n = 50$)}
		\label{figure: d=5 n=50 buff}
	\end{minipage}
    \hfill
	\begin{minipage}[t]{.45\linewidth}
		\centering
		\includegraphics[width=\linewidth]{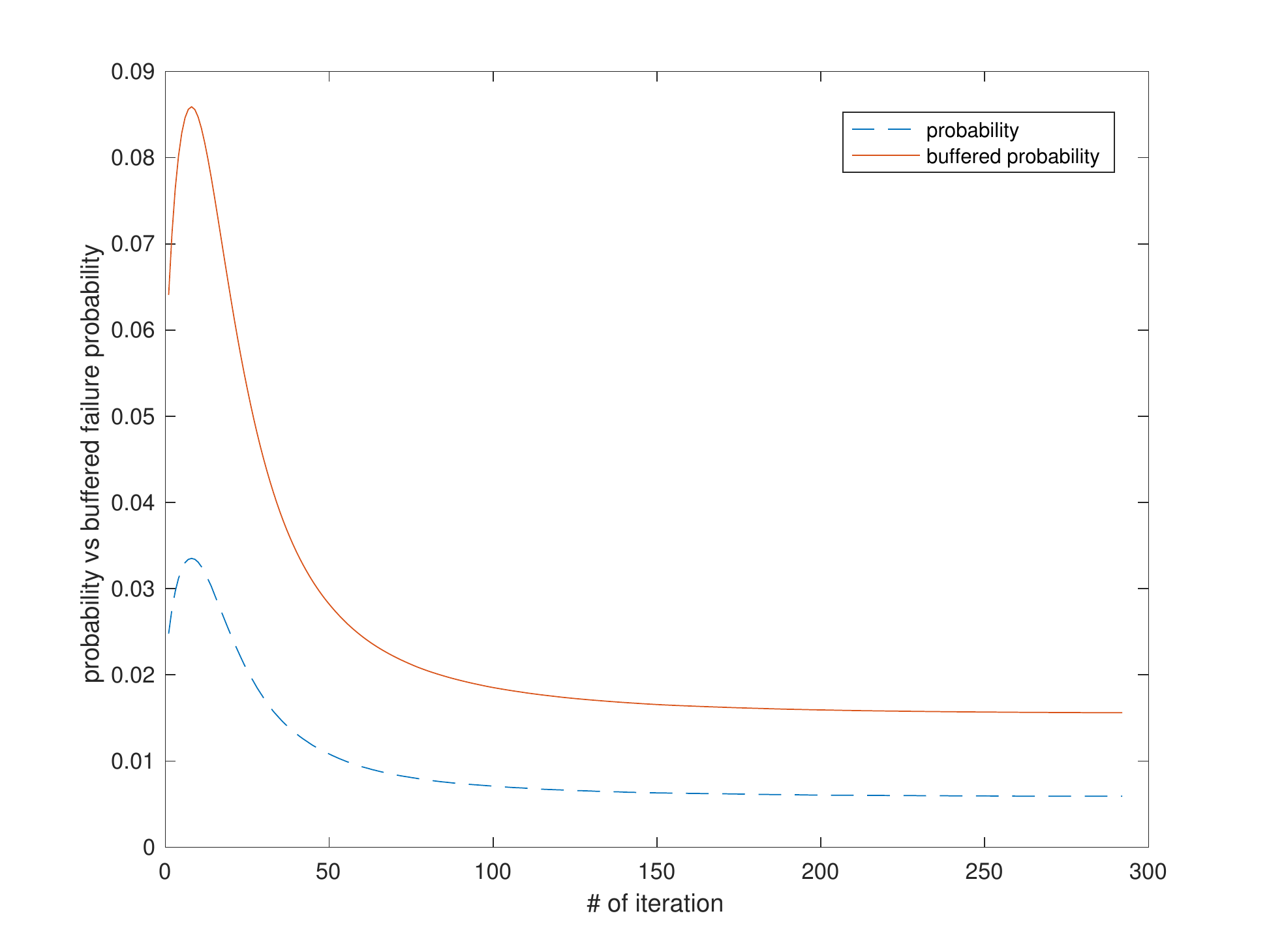}
		\caption{Probabilities and buffered probabilities corresponding to Figure \ref{figure: d=5 n=50 buff}}
		\label{figure: n=50 prob_vs_buffer}
	\end{minipage}
\end{figure}

We repeat the above calculation at $n=100$ and enlarge the sample size to $N=2.5 \times 10^6$. For the problem (\ref{q:max_gn}), the solution is $\theta^{864}= [0.7524, 1.1670, 0.7242, 0.7504, 0.8546]^T$ with the optimal value 0.0843 (with the corresponding unnormalized value $e^{-ng^n(\theta^{864})} = 2.2\times 10^{-4}$). Approximately 0.18\% of the realizations correspond to the occurrence of the event of interest. Figure \ref{figure: d=5 Gdim1 n100 opt result} records the objective values at each iteration for  the problem \eqref{q:max_gn}. For the problem (\ref{problem: convex buffered}), we use a stricter stopping criterion by setting $\Delta = 10^{-5}$. The solution is $\theta^{397}=[0.7203, 1.1344, 0.7186, 0.7505,0.8304]^T$ with the optimal value $6.9527\times 10^{-4}$. Figure \ref{figure: d=5 buff} shows the objective values at each iteration for the problem \eqref{problem: convex buffered}.
\begin{figure}[H]
	\centering
	\begin{minipage}[t]{.45\linewidth}
		\centering
		\includegraphics[width=\linewidth]{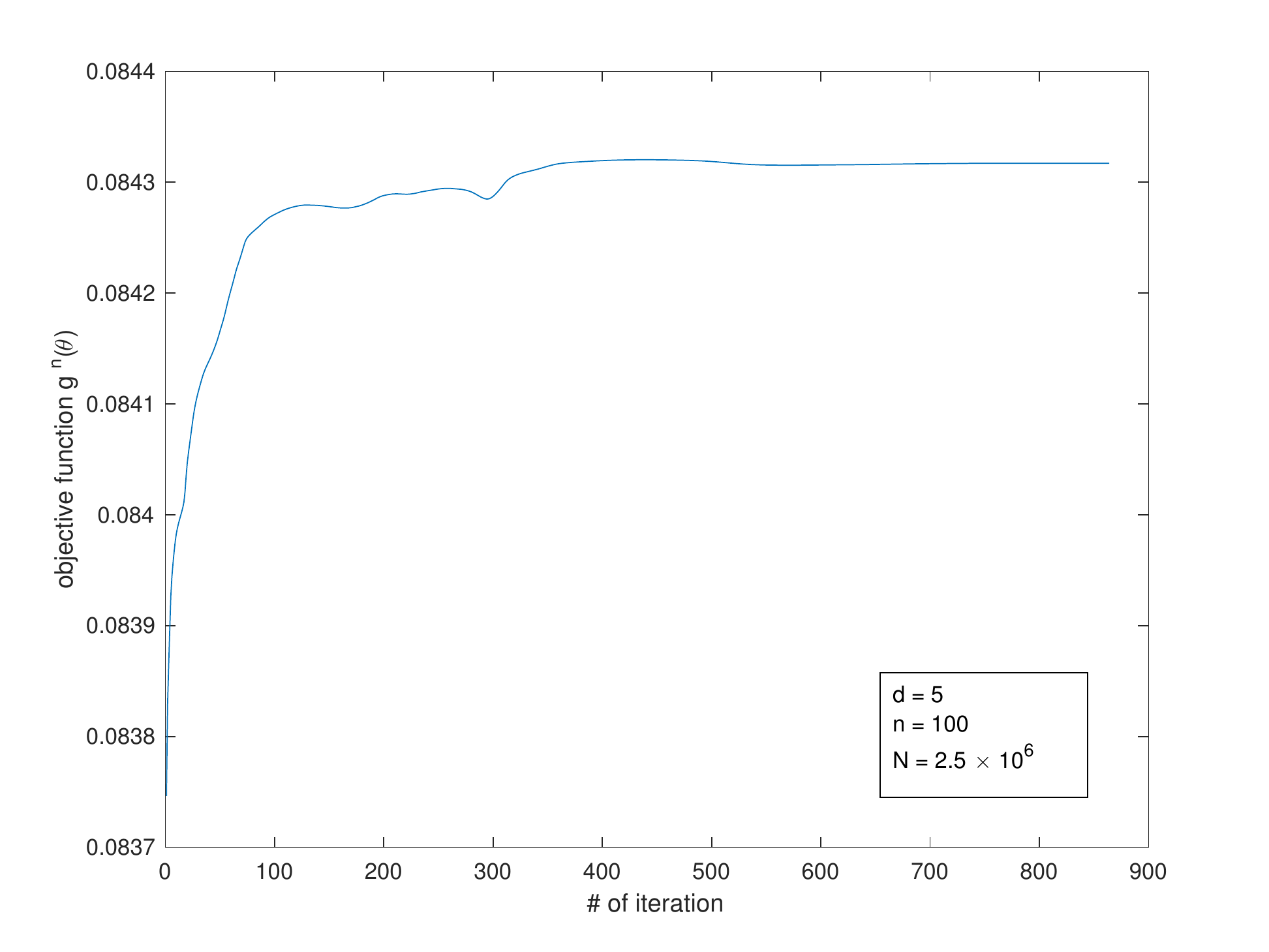}
		\caption{Objective values of \eqref{q:max_gn} for the gradient method in Example 3 ($n = 100$)}
		\label{figure: d=5 Gdim1 n100 opt result}
	\end{minipage}
    \hfill
	\begin{minipage}[t]{.45\linewidth}
		\centering
		\includegraphics[width=\linewidth]{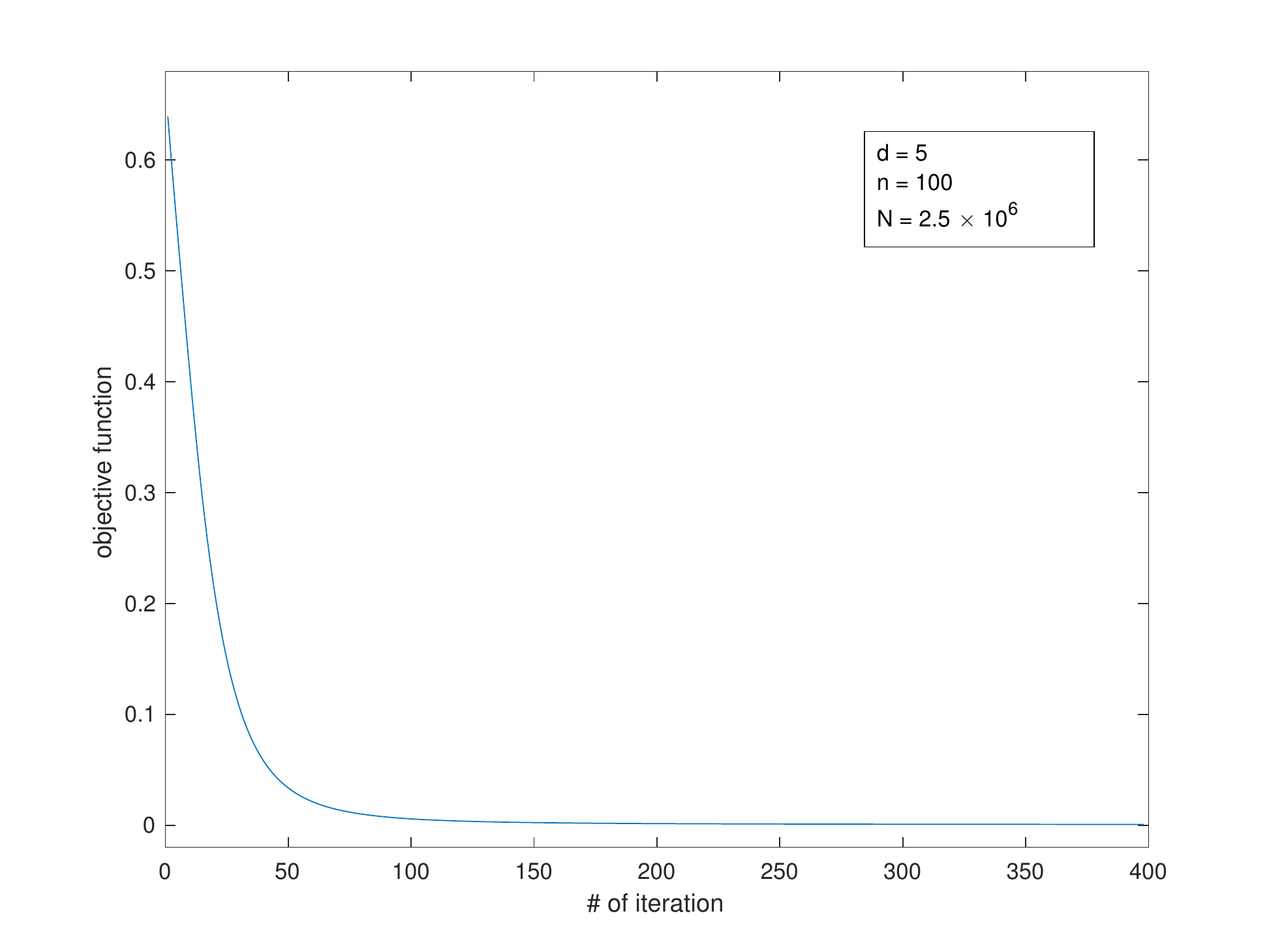}
		\caption{Objective values of \eqref{problem: convex buffered} for the gradient method in Example 3 ($n = 100$)}
		\label{figure: d=5 buff}
	\end{minipage}
\end{figure}

 In the above examples, the oscillations in the paths of objective values in Figure \ref{figure: d=5 Gdim1 opt result} and Figure \ref{figure: d=5 Gdim1 n100 opt result} are largely due to the variations in the estimation of the objective values and the gradients.

\section{Appendix}
\subsection{Proofs of Theorem \ref{t:main} and Theorem \ref{thm:main2}}
{\em Proof of Theorem \ref{t:main}}
The proof is adapted from \cite{dupuis2007subsolutions}.
For $1\in k\in K$, $j=0,\cdots, n-1$ and $y\in \mathbb{R}^m$, define $\rho_{k,j}^n(y)= \rho_{k}(y,j/n)$ and $\bar{a}^n_{k,j}(y)=\bar{a}_{k}(y,j/n)$. Using a property of Radon-Nikodym derivatives, we write the second moment of $Z^n$ in terms of the original variables $X_1,\cdots, X_n$ as
	$$ V^n = \mathbb{E}(Z^n)^2=\mathbb{E}\left[e^{-2nF(Y_n^n)}\prod_{j=0}^{n-1}\left(\sum_{k=1}^{K}\rho_{k,j}^n(Y^n_j)e^{\langle\bar{a}^n_{k,j}(Y^n_j),X_{j+1}\rangle-H_1(\bar{a}^n_{k,j}(Y^n_j))}\right)^{-1}\right],$$
where
\[
Y^n_j  = \frac{1}{n}\sum_{i=1}^j G(X_i),\  j=1, \cdots, n, \; Y^n_0=0.
\]	
Next, letting $B(y)=\bar{W}(y,1)$, we have by assumption that $B(y)\leq 2F(y)$. Define		
$$\tilde{V}^n=\mathbb{E}\left[e^{-nB(Y^n_n)}\prod_{j=0}^{n-1}\exp\left(-\sum_{k=1}^{K}\rho_{k,j}^n(Y^n_j)\bigg(\langle\bar{a}^n_{k,j}(Y^n_j),X_{j+1}\rangle-H_1(\bar{a}^n_{k,j}(Y^n_j))\bigg)\right)\right]$$
and $\tilde{W}^n= -\frac{1}{n}\log \tilde{V}^n$. 	
The fact  $B(Y^n_n)\leq 2F(Y^n_n)$ and the convexity of the exponential function imply that $V^n \le \tilde{V}^n$. Hence, to prove the theorem it suffices to  show $\liminf \tilde{W}^n \geq \bar{W}(0,0)$.\par
	
Recall from the definition of generalized solutions that $\rho_k, r_k$ and $s_k$  are uniformly bounded  which  implies the Lipschitz continuity of $\bar{W}$. By assumption,  $H_1$ is finite everywhere. Since it is convex, it is continuous and bounded on any compact set. Using these properties one can establish the following representation (see \cite[Lemma A.1]{dupuis2007subsolutions})
	\begin{align*}
	\tilde{W}^n&= \inf_{\bar{\nu}^n\in\mathscr{P}(\mathbb{R}^{nh})} \Big \{ \frac{1}{n}R(\bar{\nu}^n\|\eta^{\bigotimes n})+ {\mathbb{E}}\Big[ \frac{1}{n}\sum_{j=0}^{n-1}\sum_{k=1}^{K} \rho^n_{k,j}(\tilde{Y}^n_j)\\
	&\quad \quad \quad \quad \quad \quad\quad\Big[\langle\bar{a}^n_{k,j}(\tilde{Y}^n_j),\tilde{X}^n_{j+1}\rangle-H_1(\bar{a}^n_{k,j}(\tilde{Y}^n_j)) \Big] + B(\tilde{Y}^n_n) \Big]\Big\}
	\end{align*}
where $\eta^{\bigotimes n}$ is the $n$-fold product measure of $\eta$, $(\tilde{X}^n_1,\cdots, \tilde{X}^n_n)$ follows the distribution $\bar{\nu}^n$,  and
\[
\tilde{Y}^n_j  = \frac{1}{n}\sum_{i=1}^j G(\tilde{X}^n_i),\  j=1, \cdots, n, \; \tilde{Y}^n_0=0.
\]	
Using the chain rule for the relative entropy, we can rewrite $\tilde{W}^n$ as
\[
	\tilde{W}^n=\inf_{\bar{\nu}^n \in\mathscr{P}(\mathbb{R}^{nh})}{\mathbb{E}}\left[ \frac{1}{n}\sum_{j=0}^{n-1}\sum_{k=1}^{K} \rho^n_{k,j}(\tilde{Y}^n_j)\left[ R(\nu^n_j \|\eta)+\langle\bar{a}^n_{k,j}(\tilde{Y}^n_j),\tilde{X}^n_{j+1}\rangle-H_1(\bar{a}^n_{k,j}(\tilde{Y}^n_j )) \right] + B(\tilde{Y}^n_n) \right],
\] 	
	where $\nu^n_j$ is  the conditional distribution of $\tilde{X}^n_{j+1}$ given $(\tilde{X}^n_1, \cdots, \tilde{X}^n_j)$ (a random probability measure on $\mathbb{R}^h$). By defining
\begin{equation}\label{q:def_Jn}
	J^n(\bar{\nu}^n)=  {\mathbb{E}}\left[ \frac{1}{n}  \sum_{j=0}^{n-1}\sum^{ K}_{k=1} \rho_{k,j}^n(\tilde{Y}^n_j)\left[R(\nu^n_j\|\eta) -  H_1(\bar{a}_{k,j}^n(\tilde{Y}^n_j))+ \langle\bar{a}_{k,j}^n(\tilde{Y}^n_j),\tilde{X}^n_{j+1}\rangle\right] +B(\tilde{Y}^n_n)\right],
\end{equation}
we have  $\tilde{W}^n=\inf_{\bar{\nu}^n\in\mathscr{P}(\mathbb{R}^{nh})} J^n(\bar{\nu}^n)$. To prove the theorem,
 it suffices to prove
\begin{equation}\label{q:Jn}
\liminf J^n(\bar{\nu}^n) \geq \bar{W}(0,0)
\end{equation}
for an arbitrary sequence $\bar{\nu}^n$ of probability measures on $\mathbb{R}^{nh}$.

To prove \eqref{q:Jn}, we will use a continuous time interpolation.  To this end,
for $j=0,\dots,n-1$ and $t\in [j/n,(j+1)/n)$, define $\tilde{Y}^n(t) = \tilde{Y}^n_j$ and $\nu^n(dx|t)=\nu^n_j(dx)$, and let $\tilde{Y}^n(1)=\tilde{Y}^n_n$.
 Then define a probability measure $\nu^n$ on $\mathbb{R}^h\times [0,1]$ by  $\nu^n(A\times C)=\int_{C}\nu^n(A|t)dt$ for  $A\in \mathcal{B}(\mathbb{R}^h)$ and $C\in \mathcal{B}([0,1])$. In addition, define another probability measure $\eta'$ on  $\mathbb{R}^h\times [0,1]$ as the product measure $\eta'(dx\times dt)=\eta(dx)dt$. Note that $\nu^n$ is a random probability measure on $\mathbb{R}^h\times [0,1]$. The distribution of $\nu^n$ is determined by $\bar{\nu}^n$, a non-random probability measure on $\mathbb{R}^{nh}$. Another application of the chain rule for the relative entropy gives
	\begin{align*}
	{\mathbb{E}}R(\nu^n\|\eta') = {\mathbb{E}}\left[ \frac{1}{n}  \sum_{j=0}^{n-1} R(\nu^n_j\|\eta) \right].
	\end{align*}
We can then write $J^n(\bar{\nu}^n)$ defined in \eqref{q:def_Jn} as
\begin{align*}
	J^n(\bar{\nu}^n)= & {\mathbb{E}}\left[ R(\nu^n\|\eta') -\sum^{ K}_{k=1}\int_{0}^{1} \rho_k(\tilde{Y}^n(t),\floor{tn}/n) H_1\big(\bar{a}_k(\tilde{Y}^n(t),\floor{tn}/n)\big)dt\right.\\
	&\left.+\sum_{k=1}^{K}\int_{\mathbb{R}^h\times[0,1]}\rho_k(\tilde{Y}^n(t),\floor{tn}/n) \big\langle\bar{a}_k(\tilde{Y}^n(t),\floor{nt}/n),x\big\rangle\nu^n(dx\times dt)+B(\tilde{Y}^n(1))\right].
	\end{align*}
We define a time-continuous version of $J^n$ as
	\begin{align*}
	\bar{J}^n(\bar{\nu}^n) = & {\mathbb{E}}\left[ R(\nu^n\|\eta') -\sum^{ K}_{k=1}\int_{0}^{1} \rho_k(\tilde{Y}^n(t),t) H_1\big(\bar{a}_k(\tilde{Y}^n(t),t)\big)dt\right.\\
	&\left.+\sum_{k=1}^{K}\int_{\mathbb{R}^h\times[0,1]}\rho_k(\tilde{Y}^n(t),t) \big\langle\bar{a}_k(\tilde{Y}^n(t),t),x\big \rangle\nu^n(dx\times dt)+B(\tilde{Y}^n(1))\right].
	\end{align*}

We will show
\begin{equation}
\liminf_{n\rightarrow \infty}J^n(\bar{\nu}^n)=\liminf_{n\rightarrow \infty}\bar{J}^n(\bar{\nu}^n) \mbox{ and }
 \liminf_{n\rightarrow \infty}\bar{J}^n(\bar{\nu}^n)\geq \bar{W}(0,0). \label{eq:eq153}
\end{equation} 
The theorem is an immediate consequence of the statements in \eqref{eq:eq153}.
The proofs of these statements rely on the following two lemmas, the proofs of which are omitted since they are analogous to those of Lemma A.2 and Lemma A.3 in \cite{dupuis2007subsolutions}.


\begin{lemma}
		Assume that  $H(a,\alpha) < \infty$ for all $(a,\alpha)\in \mathbb{R}^{h+m}$, and that $(\bar{W},\rho_k,\bar{a}_k)$ is a generalized subsolution/control to \eqref{q:isaccs_H}. Consider a subsequence of $\{\bar{\nu}^n\}$ along which $J^n(\bar{\nu}^n)$ is  bounded.
		Then, relabeling this sequence as $\{n\}$,
\begin{equation}\label{q:uni_integrable_G}				
\lim_{C\rightarrow \infty} \sup_n {\mathbb{E}}\left[ \frac{1}{n} \sum_{j=1}^{n}\|G(\tilde{X}^n_j)\|1_{\{\|G(\tilde{X}^n_j)\|>C\}} \right]=0,
\end{equation}
\begin{equation}\label{q:uni_integrable_X}				
\lim_{C\rightarrow \infty} \sup_n {\mathbb{E}}\left[ \frac{1}{n} \sum_{j=1}^{n}\|\tilde{X}^n_j\|1_{\{\|\tilde{X}^n_j\|>C\}} \right]=0,
\end{equation}
		$\{(\tilde{Y}^n,\nu^n)\}$ is tight, $\{\tilde{Y}^n(1)\}$ is uniformly integrable and $\{\nu^n\}$ satisfies
\begin{equation}\label{q:nu_n_uni_integrable_G}		
\lim_{C\rightarrow \infty} \sup_n {\mathbb{E}} \left[ \int_{\mathbb{R}^h\times [0,1]} \|G(x)\|1_{\{\|G(x)\| \geq C\}}\nu^n(dx\times dt) \right]=0
\end{equation}
and
\begin{equation}\label{q:nu_n_uni_integrable}		
\lim_{C\rightarrow \infty} \sup_n {\mathbb{E}} \left[ \int_{\mathbb{R}^h\times [0,1]} \|x\|1_{\{\|x\| \geq C\}}\nu^n(dx\times dt) \right]=0.
\end{equation}
		\label{lemma A.1}
	\end{lemma}

	\begin{lemma}
		Assume that  $H(a,\alpha) < \infty$ for all $(a,\alpha)\in \mathbb{R}^{h+m}$, and that $(\bar{W},\rho_k,\bar{a}_k)$ is a generalized subsolution/control to \eqref{q:isaccs_H}. Let $\{\bar{\nu}^n\}$ be a subsequence as in Lemma \ref{lemma A.1} and
		suppose that $(\tilde{Y}^n,\nu^n)\rightarrow(\tilde{Y},\nu)$ in distribution. Then $\nu(dx\times dt)$ can be factored as $\nu(dx\times dt)=\nu(dx|t)dt$, with
\begin{equation}\label{q:tilde_Y}
\tilde{Y}(t)=\int_{[0,t]}\int_{\mathbb{R}^h}G(x)\nu(dx|s)ds, \mbox{ for all } t \in [0,1], \mbox{ a.s. }
\end{equation}
		\label{lemma A.2}
	\end{lemma}

With these two lemmas we can now complete the proof of \eqref{eq:eq153}. Without loss of generality  we can assume that $J^n(\bar{\nu}^n)$ is bounded.
The uniform boundedness and Lipschitz continuity of $\rho_k$ and $\bar{a}_k$, the continuity of $H_1$ and the uniform integrability of $\nu^n$ in \eqref{q:nu_n_uni_integrable} imply $\lim_{n\rightarrow \infty} |J^n(\bar{\nu}^n) - \bar{J}^n(\bar{\nu}^n)| = 0$. In the remainder of the proof we show $\liminf_{n\rightarrow \infty}\bar{J}^n(\bar{\nu}^n)\geq \bar{W}(0,0)$ along any such sequence.

Since $\{(\tilde{Y}^n,\nu^n)\}$ is tight along such a subsequence (Lemma \ref{lemma A.1}), by passing to a further subsequence if necessary we may assume that $(\tilde{Y}^n,\nu^n)\rightarrow(\tilde{Y},\nu)$ in distribution. Below we consider the limit of each term of $\bar{J}^n(\bar{\nu}^n)$. For its first term, note that
\begin{equation}\label{q:rela_ent}
	\liminf_{n\rightarrow \infty} {\mathbb{E}} \left[ R(\nu^n\|\eta') \right]  \geq {\mathbb{E}}\left[ \liminf_{n\rightarrow\infty} R(\nu^n\|\eta') \right] \geq {\mathbb{E}}\left[ R(\nu\|\eta') \right]
\end{equation}
where the first inequality is by Fatou's Lemma and the second follows from the lower semi-continuity of the relative entropy.
For the second term in $\bar{J}^n(\bar{\nu}^n)$, using the continuity and boundedness of $\rho_k$ and $\bar{a}_k$, and the weak convergence of $\tilde{Y}^n$ to $\tilde{Y}$, an application of the dominated convergence theorem gives
\begin{equation}\label{q:rho_kH1}
	\lim_{n\rightarrow \infty} {\mathbb{E}} \left[ \sum^{ K}_{k=1}\int_{0}^{1} \rho_k(\tilde{Y}^n(t),t) H_1(\bar{a}_k(\tilde{Y}^n(t),t))dt \right] = {\mathbb{E}} \left[ \sum^{ K}_{k=1}\int_{0}^{1} \rho_k(\tilde{Y}(t),t) H_1(\bar{a}_k(\tilde{Y}(t),t))dt \right].
\end{equation}
For the third term, 
the uniform integrability of $\nu^n$ and continuity and boundedness of  $\rho_k$ and $\bar{a}_k$ implies
\begin{equation}\label{q:rho_knun}
\begin{split}
	&\lim_{n\rightarrow \infty} {\mathbb{E}} \left[ \sum_{k=1}^{K}\int_{\mathbb{R}^h\times[0,1]}\rho_k(\tilde{Y}^n(t),t) \langle\bar{a}_k(\tilde{Y}^n(t),t),x\rangle\nu^n(dx\times dt) \right] \\
	& = {\mathbb{E}} \left[ \sum_{k=1}^{K}\int_{\mathbb{R}^h\times[0,1]}\rho_k(\tilde{Y}(t),t) \langle\bar{a}_k(\tilde{Y}(t),t),x\rangle\nu(dx\times dt) \right].
\end{split}
\end{equation}
For the last term, note that the Lipschitz continuity of $\bar{W}$ implies $B(y)=\bar{W}(y,1)$ has linear growth.
%
From the uniform integrability of $\{\tilde{Y}^n(1)\}$ in Lemma \ref{lemma A.1} we then have that
\begin{equation}\label{q:BtildeY}
\lim_{n\rightarrow\infty} {\mathbb{E}}[B(\tilde{Y}^n(1))]={\mathbb{E}}[B(\tilde{Y}(1))].
\end{equation}
Combining \eqref{q:rela_ent}, \eqref{q:rho_kH1}, \eqref{q:rho_knun} and \eqref{q:BtildeY}, we obtain the following lower bound for $\liminf_{n\rightarrow \infty} \bar{J}^n(\bar{\nu}^n)$:
	\begin{align}
	&{\mathbb{E}}\left[ R(\nu\|\eta') -   \sum^{ K}_{k=1}\int_{0}^{1} \rho_k(\tilde{Y}(t),t) H_1(\bar{a}_k(\tilde{Y}(t),t))dt \right. \nonumber \\
	& \left.+ \sum_{k=1}^{K}\int_{\mathbb{R}^h\times[0,1]}\rho_k(\tilde{Y}(t),t) \big\langle\bar{a}_k(\tilde{Y}(t),t),x\big \rangle\nu(dx\times dt) + B(\tilde{Y}(1)) \right].
	\label{lower bound}
	\end{align}

Next, using the chain rule of the relative entropy and the representation \eqref{L representation}, we have
	\begin{align*}
	R(\nu\|\eta') & =\int_{0}^{1}R(\nu(\cdot|t)\|\eta)dt \geq \int_{0}^{1} L(b(t), \beta(t))dt,
	\end{align*}
where $b(t)=\int_{\mathbb{R}^h}x\nu(dx|t)$ and $\beta(t)=\int_{\mathbb{R}^h}G(x)\nu(dx|t)$. From the definition of $b(t)$ 	$$\int_{\mathbb{R}^h\times[0,1]}\big\langle\bar{a}_k(\tilde{Y}(t),t),x\big \rangle\nu(dx\times dt)=\int_{[0,1]}\langle\bar{a}_k(\tilde{Y}(t),t),b(t)\rangle dt.$$	This gives the following lower bound for (\ref{lower bound}):
	\begin{align}
	{\mathbb{E}}\left[ \int_{0}^{1} \sum_{k=1}^{K} \rho_k(\tilde{Y}(t),t) \left[ L(b(t),\beta(t))-H_1(\bar{a}_k(\tilde{Y}(t),t))+\langle\bar{a}_k(\tilde{Y}(t),t),b(t)\rangle \right]dt +B(\tilde{Y}(1)) \right].
	\label{lower lower bound}
	\end{align}
	By the definition of generalized solutions (see \eqref{q:def_gen_sub_H}),
	\begin{align*}
	&\bar{W}_t(\tilde{Y}(t),t)+\langle D\bar{W}(\tilde{Y}(t),t),\beta(t)\rangle\\
	=&\sum^{K}_{k=1} \rho_k(\tilde{Y}(t),t)\left[ r_k(\tilde{Y}(t),t)+\big\langle s_k(\tilde{Y}(t),t),\beta(t)\big\rangle \right]\\
	\geq& -\sum^{K}_{k=1} \rho_k(\tilde{Y}(t),t)\left[ L(b(t),\beta(t))-H_1(\bar{a}_k(\tilde{Y}(t),t))+\big\langle\bar{a}_k(\tilde{Y}(t),t),b(t)\big\rangle \right].
	\end{align*}
From \eqref{q:tilde_Y} we have $\beta(t)=d\tilde{Y}(t)/dt$ for almost every $t$. Integrating over $[0,1]$ and taking expectations, we get	
	\begin{align*}
	&\bar{W}(0,0)-\mathbb{E}\bar{W}(\tilde{Y}(1),1)\\
	\leq &\mathbb{E}\left[ \int_{0}^{1} \sum^{K}_{k=1} \rho_k(\tilde{Y}(t),t)\left[ L(b(t), \beta(t))-H_1(\bar{a}_k(\tilde{Y}(t),t))+\langle\bar{a}_k(\tilde{Y}(t),t),b(t)\rangle \right]dt \right].
	\end{align*}
Since $B(\tilde{Y}(1)=\bar{W}(\tilde{Y}(1),1)$, we have shown that $\bar{W}(0,0)$ is a lower bound of (\ref{lower lower bound}) and thereby completed the proof of Theorem \ref{t:main}.
\hfill \qed \\  \ \\

{\em Proof of Theorem \ref{thm:main2}}
The unbiasedness of $Z^n(\lambda)$ is easy to check. Consider now
 $V^n(\lambda)\doteq \mathbb{E}(Z^n(\lambda))^2$. Let $m\ge 1$.
Then with
$$\Upsilon^n\doteq \prod_{j=0}^{n-1}\left[\sum_{k=1}^{K}\rho_k({Y}^n_j,j/n)e^{\langle \bar{a}_k({Y}^n_j,j/n),{X}^n_{j+1} \rangle - H_1(\bar{a}_k({Y}^n_j,j/n))}\right]^{-1},$$
we have
\begin{align}
	V^n(\lambda) &= \mathbb{E}\left(([\lambda (Y_n-c) + 1]^+)^2 \Upsilon^n\right)\nonumber\\
	&= \mathbb{E}\left(([\lambda (Y_n-c) + 1]^+)^2 \Upsilon^n 1_{\{Y_n \ge c-1/\lambda\}}\right)\nonumber\\
	&= \mathbb{E}\left(([\lambda (Y_n-c) + 1]^+)^2 \Upsilon^n 1_{\{ c-1/\lambda \le Y_n \le c+m\}}\right)
	+ \mathbb{E}\left(([\lambda (Y_n-c) + 1]^+)^2 \Upsilon^n 1_{\{ Y_n > c+m\}}\right).\label{eq:eq1207}
\end{align}
For the second term on the last line, we have by the Cauchy-Schwarz inequality
\begin{align*}
	\mathbb{E}\left(([\lambda (Y_n-c) + 1]^+)^2 \Upsilon^n 1_{\{ Y_n > c+m\}}\right) \le \left[\mathbb{E}\left(([\lambda (Y_n-c) + 1]^+)^4  1_{\{ Y_n > c+m\}}\right)\right]^{1/2}
	\left[\mathbb{E}\left(\Upsilon^n\right)^2\right]^{1/2}.
\end{align*}
By Jensen's inequality
$$0 \le \Upsilon^n \le \tilde \Upsilon^n \doteq  \prod_{j=0}^{n-1} \exp\left\{\sum_{k=1}^{K}\rho_k({Y}^n_j,j/n) (\langle \bar{a}_k({Y}^n_j,j/n),{X}^n_{j+1} \rangle - H_1(\bar{a}_k({Y}^n_j,j/n)))\right\}.$$
From this, the boundedness of $\rho_k$ and $\bar a_k$, and our assumption on the finiteness of $H$, we have that for some $c_1 <\infty$
$$	\left[\mathbb{E}\left(\Upsilon^n\right)^2\right]^{1/2} \le e^{nc_1} \mbox{ for all } n \ge 1.$$
Also, for some $c_2<\infty$
\begin{align*}
\mathbb{E}\left(([\lambda (Y_n-c) + 1]^+)^4  1_{\{ Y_n > c+m\}}\right)	&\le c_2 \mathbb{E}\left( \left(1 +\frac{(\lambda (Y_n-c))^4}{4!}\right)  1_{\{ Y_n > c+m\}}\right)\\
&\le c_2\mathbb{E}\left(e^{\lambda (Y_n-c)}e^{n\gamma_n(Y_n-c-m)}\right),
\end{align*}
where $\gamma_n$ is as introduced above \eqref{eq:eq417}.  The same calculation as in \eqref{eq:eq417} now shows that
$$\frac{1}{n}\log\left[\mathbb{E}\left(([\lambda (Y_n-c) + 1]^+)^4  1_{\{ Y_n > c+m\}}\right)\right]^{1/2} \le -\frac{L(c+m)}{2} + \frac{\lambda m + \log c_2}{2n}.$$
Thus
$$\frac{1}{n}\log\mathbb{E}\left(([\lambda (Y_n-c) + 1]^+)^2 \Upsilon^n 1_{\{ Y_n > c+m\}}\right) \le -\frac{L(c+m)}{2} + c_1 + \frac{\lambda m + \log c_2}{2n}.$$
Now fix $m\ge 1$ such that $L(c+m)/2 \ge \bar W(0,0) +1 + c_1$.

Now consider the first term on the right side of \eqref{eq:eq1207}. We have
$$
\mathbb{E}\left(([\lambda (Y_n-c) + 1]^+)^2 \Upsilon^n 1_{\{ c-1/\lambda \le Y_n \le c+m\}}\right) \le (\lambda m+1)^2  \mathbb{E}\left(\tilde \Upsilon^n 1_{\{Y_n\ge c-1/\lambda\}}\right).$$
Choose $\gamma$ large enough so that $\bar W(y,1) \le 0$ for $y \ge c-1/\gamma$. Then with $B$ as in the proof of Theorem \ref{t:main}
we have $1_{\{Y_n\ge c-1/\lambda\}} \le e^{-nB(Y_n)}$ for $\lambda \ge \gamma$. Thus we have
$$\frac{1}{n}\log\mathbb{E}\left(([\lambda (Y_n-c) + 1]^+)^2 \Upsilon^n 1_{\{ c-1/\lambda \le Y_n \le c+m\}}\right) \le
\frac{2\log(\lambda m+1)}{n} + \frac{1}{n} \log \tilde V^n$$
where $\tilde V^n$ is as in the proof of Theorem \ref{t:main}. Choose $n_1 \in \mathbb{N}$ such that $\frac{(\lambda m + \log c_2)}{2n_1}<1$.
Thus for all $\lambda \ge \gamma$ and $n \ge n_1$
$$\frac{1}{n}\log V^n(\lambda) \le \frac{\log 2}{n} + \max \left\{ \frac{2\log(\lambda m+1)}{n} + \frac{1}{n} \log \tilde V^n , - \bar W(0,0)\right\}.$$

Taking limit as $n\to \infty$, we now have from the proof of Theorem \ref{t:main} that for all $\lambda \ge \gamma$
$\limsup_{n\to \infty}\frac{1}{n}\log V^n(\lambda) \le - \bar W(0,0)$. The result follows.
\hfill \qed

{\bf Acknowledgements.} Research of AB and SL were supported in part by the National Science
Foundation (DMS- 1814894).

\bibliographystyle{amsplain}

\providecommand{\bysame}{\leavevmode\hbox to3em{\hrulefill}\thinspace}
\providecommand{\MR}{\relax\ifhmode\unskip\space\fi MR }
\providecommand{\MRhref}[2]{%
  \href{http://www.ams.org/mathscinet-getitem?mr=#1}{#2}
}
\providecommand{\href}[2]{#2}

\vspace{\baselineskip}

\textsc{\kern-1.4em A. Budhiraja (email:
budhiraj@email.unc.edu) \newline
S. Lu (email: shulu@email.unc.edu) \newline
Y. Yu (email: yy0324@live.unc.edu)\newline
Q. Tran-Dinh (email: quoctd@email.unc.edu) \newline\newline
Department of Statistics and
Operations Research\newline University of North Carolina\newline Chapel Hill,
NC 27599, USA}

\end{document}